\documentclass{article}
\usepackage{graphicx}
\usepackage{hyperref}
\usepackage{amsmath,amssymb,amsthm}
\usepackage{xcolor}
\usepackage{tikz}
\usepackage{dsfont}
\usepackage{cancel}
\usepackage{enumitem}
\usepackage{mathtools}
\usepackage{geometry}
\usepackage{authblk}

\mathtoolsset{showonlyrefs}

\newcommand{\R}{\mathbb{R}}

\newcommand{\N}{\mathbb{N}}
\newcommand{\Z}{\mathbb{Z}}
\newcommand{\C}{\mathbb{C}}

\DeclareMathOperator\supp{supp}
\DeclareMathOperator\im{Im}
\DeclareMathOperator\spn{span}
\DeclareMathOperator{\Tr}{Tr}
\DeclareMathOperator{\diag}{diag}

\newtheorem{thm}{Theorem}[section]
\newtheorem{prop}[thm]{Proposition}
\newtheorem{assump}[thm]{Assumptions}
\newtheorem{assumption}[thm]{Assumption}
\newtheorem{defi}[thm]{Definition}
\newtheorem{lemma}[thm]{Lemma}
\newtheorem{cor}[thm]{Corollary}

\usepackage{pgfplots}
\pgfplotsset{compat=1.15}
\usepackage{mathrsfs}
\usetikzlibrary{arrows}
\definecolor{ccqqqq}{rgb}{0.8,0.,0.}
\definecolor{ududff}{rgb}{0.30196078431372547,0.30196078431372547,1.}
\definecolor{xdxdff}{rgb}{0.49019607843137253,0.49019607843137253,1.}

\title{Asymptotics of Green functions for Markov-additive processes: \\
an approach via dyadic splitting of integrals
}
\author{Théo Ballu\thanks{Théo Ballu, Univ Angers, CNRS, LAREMA, SFR MATHSTIC, F-49000 Angers, France.}}
\date{}
\begin{document}

\maketitle

\begin{abstract}
    We consider a discrete Markov-additive process, that is a Markov chain on a state space $\Z^d \times E$ with invariant jumps along the $\Z^d$ component. In the case where the set $E$ is finite, we derive an asymptotic equivalent of the Green function of the process, providing a new proof of a result obtained by Dussaule in 2020. This result generalizes the famous theorem of Ney and Spitzer of 1966, that deals with the sum of independent and identically distributed random variables, to a spatially non-homogeneous case.
    
    In this new proof, we generalize the arguments used in Woess's book \cite{Wo-00} to prove Ney and Spitzer's theorem, that consists in establishing an integral formula of the Green function from which we get the asymptotic equivalent. To do so, we use techniques developed by Babillot. In particular, we use \emph{dyadic splitting of integrals}, a powerful Fourier analysis tool that enables us to control the Fourier transform of a function that has a singularity at the origin.
\end{abstract}


\tableofcontents

\section{Introduction}\label{sec:intro}
\subsection{Context}\label{subsec:context}
In their article \cite{NeSp-66} of 1966, Ney and Spitzer studied random walks, that is the sum~$S_n = X_1 + \ldots + X_n$ of independent and identically distributed random variables $\left(X_k\right)_{k \geq 1}$ on $\Z^d$, with finite range. They proved that in the non-zero drift case, its \emph{Green function}
\[G(x,y) := \sum_{n=0}^{+\infty} \mathbb{P}_x(S_n = y), \]
where $\mathbb{P}_x(S_n = y):=  \mathbb{P}(x + S_n = y)$,
decreases like $\|y\|^{- \frac{d-1}{2}} e^{-c_y \cdot y}$ as $\|y\|$ tends to infinity, where the exponential decay constant $c_y$ depends on the direction $\frac{y}{\|y\|}$. See \cite[Theorem 2.2]{NeSp-66} which we will call \emph{Ney and Spitzer's theorem} throughout this article. As a consequence of these asymptotics, they described the Martin boundary of the random walk. We will call this process \emph{spatially homogeneous} as for every $x,y \in \Z^d$, $\mathbb P\left(S_{n+1} = y ~|~S_n = x \right) = \mathbb P(X_1 = y-x)$ only depends on $y-x$.
The study of this spatially homogeneous process was the first of many works on asymptotics of Green functions of random walks and descriptions of their Martin boundaries. Staying in the homogeneous case, Woess's book \cite{Wo-00} contains another proof of Ney and Spitzer's result. 

In order to solve \emph{spatially inhomogeneous} cases, many techniques were developed, as inhomogeneities take numerous forms. 
Analytic methods based on elliptic curves were used by Kurkova and Malyshev in \cite{kourkova_malyshev} to compute the asymptotics of Green functions of homogeneous random walks on $\Z^2$ modified on a finite number of vertices, for which the Martin boundary is the same as in the homogeneous case. They also studied random walks on half-planes and quarter planes in the reflected case, with jumps to nearest neighbors. Regarding quarter planes, this analytic approach was used by Kurkova and Raschel in \cite{Kurkova_raschel_2011} to study Green functions of random walks to nearest neighbors with a drift on a quarter plane, absorbed at the axes. Using large deviations, Ignatiouk-Robert studied the Martin boundary of random walks on the half-space $\Z^d \times \N$ that are killed \cite{ignatiouk_killed} or reflected \cite{ignatiouk_reflected, ignatiouk_t_martin_reflected} at the boundary.  A combination of analytic and probabilistic methods allowed the authors of \cite{ignatiouk_kourkova_raschel} to derive the asymptotics of Green functions of random walks on quarter planes that are reflected at the boundary.
On more complicated state spaces than subsets of $\Z^d$, we mention the work of Boutillier and Raschel \cite{BR22} about random walks on isoradial graphs and the more algebraic articles \cite{biane_su2, kilian_sp4}.

In this article, we study \emph{Markov-additive} processes, that is, Markov chains on a Cartesian product $\Z^d \times E$ whose jumps are invariant by translation along the $\Z^d$ component, see \eqref{eq:prop_markov_addit} for a precise definition. These processes are sometimes called \emph{Markov modulated} processes, \emph{Markov walks} or \emph{semi-Markovian} processes. The works \cite{ignatiouk_killed, ignatiouk_reflected, ignatiouk_t_martin_reflected} we mentioned on half-spaces are examples of such processes. Indeed, they are Markov chains on $\Z^d \times\N $ whose jumps only depend on whether or not the process lies on the boundary of the half-space, and are thus invariant along the $\Z^d$ component. In \cite{MR0890364}, Guivarc'h gave transience and recurrence criteria for centered Markov-additive processes. Many of the techniques we use were developed in Babillot's article \cite{MR0978023}, dealing with asymptotics of Green functions of Markov-additive processes on continuous state spaces, which was followed by Guibourg in \cite{these_guibourg}. In the discrete setting, Dussaule generalized Ney and Spitzer's theorem in Sections 3 and 4 of \cite{Du-20}, in the case of the state space~$\Z^d \times E$ for a finite set $E$, which he called a thickened lattice. In particular, he studied the asymptotics of the Green functions of non-centered processes in his Proposition 3.27. 

In this paper, we provide a different proof of this result, which is our Theorem \ref{thm:main}. While Dussaule generalized the original proof of Ney and Spitzer from \cite{NeSp-66}, we use Babillot's techniques from \cite{MR0978023} to generalize the proof of Ney and Spitzer's theorem from Woess's book \cite[(25.15)]{Wo-00}. We mention two of these techniques, namely \emph{changes of sections}, which are modifications of the process that will help us to differentiate twice the Fourier transform of the process, see Section \ref{subsec:spectra_rad_fourier}, and a little known method of advanced Fourier analysis named \emph{dyadic splitting of integrals}, which we think could be of wider interest, see Appendix \ref{subsec:dyadic_integrals}. This method enables us to isolate the leading term in an integral formula of the Green function.  More details on the proof strategy and comments on the original contributions are given in Sections \ref{subsec:outline} and \ref{subsec:contributions}. We mention that the proof of Ney and Spitzer's theorem in Woess's book was adapted by Ignatiouk-Robert in \cite{Ig-24} to deal with killed random walks in cones.

The traditional motivation to study the Green function of a process is the description of its \emph{Martin boundary}. Martin boundaries for countable transient Markov chains were introduced by Doob in his article \cite{doob} of 1959. The books \cite{sawyer, Wo-00} contain an introduction to the theory of Martin boundaries in the discrete setting. We briefly remind basic results of this theory. If an origin $o$ of the state space is fixed, the Martin kernel of the process is \[K(x,y) = \frac{G(x,y)}{G(o,y)}.\] The Martin compactification of the state space is the smallest compactification that allows to extend continuously the functions $K(x, \cdot)$. The set of points added to get the Martin compactification is called the Martin boundary of the process. The two main theorems related to Martin boundaries are the Poisson-Martin representation theorem, which states that harmonic functions of the process can be represented as integrals of the extended Martin kernel on the Martin boundary, and the convergence theorem which states that for every state $x$, the process converges $\mathbb{P}_x$-almost surely to a random variable that takes its values on the Martin boundary.

Another interest towards Green functions comes from statistical mechanics models. The \emph{transfer impedance theorem} of \cite{Burton_Pemantle} expresses the distribution of a random spanning tree on a periodic graph with determinants of matrices involving Green functions. A similar result appears in \cite[Section 6.2]{BdTR} for random forests on isoradial graphs.
An unexpected application of the asymptotics of the Green functions of random walks appeared in \cite{AM22}, in which the authors use it to compute the limit shape of the \emph{leaky abelian sandpile model}, a cellular automaton that first arose in statistical physics.

\subsection{Description of the model}\label{subsec:model}

In this work, we consider a discrete time Markov-additive process \sloppy $\left(Z_n\right)_{n \geq 0} = \left( A_n, M_n \right)_{n \geq 0}$ on the state space $\Z^d \times \{1, \ldots, p\}$ where $d,p \in \N := \{1, 2, \ldots\}$, that is, a Markov chain on $\Z^d \times \{1, \ldots, p\}$ such that for each $x,x' \in \Z^d$ and $i,i' \in \{1, \ldots, p \}$,
\begin{equation}\label{eq:prop_markov_addit}
\mathbb{P}_{(x,i)}\big( (A_{1}, M_{1}) = (x',i') \big) = \mathbb{P}_{(0,i)}\big( (A_{1}, M_{1}) = (x'-x,i') \big).
\end{equation}
The process $\left(A_n\right)_{n \geq 0}$ is called the \emph{additive} part of the process $\left(Z_n\right)_{n \geq 0}$  while $\left(M_n\right)_{n \geq 0}$ is called its \emph{Markovian} part. The set $\{1, \ldots, p \}$ will be called the \emph{modulating set}.

Generalities on Markov-additive processes and their applications to queuing theory can be found in \cite[Chapter XI]{Asm03}.

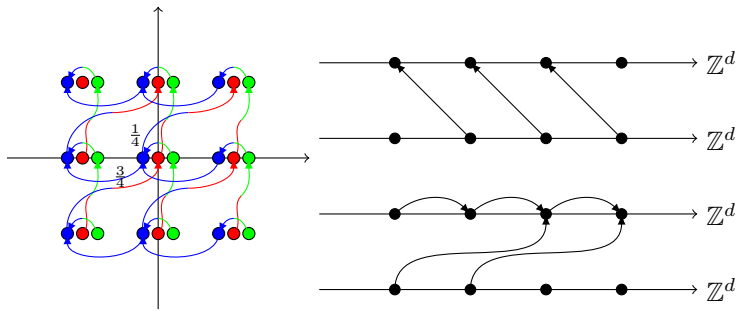
\begin{figure}[htbp]
    \centering
    \begin{tikzpicture}
        \draw[->](-2,0) -- (2,0);
        \draw[->] (0,-2) -- (0,2);
        \draw[fill=blue] (-0.2, 0) circle (0.08);
        \draw[fill=red] (0, 0) circle (0.08);
        \draw[fill=green] (0.2, 0) circle (0.08);
        
        \draw[fill=blue] (-0.2, 1) circle (0.08);
        \draw[fill=red] (0, 1) circle (0.08);
        \draw[fill=green] (0.2, 1) circle (0.08);
        
        \draw[fill=blue] (-0.2, -1) circle (0.08);
        \draw[fill=red] (0, -1) circle (0.08);
        \draw[fill=green] (0.2, -1) circle (0.08);

        \draw[fill=blue] (-1.2, 0) circle (0.08);
        \draw[fill=red] (-1, 0) circle (0.08);
        \draw[fill=green] (-0.8, 0) circle (0.08);

        \draw[fill=blue] (0.8, 0) circle (0.08);
        \draw[fill=red] (1, 0) circle (0.08);
        \draw[fill=green] (1.2, 0) circle (0.08);

        \draw[fill=blue] (-1.2, 1) circle (0.08);
        \draw[fill=red] (-1,1) circle (0.08);
        \draw[fill=green] (-0.8, 1) circle (0.08);

        \draw[fill=blue] (-1.2, -1) circle (0.08);
        \draw[fill=red] (-1,-1) circle (0.08);
        \draw[fill=green] (-0.8, -1) circle (0.08);

        \draw[fill=blue] (0.8, 1) circle (0.08);
        \draw[fill=red] (1, 1) circle (0.08);
        \draw[fill=green] (1.2, 1) circle (0.08);

        \draw[fill=blue] (0.8, -1) circle (0.08);
        \draw[fill=red] (1, -1) circle (0.08);
        \draw[fill=green] (1.2, -1) circle (0.08);

        \draw[red] (0,0) to[out=45,in=-135] (0.1,0.5);
        \draw[->, >=latex, green] (0.1,0.5) to[out=45,in=-90] (0.2,1);

        \draw[red, shift={(-1 ,0)}] (0,0) to[out=45,in=-135] (0.1,0.5);
        \draw[->, >=latex, green, shift={(-1 ,0)}] (0.1,0.5) to[out=45,in=-90] (0.2,1);

        \draw[red, shift={(-1 ,-1)}] (0,0) to[out=45,in=-135] (0.1,0.5);
        \draw[->, >=latex, green, shift={(-1 ,-1)}] (0.1,0.5) to[out=45,in=-90] (0.2,1);

        \draw[red, shift={(0 ,-1)}] (0,0) to[out=45,in=-135] (0.1,0.5);
        \draw[->, >=latex, green, shift={(0 ,-1)}] (0.1,0.5) to[out=45,in=-90] (0.2,1);

        \draw[red, shift={(1 ,0)}] (0,0) to[out=45,in=-135] (0.1,0.5);
        \draw[->, >=latex, green, shift={(1 ,0)}] (0.1,0.5) to[out=45,in=-90] (0.2,1);

        \draw[red, shift={(1 ,-1)}] (0,0) to[out=45,in=-135] (0.1,0.5);
        \draw[->, >=latex, green, shift={(1 ,-1)}] (0.1,0.5) to[out=45,in=-90] (0.2,1);

        \draw[green] (0.2,0) to[out=115,in=0] (0,0.2);
        \draw[->, >=latex, blue] (0,0.2) to[out=180,in=65] (-0.2,0);

        \draw[green, shift={(-1 ,-1)}] (0.2,0) to[out=115,in=0] (0,0.2);
        \draw[->, >=latex, blue, shift={(-1 ,-1)}] (0,0.2) to[out=180,in=65] (-0.2,0);

        \draw[green,shift={(-1 ,0)}] (0.2,0) to[out=115,in=0] (0,0.2);
        \draw[->, >=latex, blue,shift={(-1 ,0)}] (0,0.2) to[out=180,in=65] (-0.2,0);

        \draw[green, shift={(0 ,-1)}] (0.2,0) to[out=115,in=0] (0,0.2);
        \draw[->, >=latex, blue, shift={(0 ,-1)}] (0,0.2) to[out=180,in=65] (-0.2,0);

        \draw[green, shift={(-1 ,1)}] (0.2,0) to[out=115,in=0] (0,0.2);
        \draw[->, >=latex, blue, shift={(-1 ,1)}] (0,0.2) to[out=180,in=65] (-0.2,0);

        \draw[green, shift={(0 ,1)}] (0.2,0) to[out=115,in=0] (0,0.2);
        \draw[->, >=latex, blue, shift={(0 ,1)}] (0,0.2) to[out=180,in=65] (-0.2,0);

        \draw[green, shift={(1 ,-1)}] (0.2,0) to[out=115,in=0] (0,0.2);
        \draw[->, >=latex, blue, shift={(1 ,-1)}] (0,0.2) to[out=180,in=65] (-0.2,0);

        \draw[green, shift={(1 ,0)}] (0.2,0) to[out=115,in=0] (0,0.2);
        \draw[->, >=latex, blue, shift={(1 ,0)}] (0,0.2) to[out=180,in=65] (-0.2,0);

        \draw[green, shift={(1 ,1)}] (0.2,0) to[out=115,in=0] (0,0.2);
        \draw[->, >=latex, blue, shift={(1 ,1)}] (0,0.2) to[out=180,in=65] (-0.2,0);

        \draw[blue] (-0.2,0) to[out=100,in=-180] (0.4,0.6);
        \draw[->, >=latex, red] (0.4,0.6) to[out=0,in=-90] (1,1);

        \draw[blue, shift={(-1,0)}] (-0.2,0) to[out=100,in=-180] (0.4,0.6);
        \draw[->, >=latex, red, shift={(-1,0)}] (0.4,0.6) to[out=0,in=-90] (1,1);

        \draw[blue, shift={(-1,-1)}] (-0.2,0) to[out=100,in=-180] (0.4,0.6);
        \draw[->, >=latex, red, shift={(-1,-1)}] (0.4,0.6) to[out=0,in=-90] (1,1);

        \draw[blue, shift={(0,-1)}] (-0.2,0) to[out=100,in=-180] (0.4,0.6);
        \draw[->, >=latex, red, shift={(0,-1)}] (0.4,0.6) to[out=0,in=-90] (1,1);

        \draw[->, >=latex, blue] (-0.2,0) to[out = -90,in = -90] (-1.2,0);
        \draw[->, >=latex, blue, shift={(1,0)}] (-0.2,0) to[out = -90,in = -90] (-1.2,0);
        \draw[->, >=latex, blue, shift={(0,1)}] (-0.2,0) to[out = -90,in = -90] (-1.2,0);
        \draw[->, >=latex, blue, shift={(0,-1)}] (-0.2,0) to[out = -90,in = -90] (-1.2,0);
        \draw[->, >=latex, blue, shift={(1,1)}] (-0.2,0) to[out = -90,in = -90] (-1.2,0);
        \draw[->, >=latex, blue, shift={(1,-1)}] (-0.2,0) to[out = -90,in = -90] (-1.2,0);

        \draw (-0.5,-0.25) node{\tiny{$\frac{3}{4}$}};
        \draw (-0.3,0.3) node{\tiny{$\frac{1}{4}$}};
    \end{tikzpicture}
    \begin{tikzpicture}
         \draw[->] (-2,0) -- (3,0);
         \draw (3,0) node[right]{$\Z^d$};

         \draw[->, shift={(0,1)}] (-2,0) -- (3,0);
         \draw[ shift={(0,1)}] (3,0) node[right]{$\Z^d$};

         \draw[->, shift={(0,2)}] (-2,0) -- (3,0);
         \draw[ shift={(0,2)}] (3,0) node[right]{$\Z^d$};

         \draw[->, shift={(0,3)}] (-2,0) -- (3,0);
         \draw[ shift={(0,3)}] (3,0) node[right]{$\Z^d$};

         \draw[fill] (-1,0) circle (0.07);
         \draw[fill] (0,0) circle (0.07);
         \draw[fill] (1,0) circle (0.07);
         \draw[fill] (2,0) circle (0.07);

         \draw[fill] (-1,1) circle (0.07);
         \draw[fill] (0,1) circle (0.07);
         \draw[fill] (1,1) circle (0.07);
         \draw[fill] (2,1) circle (0.07);

         \draw[fill] (-1,2) circle (0.07);
         \draw[fill] (0,2) circle (0.07);
         \draw[fill] (1,2) circle (0.07);
         \draw[fill] (2,2) circle (0.07);

         \draw[fill] (-1,3) circle (0.07);
         \draw[fill] (0,3) circle (0.07);
         \draw[fill] (1,3) circle (0.07);
         \draw[fill] (2,3) circle (0.07);

         \draw[->,  >=latex] (-1,0) to[out=90, in=-90] (1,1);
         \draw[->,  >=latex, shift={(1,0)}] (-1,0) to[out=90, in=-90] (1,1);

         \draw[->,  >=latex, shift={(1,0)}] (-1,2) to (-2,3);
         \draw[->,  >=latex, shift={(2,0)}] (-1,2) to (-2,3);
         \draw[->,  >=latex, shift={(3,0)}] (-1,2) to (-2,3);

         \draw[->, >=latex] (-1,1) to[out=45,in=135] (0,1);
         \draw[->, >=latex,shift={(1,0)}] (-1,1) to[out=45,in=135] (0,1);
         \draw[->, >=latex,shift={(2,0)}] (-1,1) to[out=45,in=135] (0,1);
    \end{tikzpicture}
    \caption{On the left side, $d = 2$ and $p = 3$. On the right side, $d = 1$, $p = 4$.}
    \label{fig:visualizing_the_process}
\end{figure}

Figure~\ref{fig:visualizing_the_process} provides two convenient ways of visualizing the kind of process studied. On the left, we see it as $\Z^d$ where every vertice appears $p$ times (in $p$ colors). On the right, we see it as $p$ layers of $\Z^d$, like a random walk in $\Z^{d+1}$ confined between two parallel hyperplanes (although in the process studied here, the order of layers is irrelevant). In both cases, the jumps are invariant by translation along $\Z^d$, as indicated by \eqref{eq:prop_markov_addit}.

\begin{figure}[htbp]
    \centering
    \includegraphics[scale=0.3]{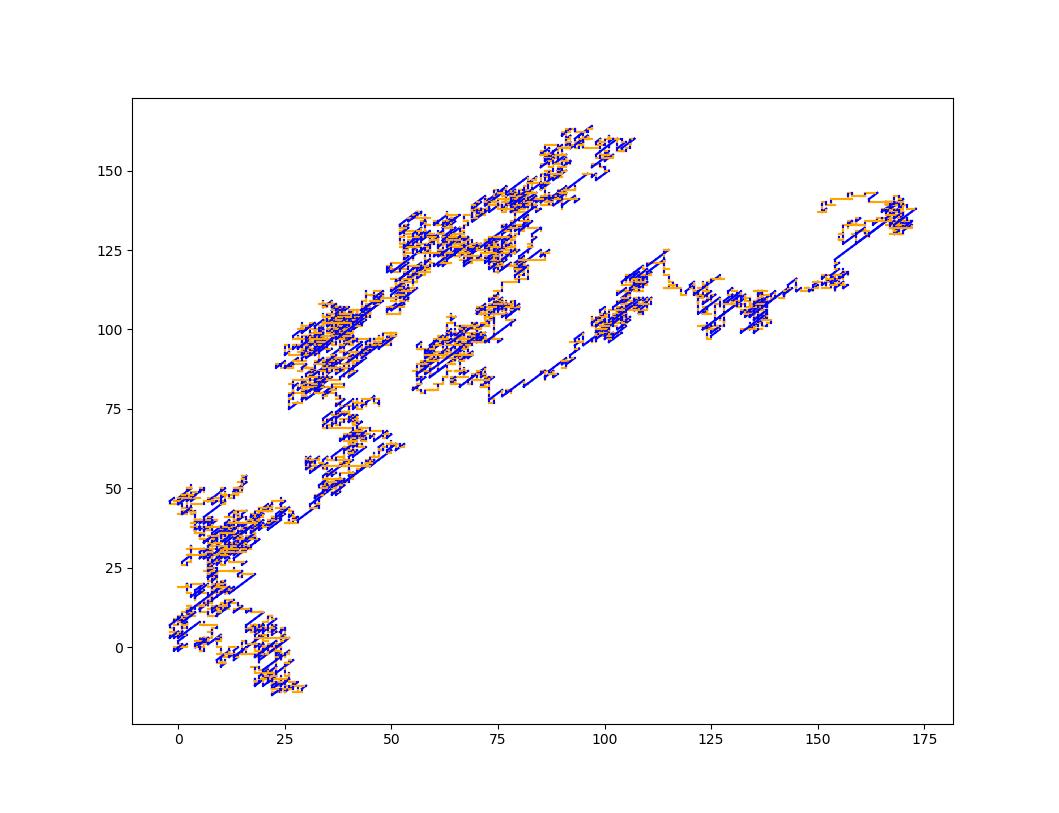}
    \includegraphics[scale=0.3]{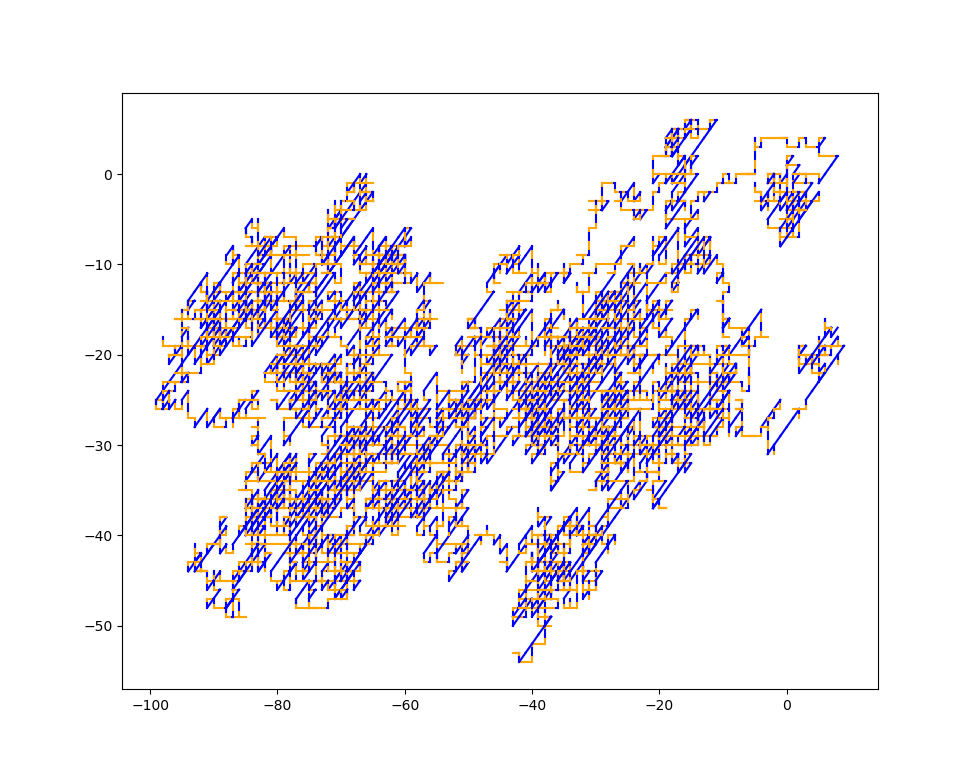}
    \caption{Two trajectories of 10.000 iterations of the process, with $p=d=2$, using the visualization of the left image of Figure \ref{fig:visualizing_the_process} with colors black and grey. The left-hand one has a north-east oriented drift while the right-hand one has a smaller drift oriented south-west.}
    \label{fig:trajectoires}
\end{figure}

\subsection{Notations and assumptions}

We start with general notations and then introduce the main definitions of this article. Throughout the paper:
\begin{itemize}
    \item $\mathbf{i}$ denotes the usual imaginary number, while $i$ denotes an index of a vector or matrix;
    \item $x \cdot y$ denotes the usual scalar product of the vectors $x$ and $y$;
    \item for $(x,i) \in \Z^d \times \{1, \ldots, p \}$, $\mathbb{P}_{(x,i)}$ denotes the probability when the process starts from the state $(x,i)$;
    \item $\nabla f(x)$ denotes the gradient of the function $f$ at $x$ and $\mathbf{H}_f(x)$ the Hessian matrix of $f$ at $x$;
    \item if we do not precise which norm we use, $\| x \|$ denotes the euclidean norm of a vector~$x$ and $\|M\|$ denotes the operator norm of the matrix $M$, induced by the euclidean norm.
\end{itemize}

The \emph{Green function} $G$ of the process is defined for $z,x \in \Z^d$ and $i,j \in \{1, \ldots, p\}$ by
\begin{equation}\label{eq:def_green}
G\big( (z,i), (x,j) \big) := \sum_{n=0}^{+ \infty} \mathbb{P}_{(z,i)}\left( Z_n = (x,j) \right).
\end{equation}
It is also the expected number of visits in $(x,j)$ starting from $(z,i)$, because \[G\big( (z,i), (x,j) \big) =  \sum_{n=0}^{+ \infty} \mathbb{E}_{(z,i)}\left[ \mathds{1}_{\left\{Z_n = (x,j)\right\}} \right] =   \mathbb{E}_{(z,i)}\left[ \sum_{n=0}^{+ \infty} \mathds{1}_{\left\{Z_n = (x,j)\right\}} \right].\]
The Green function takes values in $[0, + \infty]$. If the process is irreducible, then $G$ does not vanish, and it is finite if and only if the process is transient. In the non-centered case, the process is transient, but it is not as easy to prove as in the case of a sum of independent and identically distributed variables. It will be a consequence of the integral formula in Theorem \ref{th:form_int}.
As a consequence of \eqref{eq:prop_markov_addit}, $G\big( (z,i), (x,j) \big) = G\big( (0,i), (x-z,j) \big)$, so we will always work with values of the Green function of the form $G\big( (0,i), (x,j) \big)$.

Because of the property \eqref{eq:prop_markov_addit} of Markov-additive processes, the transitions of $\left(Z_n\right)_{n \geq 0}$ are fully described by the sub-probability measures on $\Z^d$ 
\begin{equation}
\mu_{i,j} := \mathbb{P}_{(0,i)}\big( (A_{1}, M_{1}) = (\cdot,j) \big),~~~~~~ 1 \leq i,j \leq p.
\end{equation}
Those measures can be stored in a \emph{jump matrix} $\mu = \left( \mu_{i,j}\right)_{1 \leq i,j \leq p}$. The \emph{support} of $\mu$ is defined by \[{\supp(\mu) := \bigcup_{i,j =1}^p \supp(\mu_{i,j})}.\]
For example, on the left side of Figure~\ref{fig:visualizing_the_process}, if the colors are ordered from lightest to darkest, then the jump matrix is  \[{\mu = \begin{pmatrix}
    \frac{3}{4} \delta_{(-1,0)}&\frac{1}{4}\delta_{(1,1)}&0\\
    0&0&\delta_{(0,1)} \\
    \delta_{(0,0)}&0&0
\end{pmatrix}}.\] 

The additive part $\left(A_n\right)_{n \geq 0}$ is not necessarily a Markov chain, however it is well known that the Markovian part $\left(M_n\right)_{n \geq 0}$ is a Markov chain on $\{1, \ldots , p\}$ with transition matrix
\begin{equation}\label{eq:markovian_part_transition_matrix}
\left( \displaystyle \sum_{x\in\Z^d} \mu_{i,j}(x) \right)_{1 \leq i,j \leq p}.
\end{equation}

We define the first and second moments of the process, adapting the functional definitions of \cite[(1.16)]{MR0978023} to our discrete case. Since we have $p^2$ types of jumps (one for each of the $\mu_{i,j}$) we have $p^2$ types of first and second order moments. They will be combined to define the \emph{global} drift in Definition \ref{def:global_drift} and the \emph{energy matrix} in Definition \ref{def:energy_matrix}.
\begin{defi}\label{def:moments}
Provided the following sums are finite:
    \begin{itemize}
        \item the \emph{local drifts} are the $m_{i,j} :=  \displaystyle\sum_{x \in \Z^d}  \mu_{i,j}(x) x $ for $1 \leq i,j \leq p$;
    \item the \emph{second order moments} are the ${\Sigma_{i,j} :=  \left( \displaystyle \sum_{x \in \Z^d} x_k x_l \mu_{i,j}(x) \right)_{1 \leq k,l \leq d}}$ for $1 \leq i,j \leq p$, where $x = (x_1, \ldots, x_d)$.
    \end{itemize}
\end{defi}

We can now state the assumptions under which we will work. From now on, we assume the following.
\begin{assump}\label{assump:main_assumptions}
\begin{enumerate}
    \item The process $\left(Z_n\right)_{n \geq 0}$ is \emph{irreducible}.
    \item The process $\left(Z_n\right)_{n \geq 0}$ is \emph{aperiodic}.
    \item\label{it:assump:expo_moments} The process has finite exponential moments, that is, for every $i,j \in \{1, \ldots, p \}$ and $\alpha >0$, \begin{equation}
        \sum_{x \in \Z^d} e^{\alpha \|x\|} \mu_{i,j}(x) < +\infty.
    \end{equation}
\end{enumerate}
\end{assump}

By aperiodicity, we mean that for all $(x,i) \in \Z^d \times \{1, \ldots, p \}$, \[\mathrm{gcd} \left\{ n \geq 1~|~ \mathbb P_{(x,i)}\big( Z_n = (x,i) \big) > 0 \right\} = 1.\] Some authors, like \cite{MR0978023, MR0890364}, use a different definition of aperiodicity involving changes of sections (see the remark after Proposition~\ref{prop:inv_change_sect}).

A special case of process with finite exponential moments is when the support of $\mu$ is finite.

Here are a few easy consequences of Assumptions \ref{assump:main_assumptions}:
\begin{itemize}
    \item The existence of moments of order $1$ and $2$ is a consequence of the finite exponential moments.
    \item The Markovian part $(M_n)$ is an irreducible aperiodic Markov chain on a finite state space. Therefore, it has a unique stationary distribution, denoted by $\pi$, whose all entries are non-zero. We will identify $\pi$ with a row vector of $\R^p$. 
    \item Besides, $1$ is the simple leading eigenvalue of the transition matrix of $(M_n)$, as a consequence of the Perron-Frobenius theorem.
\end{itemize}

Using the stationary distribution $\pi$, we can define the global drift of the process, like in Babillot's work \cite[(1.17)]{MR0978023} or in Theorem 6.7 of \cite{Wo-00}.

\begin{defi}\label{def:global_drift}
    The \emph{global drift} $m$ of the walk is obtained by weighting the local drifts by the stationary distribution~$\pi$, that is 
    \begin{equation}\label{eq:drift}
   m := \sum_{i,j =1}^p \pi_i m_{i,j}.
    \end{equation}
\end{defi}
The longer the process stays on the layer $\Z^d \times \{i\}$, the more local drifts starting from this layer count in the global drift.

In this paper, the asymptotics of the Green function will be obtained in the non-centered case, which leads to the following assumption. 

\begin{assumption}\label{assump:non_centered}
    The process is non-centered, i.e., the drift $m$ defined in \eqref{eq:drift} is non-zero.
\end{assumption}

Unlike Assumptions~\ref{assump:main_assumptions}, we will not always work under Assumption \ref{assump:non_centered}. Indeed, in Section \ref{sec:laplace_doob}, we will perform a Doob transform, which we will \emph{at first} not know if it yields a centered process, hence the need for results without this assumption. It will eventually be established in item \ref{it:non_centered_doob} of Proposition \ref{prop:prop_doob_transform} that the Doob transform is non-centered. Up to Section \ref{sec:fourier_laplace}, we will mention in the statement of results when Assumption~\ref{assump:non_centered} is needed, but starting from Section \ref{sec:proof_main}, Assumption \ref{assump:non_centered} will always be made.

\subsection{Main result}\label{subsec:main_result}

Our main theorem, similar to the Proposition 3.27 of \cite{Du-20}, provides the asymptotics of the Green function in the non-centered case. It is an extension of Theorem 2.2 of \cite{NeSp-66} and a lattice case version of Theorem (2.9) of \cite{MR0978023}.

\begin{thm} \label{thm:main}
    Under Assumptions~\ref{assump:main_assumptions} and~\ref{assump:non_centered}, there exist a continuous function $c : \mathbb{S}^{d-1} \to \R^d$ and for each $i,j \in \{1, \ldots, p\}$ a continuous function $\chi_{i,j} : \mathbb{S}^{d-1} \to \R$ and positive constants $C_1, C_2$ such that $C_1 \leq \chi_{i,j} \leq C_2$ and 
    \begin{equation}\label{eq:main}
        G\big( (0,i), (x,j) \big) \sim \chi_{i,j} \left( \frac{x}{\|x\|} \right) \|x\|^{-\frac{d-1}{2}} e^{- c\left( \frac{x}{\|x\|} \right) \cdot x}
    \end{equation}
    as $\|x\|$ tends to infinity.
\end{thm}
The functions $c$ and $\chi_{i,j}$ will be characterized in the precise formulation of Theorem \ref{thm:main_precise}. 

Another formulation, like the one used in Theorem (2.9) of \cite{MR0978023}, is to say that if a direction $u \in \mathbb{S}^{d-1}$ is fixed, then \[t^{\frac{d-1}{2}} e^{t c(u) \cdot u} G\big( (0,i), (t u,j) \big) \xrightarrow[t \to +\infty]{}\chi_{i,j}(u),\] and the convergence is uniform in $u$, which comes to saying that the asymptotic equivalent of \eqref{eq:main} is not an equivalent in a fixed direction, but a ``global'' equivalent.

As mentioned in Section \ref{subsec:context}, the usual consequence of the asymptotic expansion of the Green function is the description of the Martin boundary of the Markov chain. For the process studied here, Dussaule proved in \cite{Du-20} that the Martin boundary is homeomorphic to \emph{one} sphere of $\R^d$. As we do not have an alternative proof of this result, which is a consequence of Theorem \ref{thm:main}, we will not get into more details.

\subsection{Outline of the paper}\label{subsec:outline}

The general idea of the proof of Theorem \ref{thm:main} is to establish an integral formula of the Green function using Fourier tools, see Theorem \ref{th:form_int}, and then to get the asymptotics of the Green function by studying the integral.

\vspace{0.3cm}

In Section \ref{sec:fourier_laplace}, we first generalize convolution and the Fourier and Laplace transforms to matrices of measures. 

The Fourier transform is studied in Section~\ref{subsec:spectra_rad_fourier}. The decomposition of Proposition~\ref{prop:jordan_chevalley} isolates the leading eigenvalue of the Fourier transform. A Taylor expansion of order~$2$ of this leading eigenvalue at $0$ will be crucial to obtain the integral formula of the Green function. To get this expansion, we compute the two first derivatives of the leading eigenvalue at $0$ in Theorem \ref{thm:diff_k}. Although the first derivative at $0$ is directly related to the global drift $m$, we will need a \emph{change of section} that comes from \cite{MR0978023, MR0890364}, see Definition \ref{defi:change_section}, to compute the second derivative.

In Section \ref{sec:laplace_doob}, we study the spectral radius of the (real) Laplace transform in order to use a \emph{Doob transform} (also called $h$-process) introduced in Definition \ref{def:doob_transform}. The Doob transform kills the exponential decay of the Green function, see Proposition \ref{prop:link_green_functions}, which leaves us to prove the polynomial decay of \eqref{eq:main}.

\vspace{0.3cm}

Thanks to these results, we can establish the integral formula and study its asymptotics in Section \ref{sec:proof_main}.

Section \ref{subsec:integral_formula} is devoted to the proof of the integral formula, using Fourier analysis tools. Then in Section \ref{sec:dominating_part_without_proofs}, we identify the dominating part of the integral to make it easier to study. This step of the proof is more delicate than in the case of Ney and Spitzer's theorem. We use \emph{dyadic splitting of integrals}, an advanced Fourier analysis technique coming from Babillot in \cite{MR0978023}. As we find this technique has its own interest, independently of the problem studied here, we separated this section from the rest of the text in Appendix \ref{subsec:dyadic_integrals}.

Once the integral has been transformed into a simpler one, Woess's proof of Ney and Spitzer's theorem from his book \cite{Wo-00} is easily adapted, see Theorem \ref{thm:main_with_doob_transform}. It is based on a Weierstrass decomposition theorem that we adapt from Woess in Section \ref{sec:complex_analysis_ingredient_without_proofs}. As the proof of this complex analysis result is rather similar to Woess's one, it is transferred to Appendix \ref{sec:app:complex_analysis}.

\subsection{Original contributions}\label{subsec:contributions}

As mentioned earlier, Theorem \ref{thm:main} already appeared in Dussaule's work \cite[Proposition 3.27]{Du-20} with another proof. In this section, we explain the differences between the two proofs and more generally the original contributions of this article.

\vspace{0.3cm}

We start with Section \ref{sec:fourier_laplace}, which consists of preliminary results on the Fourier and Laplace transforms.

In Section \ref{subsec:spectra_rad_fourier}, the use of a \emph{change of section}, which comes from \cite{MR0978023, MR0890364}, allows us to give an explicit formula for the Hessian matrix of the leading eigenvalue of the Fourier transform, called the \emph{energy matrix} in \cite[(1.19)]{MR0978023}, while Dussaule proves that it is positive-definite in \cite[Proposition 3.5]{Du-20} without such an explicit formula. We insist that in our setting where the modulating set is finite, if the support of the process is also finite, then one can compute explicitly the energy matrix. Indeed, as seen in Definition~\ref{def:energy_matrix}, it requires to compute the stationary distribution $\pi$ and an \emph{appropriate} change of section $g$. Computing $\pi$ comes down to solving the linear equation~$\pi \left( \widehat \mu(0) - \mathrm{I}_p \right) = 0$, while equation \eqref{eq:appropriate_change_section} in the proof of Proposition \ref{prop:specific_g} shows that computing an appropriate change of section comes down to solving a linear equation of the form $\left( \mathrm{I}_p - \widehat \mu(0) \right) g = M$ where the unknown is $g$. Besides, using a change of section allows us to give a probabilistic interpretation of this Hessian with the $g$-changed process, as in the proof of Proposition \ref{prop:energy_mat_def_pos}.

In Section \ref{sec:laplace_doob}, the main result is the Theorem \ref{thm:hennequin}, which is a classical result of Hennequin, \cite[Proposition 4.4]{He-63}. We give more details, using the Lagrange multiplier theorem to complete Hennequin's proof.

\vspace{0.3cm}

Let us move to Section \ref{sec:proof_main}, which proves the main theorem. The proof strategy we use is quite different from Dussaule's. His approach generalizes the proof of Ney and Spitzer in \cite{NeSp-66}. He first establishes a local limit theorem, that is the study of the asymptotics of $\mathbb P_{(0,i)} \left( Z_n = (x,j) \right)$ as $x$ goes to infinity, in \cite[Proposition 3.14]{Du-20}, and then uses it to study the Green function. Our approach, which comes from Woess's proof of Ney and Spitzer's theorem in \cite[Chapter 25]{Wo-00}, is to study directly the Green function.

In Section \ref{subsec:integral_formula}, we establish an integral formula of the Green function, which is the analog of {\cite[(25.17)]{Wo-00}.} In Section \ref{sec:dominating_part_without_proofs}, our main contribution is the adaptation of Babillot's dyadic splitting of integrals from \cite[(2.38)-(2.42)]{MR0978023} to our lattice case in order to identify the dominating part of the integral. We hope that this little-known technique will benefit from this presentation in English and from the details we have added to the proofs, see Appendix \ref{subsec:dyadic_integrals}.

Finally, we adapt the proof of Theorem (25.15) from \cite{Wo-00}, which uses a Weierstrass preparation theorem, to establish the main result of the article.

\section{Fourier and Laplace transforms}\label{sec:fourier_laplace}

    \subsection{Generalized Fourier and Laplace transforms, generalized convolution}
        
    We extend convolution, the Fourier and Laplace transforms of finite measures to ${p\times p}$~matrices of finite measures. We do it coefficient by coefficient for the Fourier and Laplace transforms, and using the matrix multiplication formula for the convolution.
    
    More precisely, the Fourier and Laplace transforms of $\mu$, denoted respectively $\widehat{\mu}$ and $L\mu$, are complex matrices whose coefficients are:
    \begin{equation}\label{eq:def_fourier}
        \forall \theta \in \R ^d , ~~~~~~ \widehat{\mu}_{i,j}(\theta) := \displaystyle \sum_{x \in \Z ^d} e^{\mathbf{i} x \cdot \theta} \mu_{i,j}(x),
    \end{equation}
    \begin{equation}\label{eq:def_laplace}
        \forall c \in \C^d,~~~~(L\mu)_{i,j}(c) :=  \sum_{x \in \Z^d} e^{c \cdot x} \mu_{i,j}(x).
    \end{equation}
    The Laplace transform is defined on the whole of $\mathbb{C}^d$, and it is even an entire function, as a consequence of the assumption of finite exponential moments, that is point \ref{it:assump:expo_moments} of Assumptions \ref{assump:main_assumptions}.
    For $\theta \in \R^d$, we have \[\widehat{\mu}(\theta) = L\mu (\mathbf{i} \theta).\]
    Therefore, for any function $f$ analytic in a neighborhood of $\widehat{\mu}(0) = L\mu(0)$, we have
    \begin{equation}\label{eq:link_deriv_fourier_laplace}
    \left( f \circ \widehat{\mu} \right)^{(k)} (0) = \mathbf{i}^k \left( f \circ L\mu \right)^{(k)} (0).
    \end{equation}
    Besides, according to \eqref{eq:markovian_part_transition_matrix}, 
    \begin{equation}\label{eq:fourier_en_zero_matrice_transition}
        \widehat{\mu}(0) = \left( \sum_{x \in \Z^d} \mu_{i,j}(x) \right)_{1\leq i,j \leq p} \text{ is the transition matrix of the Markovian part } \left(M_n\right)_{n \geq 0}.
    \end{equation}
    Therefore,
    \begin{equation}\label{eq:invariance_pi}
        \pi \widehat \mu(0) = \widehat \mu (0).
    \end{equation}

    The convolution of matrices of measures is defined by
    \begin{equation}
        \left(\mu \star \nu\right)_{i,j} := \displaystyle \sum_{k = 1}^p \mu_{i,k} \star \nu_{k,j}.
    \end{equation}
    The distribution of $Z_n$ is given by the extended convolution powers of $\mu$, like in the simple case of sums of independent and identically distributed variables: 
    \begin{equation}\label{eq:loi_conv}
       \forall n \in \N, ~~~~~~  \mathbb{P}_{(0,i)}\left( Z_n = (x,j) \right) = \left(\mu^{\star n}\right)_{i,j} (x).
    \end{equation}
    The link between convolution and Fourier transform is easily generalized to extended convolution and Fourier transform. It becomes $\widehat{\mu \star \nu} = \widehat{\mu} \widehat{\nu}$, where the right-hand side is a product of complex matrices. In particular, 
    \begin{equation}\label{eq:fourier_convol_link}
    \forall n \in \N, ~~~~~~ \widehat{\mu^{\star n}} = \widehat{\mu}^n.
    \end{equation}
    A similar result holds for the Laplace transform.

    \subsection{Leading eigenvalue of the Fourier transform, changes of sections}\label{subsec:spectra_rad_fourier}

     The first step in the asymptotical study of the Green function is to establish an integral formula, like in \cite{MR0978023, Du-20, NeSp-66,Wo-00}. The heuristic behind the formula is to sum \eqref{eq:loi_conv} to obtain \[\sum_{n=0}^N \mathbb{P}_{(0,i)}\left( Z_n = (x,j) \right) = \sum_{n=0}^N \left(\mu^{\star n}\right)_{i,j}(x), \]
     and using Fourier and inverse Fourier in the right-hand side, get 
     \begin{equation}\label{eq:integral_formula_finite}
     \sum_{n=0}^N \mathbb{P}_{(0,i)}\left( Z_n = (x,j) \right) = \frac{1}{(2\pi)^d} \int_{\left[- \pi,\pi\right]^d} \sum_{n=0}^N (\widehat{\mu}^n (\theta))_{i,j} e^{- \mathbf{i} x \cdot \theta} \mathrm{d}\theta.
     \end{equation}
     Taking the limit as $N$ goes to infinity leads to 
     \begin{equation}\label{eq:integral_formula}
         G \left( (0,i), (x,j) \right) = \frac{1}{(2\pi)^d} \int_{[-\pi, \pi] ^d} {\left( \mathrm{I}_p - \widehat{\mu}(\theta) \right)^{-1}}_{i,j} e^{- \mathbf{i} x \cdot \theta} \mathrm{d} \theta.
     \end{equation}
     However, taking the limit of the right-hand side of \eqref{eq:integral_formula_finite} is not that obvious, as $\left(\mathrm{I}_p - \widehat{\mu}(\theta)\right)^{-1}$ might explode when $\theta$ gets close to $0$, because eigenvalues of $\hat{\mu}(\theta)$ may approach $1$. Therefore, we will need to study more carefully the eigenvalues of the Fourier transform to use properly the dominated convergence theorem.


    We mentioned a possible explosion of $\left(\mathrm{I}_p - \widehat{\mu}(\theta)\right)^{-1}$ when 
    $\theta$ gets close to $0$. Indeed, according to \eqref{eq:fourier_en_zero_matrice_transition}, $\widehat{\mu}(0)$ is the transition matrix of the Markovian part $\left(M_n\right)_{n \geq 0}$. Therefore, $1$ is an eigenvalue of $\widehat{\mu}(0)$, thus $\mathrm{I}_p - \widehat{\mu}(0)$ is not invertible. The following lemma shows that the only possible explosion is indeed around $\theta = 0$.
    
    \begin{lemma}\label{lemme:spec_rad_aper}
    For every $\theta \in [ - \pi, \pi]^d \setminus \{ 0 \}$, the spectral radius of $\widehat{\mu}(\theta)$ is strictly smaller than $1$.
\end{lemma}

\begin{proof}
    Gershgorin's theorem shows that the spectral radius of the Fourier transform is not larger than $1$, which leaves us with the case when it is $1$.
    
    Let $\theta \in [-\pi, \pi]^d$, and $\alpha$ an eigenvalue of $\widehat{\mu}(\theta)$ such that $|\alpha| = 1$, associated with an eigenvector \sloppy ${u = ( u_1 , \ldots, u_p )^{\mathsf{T}} \in \mathbb C^p}$. We prove that $\theta =0$.
    
    The equality $\widehat{\mu}(\theta) u = \alpha u$ means that for all $k \in \{ 1, \ldots, p \}$, \[\alpha u_k = \sum_{j = 1}^p \widehat{\mu}_{k,j}(\theta) u_j = \sum_{j = 1}^p \sum_{x \in \Z^d} \mu_{k,j}(x) e^{\mathbf{i} x \cdot \theta  } u_j.\] We choose $k$ such that $|u_k| = \|u\|_\infty := \max(|u_1|, \ldots, |u_p|)$. Using triangle inequality and $|\alpha| = 1$, we get
    \[ |u_k| \leq \sum_{j=1}^p \sum_{x \in \Z^d} \left| \mu_{k,j}(x) e^{\mathbf{i} x \cdot \theta } u_j \right| = \sum_{j=1}^p  \sum_{x \in \Z^d} \mu_{k,j}(x) |u_j| \leq \sum_{j=1}^p  \sum_{x \in \Z^d} \mu_{k,j}(x) |u_k| = |u_k|. \] 
    The equality case of the triangle inequality shows that all the complex numbers~$\mu_{k,j}(x) e^{\mathbf{i}x \cdot \theta } u_j$ for \sloppy ${j \in \{ 1,\ldots, p \}}$ and $x \in \Z^d$ have the same argument. The equality case of the second inequality shows that for all $j$ such that $\sum_{x \in \Z^d} \mu_{k,j}(x) > 0$, $|u_j| = |u_k|$. Therefore all the $e^{\mathbf{i}x \cdot \theta } u_j$ such that $\mu_{k,j}(x) >0$ are equal, because they have the same modulus and argument. Since $\sum_{j=1}^p \sum_{x \in \Z^d} \mu_{k,j}(x) e^{\mathbf{i} x \cdot \theta } u_j = \alpha u_k$, their common value is~$\alpha u_k$:
    \begin{equation}\label{eq:spec_rad_0}
        \text{if $\mu_{k,j} (x) >0$, then $e^{\mathbf{i} x \cdot \theta} u_j = \alpha u_k$.}
    \end{equation}

    Let $y,z \in \Z^d$, $k\in \{1 , \ldots, p \}$ such that $|u_k| = \|u\|_\infty$ and $j \in \{1, \ldots, p\}$. Applying~\eqref{eq:spec_rad_0} to $x = z-y$ and multiplying by $e^{\mathbf{i} y \cdot \theta}$ shows that
    \begin{equation}\label{eq:spec_rad_1}
    \textnormal{if the process can go from $(y,k)$ to $(z,j)$ in one jump, then $ e^{\mathbf{i}  z \cdot \theta } u_j = \alpha e^{ \mathbf{i}  y \cdot \theta } u_k$.}
    \end{equation}

    In particular, for such a $j$, $|u_j| = |u_k| = \|u\|_\infty$, so we can apply \eqref{eq:spec_rad_1} starting from~$j$ instead of $k$ and go on like this. Since the process is irreducible, we eventually find that all the components of $u$ have the same modulus and that for any $y,z \in \Z^d$ and~$k,j \in \{1, \ldots,p\}$:
    \begin{equation}\label{eq:spec_rad_2}
    \textrm{if the process can go from $(y,k)$ to $(z,j)$ in $t$ jumps, then $ e^{\mathbf{i}  z \cdot \theta } u_j = \alpha^t e^{ \mathbf{i}  y \cdot \theta } u_k$.}
    \end{equation}
    In particular, taking $(z,j) = (y,k)$, we get that for any $t$ such that the process can loop from $(k,y)$ to 
    $(k,y)$ in $t$ jumps, then $\alpha^t = 1$, which we note $\alpha \in \mathbb{U}_t$.  But if ${\mathrm{gcd}(t_1, \ldots, t_n) = 1}$, then $\bigcap_{i = 1}^n \mathbb{U}_{t_i} = \{1\}$. Therefore, the assumption of aperiodicity leads to $\alpha = 1$, hence \eqref{eq:spec_rad_2} combined with irreducibility show that the $ e^{\mathbf{i}  z \cdot \theta } u_j$ for $z \in \Z^d$ and $j \in \{ 1, \ldots, p \}$ are all equal. Taking $z = 0$ and $z = e_i$ (the $i^{\mathrm{th}}$ vector of the canonical basis), we get $u_j = e^{\mathbf{i} \theta_i} u_j$, so $e^{\mathbf{i} \theta_i} =1$ (because $u_j \neq 0$) hence $\theta_i = 0$. In conclusion, $\theta = 0$.
    \end{proof}

    The following decomposition of the Fourier transform $\widehat{\mu}$ isolates its dominating part, given by its leading eigenvalue, which will appear in the leading term in the integral formula of the Green function in \eqref{eq:dom_part}.
    
\begin{prop}\label{prop:jordan_chevalley}
    There exists a neighborhood $V$ of $0$ such that for $\theta \in V$, we can write \[\widehat{\mu}(\theta) = k(\theta) p(\theta) + Q(\theta),\]
    where:
    \begin{itemize}
        \item $k(\theta) \in \mathbb{C}$ is the simple leading eigenvalue of $\widehat{\mu}(\theta)$;
        \item $p(\theta)$ is the projection onto the eigenspace associated with the eigenvalue $k(\theta)$, along the sum of the other generalized eigenspaces;
        \item $p(\theta) Q(\theta) = Q(\theta)p(\theta) = 0$;
        \item the spectral radius of $Q(\theta)$ is strictly smaller than $|k(\theta)|$.
    \end{itemize}
\end{prop}
\begin{proof}
    We have seen that $\widehat{\mu}(0)$, which is the transition matrix of the Markovian part according to \eqref{eq:fourier_en_zero_matrice_transition}, has a simple leading eigenvalue as a consequence of the Perron-Frobenius theorem. By continuity of the spectrum, there exists a neighborhood $V$ of $0$ in which $\widehat{\mu}(\theta)$ has one simple leading eigenvalue, hence the existence of $k(\theta)$.

    The rest of the proposition is a direct consequence of the Jordan-Chevalley decomposition.
\end{proof}

We highlight some easy and useful relations in the previous decomposition.

\begin{lemma}\label{lem:ker_im_perp_jordan}
    The following equalities hold:
    \begin{equation}\label{eq:ker_im_perp_jordan}
        \ker p(0) = \im Q(0) = \im\left( \widehat{\mu}(0) - \mathrm{I}_p \right) = \pi^{\perp},
    \end{equation}

    \begin{equation}\label{eq:pi_p_0}
        \pi p(0) = \pi
    \end{equation}
    and
    \begin{equation}\label{eq:im_p_0}
        \im p(0) = \spn (\mathds 1)
    \end{equation}
    where $\mathds 1$ is the vector whose all entries are $1$. In particular, since $\mathds 1$ is in the image of the projection $p(0)$, we have $p(0) \mathds 1 = \mathds 1$.
    
\end{lemma}
\begin{proof}
    The equalities $\ker p(0) = \im Q(0) = \im\left( \widehat{\mu}(0) - \mathrm{I}_p \right)$ are a direct consequence of the Jordan-Chevalley decomposition: this space is simply the sum of the generalized eigenspaces of $\widehat \mu(0)$ associated with eigenvalues different from $1$. We have $\pi \widehat{\mu}(0) = \pi$ by \eqref{eq:invariance_pi}, so $\im\left( \widehat{\mu}(0) - \mathrm{I}_p \right) \subset \pi^{\perp}$. Besides, since $1$ is a simple eigenvalue of $\widehat{\mu}(0)$, $\dim \im\left( \widehat{\mu}(0) - \mathrm{I}_p \right) = d-1 = \dim \pi^{\perp}$, hence $\im\left( \widehat{\mu}(0) - \mathrm{I}_p \right) = \pi^{\perp}$, which concludes the proof of \eqref{eq:ker_im_perp_jordan}.

    Because of $\im Q(0) = \pi^\perp$, we get $\pi Q(0) = 0$. Therefore, the equation of invariance leads to \sloppy${\pi = \pi \widehat{\mu}(0) = \pi p(0) + \pi Q(0) = \pi p(0)}$, hence \eqref{eq:pi_p_0}.
    
    Since $\mathds 1$ is an eigenvector of $\widehat{\mu}(0)$ for the simple eigenvalue $1$, and $p(0)$ is a projection onto the eigenspace associated with the eigenvalue $1$, we have \eqref{eq:im_p_0}. 
\end{proof}

We will need the second order asymptotics of $k$ around $0$, or equivalently its two first derivatives at $0$.
As we will see, the gradient of $k$ is directly related to the drift $m$ defined in \eqref{eq:drift}. Its Hessian, however, is not as easily expressed with its second order moments~$\Sigma_{i,j}$. It uses a ``change of section'', which was introduced in \cite{MR0890364} and used in \cite{MR0978023}.

\begin{defi}\label{defi:change_section}
    A \emph{change of section}, denoted by $g$, is a $p \times d$ real matrix. The \emph{$g$-changed process} is the Markov-additive process with transitions $^g\mu$ defined by: \[\forall x \in \R^d,~~~~\forall i,j \in \{1, \ldots, p \},~~~~ ^g \mu_{i,j}(x) = \mu_{i,j} \left( x + g_j -g_i \right) \] where $g_i$ is the $i^{\mathrm{th}}$ row of $g$.
\end{defi}
    The additive part of the $g$-changed process no longer necessarily lies in $\Z^d$, but it is still in a countable subset of $\R^d$, namely the subgroup generated by the $\supp(\mu_{i,j}) + g_j - g_i$ for $1 \leq i,j \leq p$.

    Everything we defined for the original process has an analog for a $g$-changed process, which will be denoted with a left superscript ``$g$'', for example  $^g \pi$, $^g m$, $^g \widehat{\mu}$, $^g k$\ldots

Changes of sections preserve some properties of the process, as stated in the following proposition.

\begin{prop}\label{prop:inv_change_sect}
Let $g$ be any change of section.

\begin{enumerate}
    \item\label{it:changement_preserve_vp_fourier} For all $\theta \in \R^d$, the Fourier transform matrices $\widehat{\mu}(\theta)$ and $^g \widehat{\mu}(\theta)$ have the same eigenvalues. In particular, $^g k (\theta) = k(\theta)$.
    \item $ ^g \widehat{\mu}(0) = \widehat{\mu}(0)$, which means that the Markovian part of the $g$-changed process has the same transitions as the original one. In particular, the invariant distributions $^g \pi$ and $\pi$ are the same.
    \item\label{it:inv_change_section_3} The drifts $^g m$ and $m$ are the same.

    \item\label{it:supp_pas_inclus_hyperplan} There exists no affine hyperplane $H$ such that $\mathrm{supp}({}^g \mu) \subset H$.
    \end{enumerate}

\end{prop}
\begin{proof}
We first prove \ref{it:changement_preserve_vp_fourier}.
By definition of the Fourier transform, for all $\theta \in \R^d$, ${i,j \in \{1, \ldots, p \}}$,
\begin{align*}
    ^g \widehat{\mu}_{i,j} (\theta) &= \sum_{x \in \R^{d}}^{} {}^g \mu_{i,j}(x) e^{ \mathbf{i} x \cdot \theta}\\
    &= \sum_{x \in \R^{d}}^{} \mu_{i,j} \left( x + g_j -g_i \right) e^{ \mathbf{i} x \cdot \theta}\\
    &= \sum_{y \in \Z^{d}}^{} \mu_{i,j} \left(y \right) e^{ \mathbf{i} (y - g_j + g_i) \cdot \theta} \\
    &= \widehat{\mu}_{i,j} (\theta) e^{\mathbf{i} (g_i- g_j) \cdot \theta}.
\end{align*}
Therefore, $^g \widehat{\mu}_{i,j} (\theta)$ and $\widehat{\mu}_{i,j} (\theta)$ are similar: $^g \widehat{\mu}_{i,j} (\theta) = D(\theta) \widehat{\mu}_{i,j} D(\theta)^{-1}$ with ${D(\theta) = \diag{\left( e^{\mathbf{i} g_1 \cdot \theta}, \ldots ,  e^{\mathbf{i} g_p \cdot \theta}\right)}}$. Thus, they have the same eigenvalues.

\vspace{0.3cm}

We now prove the second point. By proving \ref{it:changement_preserve_vp_fourier}, we showed that for all $\theta \in \R^d$ and $i,j \in \{1, \ldots, p \}$, $^g \widehat{\mu}_{i,j}(\theta) = \widehat{\mu}_{i,j} (\theta) e^{\mathbf{i} (g_i- g_j) \cdot \theta}$. Taking $\theta = 0$, we get $^g \widehat{\mu}(0) = \widehat{\mu}(0)$. Since the transition matrix of the Markovian part is the Fourier transform at $0$, it means that the $g$-changed Markovian part has the same transition matrix as the original one, hence the same invariant distribution.

\vspace{0.3cm}

Since the $g$-change of section adds $g_i - g_j$ to the jumps from $i$ to $j$ but does not modify their probabilities, we get 
\begin{equation}\label{eq:drift_change_sect}
^g m_{i,j} = m_{i,j} + \widehat{\mu}_{i,j}(0) (g_i - g_j).
\end{equation}
The term $\widehat{\mu}_{i,j}(0)$ is the total probability to go from $i$ to $j$. It is here because $\mu_{i,j}$ are only \emph{sub}-probabilities.
Weighting \eqref{eq:drift_change_sect} by $^g \pi = \pi$ leads to 
\begin{align*}
    ^g m &= \sum_{i,j = 1}^p \pi_i \left( m_{i,j} + \widehat{\mu}_{i,j}(0) (g_i -g_j) \right)\\
    &= m + \sum_{ i=1}^{p} \pi_i g_i \underbrace{\sum_{ j=1}^p  \widehat{\mu}_{i,j}(0)}_{=1}  - \sum_{ j=1}^{p} g_j \underbrace{\sum_{ i=1}^p \pi_i \widehat{\mu}_{i,j}(0)}_{=\pi_j} \\
    &= m + \sum_{ i=1}^{p} \pi_i g_i - \sum_{j=1}^{p} \pi_j g_j \\
    &=m,
\end{align*}
hence \ref{it:inv_change_section_3}.

\vspace{0.3cm}

To prove the last point, we assume by contradiction that there is such a hyperplane~$H$ and take \sloppy ${\theta \in H^\perp \setminus \{0\}}$ such that $\theta \in (- \pi, \pi)^d \setminus \{0\}$. Then $s \in \mathrm{supp}\left( {}^g \mu\right) \mapsto \theta \cdot s$ is constant; we denote $\alpha$ its value. For all $i,j \in \{ 1, \ldots, p \}$, \[^g \widehat{\mu}_{i,j} (\theta) = e^{\mathbf{i} \alpha} \sum_{x \in \R^d} {}^g \mu_{i,j}(x).\] Summing over $j$, we get \[\sum_{j=1}^p {}^g \widehat{\mu}_{i,j}(\theta) = e^{\mathbf{i}\alpha}  \sum_{j=1}^p \sum_{x \in \R^d} \mu_{i,j}(x) = e^{\mathbf{i} \alpha},\] which means that the vector $\mathds{1}$ with all entries equal to $1$ is an eigenvector of $^{g} \widehat{\mu}_{i,j}(\theta)$ associated with the eigenvalue $e^{\mathbf{i} \alpha}$. Because of \ref{it:changement_preserve_vp_fourier}, this implies that $e^{\mathbf{i}\alpha}$ is also en eigenvalue of $\widehat{\mu}(\theta)$, which contradicts Lemma~\ref{lemme:spec_rad_aper}.\qedhere
\end{proof}

In \cite{MR0978023, MR0890364}, the process is called aperiodic if, for every change of section $g$, there exists no affine hyperplane $H$ such that $\mathrm{supp}({}^g \mu) \subset H$. Since we obtained this result in \ref{it:supp_pas_inclus_hyperplan}~of Proposition~\ref{prop:inv_change_sect}, our Assumptions \ref{assump:main_assumptions} are stronger, but do not require any knowledge on changes of sections to be made.

We now choose a specific change of section $g$ that will enable us to compute the Hessian of $k$. That is an adaptation in our lattice case of (1.16) of \cite{MR0978023}.

\begin{prop}\label{prop:specific_g}
    There exists a change of section $g$ such that for all $i \in \{1 , \ldots, p \}$, $\displaystyle \sum_{j=1}^p {}^g m_{i,j} = m$. We call those \emph{appropriate changes of section}. Besides, the appropriate $g$-changed process is unique.

\end{prop}
In other words, after an appropriate change of section, for any $i \in \{1, \ldots, p\}$, the drift of the sum of the jumps starting from $i$ is equal to the global drift, without needing to weight by $\pi$.

\begin{proof}
    In this proof, the local and global drifts, $m_{i,j}$ and $m$, are seen as row vectors of $\R^d$ and $\mathds 1$ denotes the column vector of $\R^p$ with all entries equal to $1$. Therefore, $\mathds 1 m$ is the $p \times d$ matrix whose all rows are equal to $m$.
    Denoting $\mathcal{M}$ the $p \times d$ matrix whose $i^{\mathrm{th}}$ row is $\sum_{j=1}^p m_{i,j}$ and $^g \mathcal{M}$ the same matrix for the $g$-changed process, we see that $g$ is appropriate if and only if $^g \mathcal{M} = \mathds 1 m$.
    
    Remember Equation \eqref{eq:drift_change_sect}: \[^g m_{i,j} = m_{i,j} + \widehat{\mu}_{i,j}(0) (g_i - g_j).\] Summing over $j$, it leads to  \[^g \mathcal{M} = \mathcal{M} + \left( \mathrm{I}_{p} - \widehat{\mu}(0) \right) g.\] Therefore, $g$ is appropriate if and only if 
    \begin{equation}\label{eq:appropriate_change_section}
    \mathds 1 m = \mathcal{M} + \left( \mathrm{I}_p - \widehat{\mu}(0)\right)g.
    \end{equation}
    Thus, in order to find an appropriate change of section $g$, we have to prove that all columns of $\mathcal{M} - \mathds 1 m$ belong to $\im(\mathrm{I}_p - \widehat{\mu}(0))$, i.e., to $\pi^\perp$ because of \eqref{eq:ker_im_perp_jordan}.
     Let $k \in \{1, \ldots, d \}$ and let $C_k$ be the the $k^{\mathrm{th}}$ column of $ \mathcal{M} - \mathds 1 m$.  We have \[\pi  C_k = \sum_{i,j =1}^p \pi_i \left( m_{i,j} \right)_k - \sum_{i=1}^p \pi_i (m)_k = (m)_k - (m)_k = 0,\] hence the existence of an appropriate change of section.

    For uniqueness, \eqref{eq:appropriate_change_section} shows that if $g_1$ and $g_2$ are appropriate changes of section, then \sloppy ${\left( \mathrm{I}_p - \widehat{\mu}(0) \right) \left( g_1 - g_2\right) = 0}$. Since $\ker \left( \mathrm{I}_p - \widehat{\mu}(0) \right) = \spn (\mathds 1) $, the columns of $g_1 - g_2$ are multiples of $\mathds 1$. Therefore, the difference between two rows of $g_1$ are equal to the difference between the same two rows of $g_2$. Since the Definition \ref{defi:change_section} of the $g$-changed process involves differences between rows of $g$, the $g_1$-changed process and the $g_2$-changed process are the same.
\end{proof}

From now on, $g$ will denote an appropriate change of section.

Following \cite[(1.19)]{MR0978023}, we define the energy matrix $\sigma$ that will turn out to be the Hessian of $k$ at $0$. It involves the second order moments $^g \Sigma_{i,j}$, introduced in Definition~\ref{def:moments}, of the $g$-changed process.

\begin{defi}\label{def:energy_matrix}
    The \emph{energy matrix} $\sigma$ is the matrix $\sum_{i,j = 1}^p \pi_i {}^g \Sigma_{i,j}$. More explicitly, 
    \[\forall k,l \in \{1, \ldots, d\},~~~~~\sigma_{k,l} =  \displaystyle \sum_{x\in \R ^d} x_k x_l \sum_{i,j = 1}^p  \pi_i {}^g \mu_{i,j}(x)  .\]
    
\end{defi}

The uniqueness of the appropriate $g$-changed process shows that $\sigma$ is independent from the appropriate change of section $g$. The following property of $\sigma$ is crucial.

\begin{prop}\label{prop:energy_mat_def_pos}
    The energy matrix $\sigma$ is symmetric, positive-definite.
\end{prop}
\begin{proof}
    By definition, $\sigma$ is symmetric. Besides, since all the $\sum_{j=1}^p {}^g \mu_{i,j}$ for $i \in \{1, \ldots, p \}$ are probability measures, $\nu := \displaystyle \sum_{i,j=1}^p \pi_i {}^g \mu_{i,j}$ is also a probability measure on $\R^d$, and the positivity of the entries of $\pi$ implies that $\supp(\nu) = \supp\left( ^g \mu \right)$. If a random vector $X$ has distribution $\nu$, then $\sigma = \left( \mathbb{E}\left[ X_k X_l \right] \right)_{1\leq k,l \leq d}$. Therefore, for any $y \in \R^d \setminus \{0\}$, \[y \cdot \sigma y = \sum_{k,l =1}^d y_k y_l \mathbb{E}\left[ X_k X_l \right] = \mathbb{E}\left[ \left(\sum_{k=1}^d y_k X_k\right)^2 \right] \geq 0.\] If $y \cdot \sigma y =0$, then $\supp(\nu) = \supp(^g \mu)$ is included in the linear hyperplane of equation~$\sum_{k=1}^d y_k x_k = 0$, which contradicts \ref{it:supp_pas_inclus_hyperplan} of Proposition~\ref{prop:inv_change_sect}.
\end{proof}

We are now able to compute the two first derivatives of $k$ at $0$. Our proof is different from the proof of Guivarc'h in \cite[Lemme 3]{MR0890364}, which Babillot cited in \cite{MR0978023}. It has the advantage of avoiding the use of some convergences that could be tricky to justify.

\begin{thm}\label{thm:diff_k}
The leading eigenvalue $k$ of $\widehat{\mu}$ is twice differentiable at $0$ and:
\begin{enumerate}
    \item its gradient is $\nabla k(0) = \mathbf{i} m$;
    \item its Hessian is $\mathbf{H}_k(0) = - \sigma$.
\end{enumerate}
\end{thm}

\begin{proof}
    In this proof, we identify the stationary distribution $\pi$ with a row vector of $\R^p$.
    
    The differentiability of $k$ at $0$ is a direct consequence of the implicit function theorem (it is even $C^\infty$).
    To compute its derivatives, we use the functions $p$ and $Q$ of Proposition~\ref{prop:jordan_chevalley}. A classical result about eigenprojections ensures that $p$ is smooth. The eigenprojection for a certain eigenvalue is even a (locally) analytic function, as stated in \cite[Chapter Two, \S 1, 4]{kato}.

    We fix a unit vector $e$ of $\R^d$ and define $\mathsf{M}(t) = \widehat{\mu}(t e)$, $\mathsf{k}(t) = k(te)$, $\mathsf{p}(t) = p ( t e)$ and $\mathsf{Q}(t) = Q( t e)$ for a real parameter $t$ around $0$, in order to work with functions of one variable.
    
\vspace{0.3cm}

    Let us start with the computation of the gradient.
     We have $\mathsf k'(0) = \nabla k (0) \cdot e$ and \[\mathsf M '(0) = \mathbf{i} \left( \displaystyle \sum_{x \in \Z^d} x \cdot e  \mu_{i,j} (x)  \right)_{1 \leq i,j \leq p} = \mathbf{i} \left( m_{i,j} \cdot e \right)_{1\leq i,j \leq p}.\]
    Since $\mathsf{M}(t) = \mathsf{k}(t) \mathsf{p}(t) + \mathsf{Q}(t)$ with $\mathsf{p} \mathsf{Q} = \mathsf{Q}\mathsf{p}=0$ and $\mathsf p ^2 = \mathsf p$ because of Proposition~\ref{prop:jordan_chevalley}, we obtain \[\mathsf p (t) \mathsf M(t) = \mathsf{M}(t) \mathsf{p}(t) = \mathsf{k}(t) \mathsf{p}(t).\]
    Differentiating the second equality leads to
    \begin{equation}\label{eq:deriv_jordan}
    \mathsf{M}'(t) \mathsf{p}(t)  = \mathsf{k}'(t) \mathsf{p}(t) + \mathsf{k}(t) \mathsf{p}'(t) - \mathsf{M}(t) \mathsf{p}'(t).
    \end{equation}
    Then we left-multiply by $\mathsf p(t)$ and get 
    \begin{equation}\label{eq:deriv_k}
    \mathsf p(t)\mathsf{M}'(t) \mathsf{p}(t)  = \mathsf k'(t) \mathsf p(t) + \mathsf{k}(t) \mathsf p(t)  \mathsf{p}'(t) - \mathsf p(t) \mathsf{M}(t) \mathsf{p}'(t) = \mathsf k'(t) \mathsf p(t)
    \end{equation}
    because $\mathsf p(t) \mathsf{M}(t) = \mathsf k (t) \mathsf p (t)$, so $ \mathsf p(t) \mathsf{M}(t) \mathsf{p}'(t) = \mathsf k(t) \mathsf p(t) \mathsf p'(t)$.
    Finally we right-multiply by~$\mathds 1$ and left-multiply by $\pi$ at $t = 0$ to obtain 
    \[  \pi  \left( \mathsf p(0) \mathsf M'(0) \mathsf p(0) \mathds 1 \right) = \mathsf k'(0) \pi  \mathsf p(0) \mathds 1 .\]
    Since $p(0) \mathds 1 = \mathds 1$ by \eqref{eq:im_p_0}, $\pi  \mathds 1 = 1$ as $\pi$ is a probability and $\pi p(0) = \pi$ by \eqref{eq:pi_p_0}, it simply becomes \[\mathsf k'(0) = \pi  \mathsf M'(0) \mathds 1. \] Therefore, $\nabla k(0) \cdot e = \mathsf k'(0) = \pi  \mathsf M '(0) \mathds{1} = \mathbf{i} \sum_{i,j=1}^p \pi_i m_{i,j}\cdot e = \mathbf{i} m\cdot e$. Since it is true for any unit vector $e$, $\nabla k(0) = \mathbf{i}m$.

\vspace{0.3cm}
    
    For the Hessian, the result involves the appropriate change of section $g$. We have $^g k=k$ and $^g \pi = \pi$ thanks to Proposition~\ref{prop:inv_change_sect}. We start from \eqref{eq:deriv_k}, which is still valid for the $g$-changed process:
    \[{}^g \mathsf p(t){}^g \mathsf{M}'(t) {}^g \mathsf{p}(t)  =  {}^g \mathsf{k}'(t) {}^g \mathsf{p}(t) .\]
    Differentiating at $t = 0$, right-multiplying by the vector $\mathds{1}$, left-multiplying by $^g \pi = \pi$ and simplifying using $^g p(0) \mathds 1 = p(0) \mathds 1 = \mathds 1$, $\pi  \mathds 1 = 1$ and $\pi ^g p(0) = \pi p(0) = \pi$, we get: 
    \begin{equation}\label{eq:second_diff}
        \pi  \left({}^g \mathsf p'(0) {}^g \mathsf M'(0) \mathds 1\right) + \pi  \left( {}^g \mathsf M''(0) \mathds 1 \right) + \pi  \left( {}^g \mathsf M'(0) {}^g \mathsf p'(0) \mathds 1\right) = {}^g \mathsf k''(0) + {}^g \mathsf k'(0) \pi  \left( {}^g \mathsf p'(0) \mathds 1 \right). 
    \end{equation}
    Since $^g \mathsf M '(0) =  \left( \mathbf{i} e \cdot {}^g m_{i,j} \right)_{1 \leq i,j \leq p}$, we have \[^g \mathsf M '(0) \mathds 1 = \mathbf{i}\left( e \cdot \displaystyle \sum_{j = 1}^p {}^g m_{i,j} \right)_{1 \leq j \leq p}.\] Since $g$ is an appropriate change of section, $\sum_{j = 1}^p {}^g m_{i,j} = m$ for every $i \in \{1, \ldots, p\}$, thus $^g \mathsf M '(0) \mathds 1 = (\mathbf{i} e \cdot m) \mathds 1 = {}^g \mathsf k '(0) \mathds 1$. Therefore, the first term of the left-hand side of \eqref{eq:second_diff} and the second term of the right-hand side are equal, which leads to 
    \begin{equation}\label{eq:second_diff_2}
        ^g \mathsf k''(0) = \pi  \left( {}^g \mathsf M''(0) \mathds 1 \right) + \pi  \left( {}^g \mathsf M'(0) {}^g \mathsf p'(0) \mathds 1\right).
    \end{equation}
    We now prove that $\pi  \left( {}^g \mathsf M'(0) {}^g \mathsf p'(0) \mathds 1\right) = 0$ in \eqref{eq:second_diff_2}.
    Differentiating \eqref{eq:deriv_jordan} at $t=0$, right-multiplying by $\mathds 1$, left-multiplying by $\pi$ and then simplifying using ${}^g \mathsf  p(0) \mathds 1 = \mathds 1$ and $\pi  \mathds 1 = 1$, we get:
    \begin{equation*}
          \pi  \left({}^g \mathsf  M ''(0) \mathds 1\right) + 2  \pi  \left({}^g \mathsf  M'(0) {}^g \mathsf  p'(0) \mathds 1 \right) + \pi  \left( ^g \mathsf M(0) {}^g \mathsf  p ''(0) \mathds 1 \right) = {}^g \mathsf  k''(0) + 2 \pi  \left({}^g \mathsf  k'(0) {}^g \mathsf  p'(0) \mathds 1\right) +  \pi  \left({}^g \mathsf  p''(0) \mathds 1\right).
    \end{equation*}
    Since $\pi ^g \mathsf M(0) = \pi \widehat{\mu}(0) = \pi$, we get $\pi  \left( {}^g \mathsf  M (0) {}^g \mathsf  p ''(0) \mathds 1 \right) = \pi  \left({}^g \mathsf  p''(0) \mathds 1\right)$, hence
    \begin{equation}\label{eq:second_diff_3}
         \pi  \left({}^g \mathsf  M ''(0) \mathds 1\right) + 2  \pi  \left({}^g \mathsf  M'(0) {}^g \mathsf  p'(0) \mathds 1 \right)  = {}^g \mathsf  k''(0) + 2 {}^g \mathsf  k'(0)  \pi \left( {}^g \mathsf  p'(0) \mathds 1\right) .
    \end{equation}
    By differentiating ${}^g \mathsf p (t) ^2 = {}^g \mathsf p(t)$ at $t = 0$ and left-multiplying by $\mathds 1$, we get ${}^g \mathsf p'(0) \mathds 1 + {}^g \mathsf p (0) {}^g \mathsf p'(0) \mathds 1 = {}^g \mathsf p'(0) \mathds 1$ (we used ${}^g \mathsf p (0) \mathds 1 = \mathds 1$), so ${}^g \mathsf p (0) {}^g \mathsf p'(0) \mathds 1 = 0$, meaning that ${ {}^g \mathsf p'(0) \mathds 1 \in \ker {}^g \mathsf p (0) = \ker p(0) = \pi^\perp}$ (the last equality comes from \eqref{eq:ker_im_perp_jordan}).  Therefore, \eqref{eq:second_diff_3} becomes 
    \begin{equation}\label{eq:second_diff_4}
        {}^g \mathsf  k''(0) = \pi  \left({}^g \mathsf  M ''(0) \mathds 1\right) + 2  \pi  \left({}^g \mathsf  M'(0) {}^g \mathsf  p'(0) \mathds 1 \right).
    \end{equation}
    Equalities \eqref{eq:second_diff_2} and \eqref{eq:second_diff_4} yield $ \pi  \left({}^g \mathsf  M'(0) {}^g \mathsf  p'(0) \mathds 1 \right) = 2 \pi \left({}^g \mathsf  M'(0) {}^g \mathsf  p'(0) \mathds 1 \right)$, hence $\pi  \left({}^g \mathsf  M'(0) {}^g \mathsf  p'(0) \mathds 1 \right) =0$, thus \eqref{eq:second_diff_2} leads to \[\mathsf k''(0) = {}^g \mathsf k ''(0) = \pi \left( {}^g \mathsf M''(0) \mathds 1 \right).\]
    By definition of $\mathsf k$ and $^g \mathsf M$, we have $\mathsf k''(0) = e \cdot \mathbf{H}_k(0) e$ and \[ ^g \mathsf M''(0) = - \left( \displaystyle \sum_{x \in \R^d} (x \cdot e)^2~  {}^g \mu_{i,j}(x) \right)_{1 \leq i,j \leq p},\] so \[\pi  \left( {}^g \mathsf M''(0) \mathds 1 \right) = - \displaystyle \sum_{i,j =1}^p \pi_i \displaystyle \sum_{x \in \R^d} (x \cdot e)^2~  {}^g \mu_{i,j}(x) = - e \cdot \sigma e.\] Thus \[e \cdot \mathbf{H}_k(0) e = - e \cdot \sigma e.\] Since these equalities are true for any $e$, we obtain \[\mathbf{H}_k(0) = - \sigma.\qedhere\]
\end{proof}

\subsection{Laplace transform and Doob transform} \label{sec:laplace_doob}

In this Section, we study the restriction to $\R^d$ of the Laplace transform $L\mu$ defined in \eqref{eq:def_laplace}. The goal of this study is to  define a \emph{Doob transform} of the process, that is, a modification of the process of a different nature from the change of section: in the change of section, we changed the jumps but kept the same probabilities while in the Doob transform, we keep the same jumps but change their probabilities by multiplying them by exponential weights, which comes down to using harmonic functions. Since we keep the same jumps, Assumptions \ref{assump:main_assumptions} will automatically stay valid for the Doob transform.

\begin{lemma}\label{lem:c_doob_transf}
    Let $c \in \R^d$. We assume that $1$ is an eigenvalue of the Laplace transform~$L\mu(c)$ with an eigenvector $\varphi_c$ whose all entries are positive. We define the measures~$\left(\mu_c\right)_{i,j}$  by \[\forall i,j \in \{1, \ldots, p \},~~~~\forall x \in \Z^d,~~~~\left(\mu_c\right)_{i,j} (x) = \frac{\left(\varphi_c\right)_j}{ \left(\varphi_c\right)_i}e^{c \cdot x} \mu_{i,j}(x).\] Then for all $i \in \{1, \ldots, p\}$, $\sum_{j=1}^p  {(\mu_c)}_{i,j}$ is a \emph{probability} measure on $\Z^d$.
\end{lemma}

\begin{proof}
    The positivity of $\varphi_c$ ensures that $\sum_{j=1}^p  {(\mu_c)}_{i,j}$ is non-negative, and
\begin{equation*}
    \sum_{x\in \Z^d} \sum_{j=1}^p \left(\mu_c\right)_{i,j}(x)  = \sum_{j=1}^p \sum_{x \in \Z^d} \frac{\left(\varphi_c\right)_j}{\left(\varphi_c\right)_i} e^{ c\cdot x } \mu_{i,j}(x) = \frac{1}{\left(\varphi_c \right)_i} \left(L\mu(c) \varphi_c\right)_i = \frac{1}{\left(\varphi_c\right)_i} \left(\varphi_c\right)_i = 1.\qedhere
\end{equation*}
\end{proof}

The assumption of irreducibility from Assumptions~\ref{assump:main_assumptions} ensures that the real Laplace transform is a non-negative irreducible matrix. If its spectral radius is $1$, then the existence of an eigenvector $\varphi_c$ as in the previous lemma is a consequence of the Perron-Frobenius theorem. Moreover, this eigenvector is unique up to multiplication by a constant. This motivates the following definitions.

\begin{defi}\label{def:level_set_laplace}
    Let $c \in \R^d$. The spectral radius of $L\mu(c)$ is denoted $\rho(c)$. We define, in $\R^d$, $\mathcal{C} := \rho^{-1}([0,1])$ and $\partial \mathcal{C} := \rho^{-1}(\{1\})$.
\end{defi}
Note that $0 \in \mathcal{C}$. Indeed, $L\mu(0)$ is the transition matrix of the Markovian part, thus its spectral radius is~$1$. 
The notation $\partial \mathcal{C}$ is consistent, as we will see in Proposition~\ref{prop:prop_doob_transform} that it is indeed the topological boundary of $\mathcal{C}$.

\begin{figure}[htbp]
    \centering
    \includegraphics[scale =0.7]{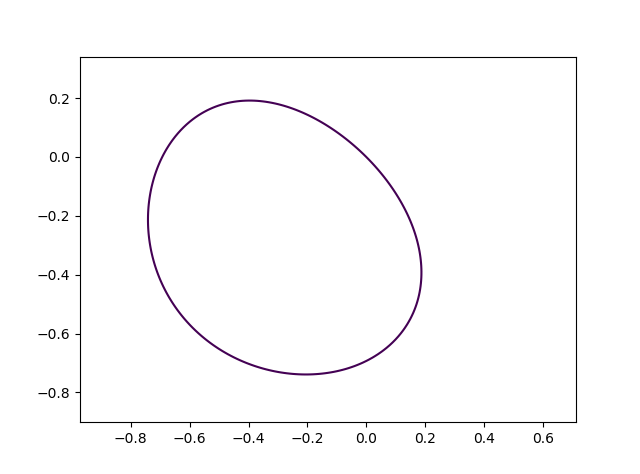}
    \caption{The set $\partial \mathcal{C}$ in a case where $d=p=2$.}
\end{figure}

\begin{defi}\label{def:doob_transform}
    Let $c \in \partial \mathcal{C}$ and $\varphi_c$ as in Lemma \ref{lem:c_doob_transf}. The \emph{Doob transform} of parameter $c$ of the process ${(Z_n)}_{n \geq 0}$ is the Markov-additive process whose jump matrix is $\mu_c$ defined by 
    \[\forall i,j \in \{1, \ldots, p \},~~~~\forall x \in \Z^d,~~~~\left(\mu_c\right)_{i,j} (x) = \frac{\left(\varphi_c\right)_j}{ \left(\varphi_c\right)_i}e^{c \cdot x} \mu_{i,j}(x).\] 
\end{defi}

Every quantity defined earlier also exists for the Doob transform, and will be denoted with an index $c$ (e.g. $\mu_c$, $m_c$, $\sigma_c$, $k_c$, $\rho_c$, $G_c$).

The interest of the Doob transform is that it removes the exponential decay from the Green function in Theorem \ref{thm:main}, as stated in the following proposition, which leaves us to prove a polynomial decay.

\begin{prop}\label{prop:link_green_functions}
    Let $c \in \partial \mathcal{C}$, $x \in \Z^d$ and $i,j \in \{1, \ldots, p\}$. The Green function of the Doob transform is given by 
    \[G_c \big( (0,i), (x,j) \big) = \frac{ \left(\varphi_c\right)_j}{\left(\varphi_c\right)_i} e^{c \cdot x} G\big( (0,i), (x,j) \big)\] where $\varphi_c$ is the Perron-Frobenius eigenvector introduced in Lemma~\ref{lem:c_doob_transf}.
\end{prop}

\begin{proof}Let $x \in \Z^d$, $i,j \in \{1, \ldots, p \}$ and $n\geq 0$.
    By definition, \[\left(\mu^{\star n}\right)_{i,j}(x) = \sum_{(0,i) = (x_0, i_0) \to (x_1,i_1) \to \ldots \to (x_n,i_n) = (x,j)} \mu_{i_0,i_1}(x_1) \mu_{i_1, i_2}(x_2-x_1) \ldots \mu_{i_{n-1},i_n}(x_n-x_{n-1})\] where the sum is taken on all paths from $(0,i)$ to $(x,j)$ of length $n$ in the graph associated with the Markov chain. Similarly, 
    \begin{align*}
     \left(\mu_c^{\star n}\right)_{i,j}(x) &=  \sum_{(0,i) =(x_0, i_0) \to \ldots \to (x_n,i_n) = (x,j)} \frac{\left(\varphi_c\right)_{i_1}}{\left(\varphi_c\right)_{i_0}} e^{ c \cdot x_1 }\mu_{i_0,i_1}(x_1)   \frac{\left(\varphi_c\right)_{i_2}}{\left(\varphi_c\right)_{i_1}} e^{ c \cdot (x_2 - x_1) }\mu_{i_1, i_2}(x_2-x_1)  \ldots \\
     & \frac{\left(\varphi_c\right)_{i_n}}{\left(\varphi_c\right)_{i_{ n-1 }}} e^{ c \cdot (x_n - x_{n-1}) } \mu_{i_{n-1},i_n}(x_n-x_{n-1}) \\
     &= \sum_{(0,i) =(x_0, i_0) \to \ldots \to (x_n,i_n) = (x,j)} \frac{\left(\varphi_c\right)_{j}}{\left(\varphi_c\right)_{i}} e^{ c \cdot x }  \mu_{i_0,i_1}(x_1) \mu_{i_1, i_2}(x_2-x_1) \ldots \mu_{i_{n-1},i_n}(x_n-x_{n-1})\\
     &= \frac{\left(\varphi_c\right)_{j}}{\left(\varphi_c\right)_{i}} e^{ c \cdot x } \left(\mu^{\star n}\right)_{i,j}(x).
    \end{align*}
    Summing over $n \geq 0$, we obtain the result of Proposition~\ref{prop:link_green_functions}.
\end{proof}

In order to work with the Doob transform, we study the set $\mathcal{C}$, or equivalently the function $\rho$. The following properties of $\rho$ are stated in \cite[(1.25)]{MR0978023}. In our lattice case, we can provide simpler proofs.

\begin{prop}\label{prop:usual_prop_spectral_radius}
The function $\rho$ is:
    \begin{enumerate}
        \item convex (and even logarithmically convex);
        \item\label{it:prop:norm_coercivity} norm-coercive, i.e., $\rho(c) \xrightarrow[\|c\| \to + \infty]{}+\infty$;
        \item\label{it:derivee_rayon_spectral_laplace} twice differentiable at $0$, with $\nabla \rho(0) = m$ and $\mathbf{H}_{\rho}(0) = \sigma$.
    \end{enumerate}
\end{prop}

\begin{proof}
    The proof of the convexity relies on the result of \cite{Ki-61}. If $\mu_{i,j}(x) >0$, then $c \mapsto e^{c \cdot x} \mu_{i,j}(x)$ is logarithmically convex. Therefore, the entries of $L\mu$ are $0$ if $\mu_{i,j} = 0$, or logarithmically convex functions, as sums of logarithmically convex functions. Besides, by irreducibility and aperiodicity, for $n$ large enough, all entries of $\left(L\mu\right) ^n$ are positive.
        If all entries are logarithmically convex, then the result of \cite{Ki-61} ensures that $\rho$ is logarithmically convex. To deal with the $0$ entries, we take a closer look at the proof of \cite{Ki-61} and notice that if $\Tr \left( (L\mu)^n \right) > 0$ for $n$ large enough, which is the case, then the result holds.
        
        \vspace{0.3cm}

         Let us move to the proof of norm-coercivity. When a function is convex, its norm-coercivity is equivalent to coercivity in every direction.  Hence we only need to prove that for every $u \in \mathbb{S}^{d-1}$, $\rho(tu) \xrightarrow[t \to +\infty]{}+\infty$. 
         According to the Perron-Frobenius theorem, see for example \cite[Corollary 1, page 8]{Se-06}, for every $c \in \R^d$ and $n \in \N$,
        \begin{equation}
            \rho(c)^n \geq \min_{i \in \{1, \ldots, p \}} \sum_{j=1}^p \left(L\mu (c)^n \right)_{i,j},
    \end{equation}
        so it is enough to prove that for every $u \in  \mathbb{S}^{d-1}$, there exists $n \in \N$ such that  for every $i,j \in \{1, \ldots, p \}$, $\left(L\mu (tu)^n \right)_{i,j} \xrightarrow[t \to + \infty]{}+\infty$.
        Let $u \in \mathbb{S}^{d-1}$ and $i,j \in \{1, \ldots,p\}$. Let $x \in \Z^d$ such that $x \cdot u >0$. Then as a consequence of aperiodicity and irreducibility from Assumptions \ref{assump:main_assumptions}, there exists $n_{i,j}\in \N$ such that for every $n \geq n_{i,j}$, $\left(\mu^{\star n}\right)_{i,j}(x)>0$. Let $n = \displaystyle \max_{i,j \in \{1, \ldots, p \}} n_{i,j}$. Then \[\left(L\mu(tu)^n\right)_{i,j} = \left(L\left( \mu^{\star n}\right)\right)_{i,j}(tu) \geq e^{t x \cdot u} \left(\mu^{\star n}\right)_{i,j}(x) \xrightarrow[t \to +\infty]{}+\infty,\] which proves the norm-coercivity.

        \vspace{0.3cm}

        Finally, using the analytic implicit function theorem, we can show that the leading eigenvalue of the complex Laplace transform is an analytic function around $0$. The formula from the derivatives then comes from the derivatives of the function $k$ in Theorem~\ref{thm:diff_k} and from \eqref{eq:link_deriv_fourier_laplace} which links the derivatives of the Laplace transform with the derivatives of the Fourier transform.
\end{proof}

As a corollary of the properties of $\rho$, we get the following properties of $\mathcal{C}$.

\begin{prop}\label{prop:ensemble_C}
The set $\mathcal{C}$ is compact and convex. Besides, under the non-centering Assumption~\ref{assump:non_centered}, the interior of $\mathcal{C}$ is non-empty.
\end{prop}

\begin{proof}
    The set $\mathcal{C} = \rho^{-1}([0,1])$ is a closed subset of $\R^d$ because $\rho$ is continuous, and bounded because $\rho$ is norm-coercive; thus it is a compact subset of $\R^d$. It is convex because $\rho$ is convex.
    
    The set $\rho^{-1}((-\infty, 1))$ is an open subset of $\R^d$ because $\rho$ is continuous. Besides, $\rho(0) = 1$ and, since $\nabla \rho(0) = m \neq 0$ under the non-centering assumption, a Taylor expansion leads to $\rho\left(-t \nabla\rho(0)\right) < \rho(0) = 1$ for $t>0$ small enough, so $\rho^{-1}((-\infty, 1))$ is non-empty. Therefore, $\mathcal{C}$ contains a non-empty open subset of $\R^d$.
\end{proof}

The study of $\rho$ tells us more about the Doob transform of the process, as seen in the following proposition. We remind that an index $c$ denotes an object associated with the Doob transform of parameter $c$, for example $m_c$ is the drift of the Doob transform.

\begin{prop}\label{prop:prop_doob_transform}
    Let $c \in  \partial \mathcal{C}$. Then:
    \begin{enumerate}
        \item\label{it:doob_transf_1} $m_c = \nabla \rho (c)$ and $\sigma_c = \mathbf{H}_{\rho}(c)$;
        \item $\sigma_c$ is positive-definite;
        \item \label{it:non_centered_doob} under the non-centering Assumption~\ref{assump:non_centered}, $m_c \neq 0$;
        \item $\partial \mathcal{C}$ is indeed the topological boundary of $ \mathcal{C}$.
    \end{enumerate}
\end{prop}

    \begin{proof}
    As mentioned in the introduction of Section \ref{sec:laplace_doob}, Assumptions~\ref{assump:main_assumptions} still hold for the Doob transform, which allows us to apply previous results.

        The result \ref{it:derivee_rayon_spectral_laplace} of Proposition~\ref{prop:usual_prop_spectral_radius} applied to the Doob transform immediately leads to the differentiability of $\rho_c$ around $0$, with $m_c = \nabla \rho_c(0)$. Besides, we have $L\mu_c = D^{-1} L\mu(c + \cdot) D$ where $D = \diag{\left( \left(\varphi_c\right)_1, \ldots, \left(\varphi_c\right)_p\right)}$, so $\rho_c = \rho(c+\cdot)$, therefore $\nabla \rho_c(0) = \nabla \rho (c)$, hence $m_c = \nabla\rho(c)$.
        Similarly, $\sigma_c = \mathbf{H}_{\rho_c}(0) = \mathbf{H}_\rho(c)$.

        The positive-definiteness of the energy matrix in Proposition~\ref{prop:energy_mat_def_pos} still holds for the Doob transform, as it satisfies Assumptions \ref{assump:main_assumptions}, so $\sigma_c$ is positive-definite.

        Let us move to the proof of \ref{it:non_centered_doob}. We remind that $\sigma_c = \mathbf{H}_\rho(c)$ is positive-definite and $m_c = \nabla \rho(c)$. Therefore, $m_c = 0$ would imply that $\rho$ has a local strict minimum at $c$, which would actually be global because of the convexity of $\rho$. But $\rho(c) = \rho(0) = 1$, so we would have $c = 0$, hence $m_c = m = 0$, which contradicts the non-centering Assumption \ref{assump:non_centered}.

         For the last point, since $\rho$ is continuous,  $\partial \left(\rho^{-1} ((-\infty, 1])\right) \subset \rho^{-1} (\{1\})$.
            Conversely, if $c \in \rho^{-1}(\{1\})$, a Taylor expansion of order $1$ in the direction of $\nabla \rho(c)$ (if $m_c = \nabla \rho(c) \neq 0$) or of order $2$ in any direction (if $m_c = \nabla \rho(c) = 0$, which cannot happen in the non-centered case) shows that there are $\tilde c$ arbitrarily close to $c$ such that $\rho(\tilde c) > 1$, therefore $c$ is not in the interior of $\rho^{-1}((-\infty, 1 ])$, and thus $c \in \partial \left( \rho^{-1}((-\infty, 1 ]) \right)$. 
    \end{proof}

    We are now able to adapt Proposition 4.4 of \cite{He-63}, which is the most important result of this section.
    \begin{thm}\label{thm:hennequin} 
   Under the non-centering assumption~\ref{assump:non_centered}, the function \[ \Gamma: \begin{array}[t]{ccl}
         \partial \mathcal{C}& \longrightarrow & \mathbb{S}^{d-1}  \\
         c& \longmapsto & \frac{1}{\|m_c\|}m_c = \frac{1}{\|\nabla \rho(c)\|} \nabla \rho(c)
    \end{array}\] is a homeomorphism between $\partial \mathcal{C}$ and the sphere $\mathbb{S}^{d-1}$ of $\R^d$. 
\end{thm}
In other words, the Doob transform allows the drift to take any direction, in a bicontinuous way.

\begin{proof}
    The function $\Gamma$ is well defined because $m_c \neq 0$ for all $c \in \partial \mathcal{C}$ according to \ref{it:non_centered_doob} of Proposition \ref{prop:prop_doob_transform}. 

    We first prove the injectivity. Assume by contradiction that there are two distinct points $c_1, c_2 \in \partial \mathcal{C}$ such that $\Gamma(c_1) = \Gamma(c_2)$, i.e., $\nabla\rho(c_1)$ and $\nabla \rho(c_2)$ have the same direction. Since $\nabla\rho(c_1)$ and $\nabla \rho(c_2)$ have the same direction, either $(c_2-c_1) \cdot \nabla \rho(c_1) \geq 0$ or $(c_1-c_2) \cdot \nabla \rho(c_2) \geq 0$.
    Without loss of generality, we assume that
    \begin{equation}\label{eq:injectivite_hennequin}
    (c_2-c_1) \cdot \nabla \rho(c_1) \geq 0.
    \end{equation}
    Since $\rho$ is smooth and its Hessian at $c_1$ is positive-definite, $\rho$ is locally strictly convex in a convex neighborhood $V$ of $c_1$. If we take $c \in (c_1, c_2] \cap V$, then the local strict convexity of $\rho$ and \eqref{eq:injectivite_hennequin} imply that \[\rho(c) > \rho(c_1) +   (c-c_1) \cdot \nabla \rho(c_1) \geq \rho(c_1) = 1,\]  hence $c \notin \mathcal{C}$, which contradicts the convexity of $\mathcal{C}$. Thus, $\Gamma$ is injective.

    We now prove the surjectivity, following Hennequin in \cite{He-63}. Let $M \in \mathbb{S}^{d-1}$ and $p$ be the orthogonal projection on the line $\spn(M)$. Then $p(\mathcal{C})$ is a convex compact subset of the line, i.e., a line segment $[A,B]$, with $A\neq B$ because $\mathcal{C}$ has a non-empty interior. Let $c_1, c_2 \in \partial \mathcal{C}$ such that $p(c_1)= A$ and $p(c_2)=B$. Then we claim that $\nabla \rho(c_1)$ or $\nabla \rho(c_2)$ has the same direction as $M$. Indeed, if we fix $U \in (A,B)$ (we can take $U = 0$ like in Figure \ref{fig:proof_hennequin.}, except if $A=0$ or $B=0$), the function $f : x \mapsto \| p(x) - U \|^2$ restricted to the set $\partial \mathcal{C} = \rho^{-1}(\{1\})$ has a local maximum at $c_1$. Since $\nabla \rho(c_1) = m_{c_1} \neq 0$, the Lagrange multiplier theorem asserts that $\nabla \rho (c_1)$ and $\nabla f (c_1)$ are collinear. But an easy computation leads to $\nabla f(c_1) = 2 (A-U) \in \spn(M) \setminus\{0\}$, hence $\nabla \rho (c_1)$ and $M$ are collinear. Similarly, $\nabla \rho (c_2)$ and $M$ are collinear. Since the vectors $\nabla \rho(c_1)$ and $\nabla \rho(c_2)$ are both collinear with $M$ but cannot have the same direction because $\Gamma$ is injective, one of them has the same direction as $M$ and the other one has the opposite direction, let us say $\nabla \rho(c_1)$ has the same direction as $M$. Then $\Gamma(c_1) = M$, and $ \Gamma$ is surjective.

    Being a continuous bijection between compact sets, $\Gamma$ is a homeomorphism.
    \end{proof}

    \begin{figure}[htbp]
        \centering
\begin{tikzpicture}[line cap=round,line join=round,>=triangle 45,x=1.0cm,y=1.0cm]
\clip(-4.,-1.5) rectangle (3.4,2.5);
\draw [line width=0.8pt] (0.,0.) circle (1.cm);
\draw [rotate around={165.3527408595892:(-0.45303300343394715,1.1084935778939065)},line width=0.8pt] (-0.45303300343394715,1.1084935778939065) ellipse (2.3650046313211237cm and 1.005449736509424cm);
\draw [line width=0.8pt,domain=-4.:3.4] plot(\x,{(-0.--0.5472871900715321*\x)/0.836944879656723});
\draw [line width=0.4pt,dotted,domain=-4.:3.4] plot(\x,{(--1.526332228747553--0.836944879656723*\x)/-0.5472871900715321});
\draw [line width=0.4pt,dotted,domain=-4.:3.4] plot(\x,{(-1.9821382992377856--0.836944879656723*\x)/-0.5472871900715321});
\draw (-3.5560807644691983,0.7121001170975991) node[anchor=north west] {$\nabla \rho(c_1)$};
\draw (2,1.3173914790646584) node[anchor=north west] {$\nabla \rho(c_2)$};
\draw [->,line width=0.8pt] (-2.639902091464356,1.2481934930896659) -- (-3.7821472980676956,0.5012671605562483);
\draw [->,line width=0.8pt] (1.7299538798962795,0.9762009181075175) -- (3.0406870981775787,1.8333033598606854);
\begin{scriptsize}
\draw [fill=black] (0.,0.) circle (2.0pt);
\draw[color=black] (0.026853587819765937,-0.21048107170703145) node {$0$};
\draw[color=black] (-0.2,0.7950836102705024) node {$\mathbb{S}^{d-1}$};
\draw[color=black] (1.3253011868781643,0.3167081790384717) node {$\mathcal{C}$};
\draw [fill=black] (-0.836944879656723,-0.5472871900715321) circle (2.5pt);
\draw[color=black] (-0.6272515936607806,-0.6205171556202005) node {$M$};
\draw [fill=black] (-1.2774559435052986,-0.8353420765868673) circle (2.5pt);
\draw[color=black] (-1.2813567751413273,.-0.610754391717506) node {$A$};
\draw [fill=black] (1.6589405003185504,1.0847989001230134) circle (2.5pt);
\draw[color=black] (1.696286215180564,1.292984569307922) node {$B$};
\draw [fill=black] (-2.639902091464356,1.2481934930896659) circle (2.0pt);
\draw[color=black] (-2.545634700540294,1.4345446458969924) node {$c_1$};
\draw [fill=black] (1.7299538798962795,0.9762009181075175) circle (2.5pt);
\draw[color=black] (1.6133027220076586,0.7902022283191551) node {$c_2$};
\end{scriptsize}
\end{tikzpicture}
        \caption{Illustration of the proof of the surjectivity in Theorem~\ref{thm:hennequin}}
        \label{fig:proof_hennequin.}
    \end{figure}
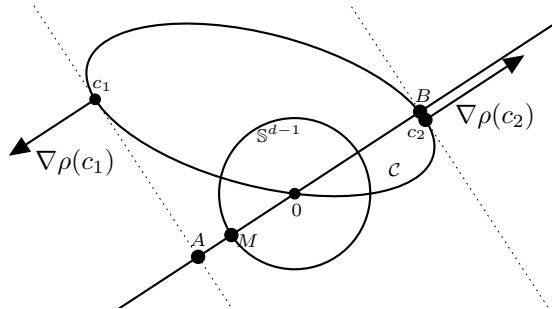

For a direction $u \in \mathbb{S}^{d-1}$, we denote $c(u) = \Gamma^{-1}(u)$, that is, the only $c \in \partial \mathcal{C}$ such that $m_c$ is a positive multiple of $u$. This function $c$ is the one that appears in the exponential decay of the main Theorem \ref{thm:main}. Proposition \ref{prop:vraie_decroissance_expo} ensures that, except in the direction of the drift, the decay of \eqref{eq:main} is actually exponential.

\begin{prop}\label{prop:vraie_decroissance_expo}
    For every $u \in \mathbb{S}^{d-1} \setminus \left\{ \frac{m}{\|m\|} \right\}$, $c(u) \cdot u > 0$.
\end{prop}

\begin{proof}
    Using the bijection $\Gamma$, it comes down to proving that for every $c \in \partial \mathcal{C} \setminus \{0\}$, $c \cdot \nabla \rho(c) >0$.
    By convexity of $\rho$, for every $c \in \partial \mathcal{C}$, $\mathcal{C}$ is included in the affine half-space~${H_c = \{x \in \R^d~|~(x-c) \cdot \nabla \rho(c) \leq 0\}}$. Besides, since $\sigma_c$ is positive-definite, $\partial H_c \cap  \mathcal{C} = \{c\}$. Therefore, if $c \neq 0$, then $0 \in H_c \setminus \partial H_c$, i.e., $c \cdot \nabla \rho(c) >0$.
\end{proof}

\section{Integral formula and proof of the main theorem}\label{sec:proof_main}

    From now on, in addition to Assumptions \ref{assump:main_assumptions}, we will always make Assumption \ref{assump:non_centered} of non-centering.
\subsection{Integral formula for the Green function}\label{subsec:integral_formula}
We are now able to prove the integral formula hinted in \eqref{eq:integral_formula}. We will actually need a more general formula that works for the Doob transform, but one can recover \eqref{eq:integral_formula} by taking $c=0$.
\begin{thm}\label{th:form_int}
    We assume~\ref{assump:main_assumptions} and~\ref{assump:non_centered}.
    Let $c \in \partial \mathcal{C}$, $x \in \Z^d$ and $i,j \in \{1, \ldots, p\}$.
    The Green function of the Doob transform has the following integral formula: 
    \begin{equation}\label{eq:form_int_doob}
     G_c\big( (0,i), (x,j) \big) = \frac{1}{(2 \pi)^d} \int_{[-\pi, \pi]^d} \left(\left( \mathrm{I}_p - \widehat{\mu}_c(\theta) \right)^{-1}\right)_{i,j} e^{- \mathbf{i}x \cdot \theta} \mathrm{d}\theta. 
     \end{equation}
     This integral is finite, so the process is transient.
\end{thm}

\begin{proof}
    The Doob transform satisfies the same assumptions as the original process: it is automatic for Assumptions~\ref{assump:main_assumptions} and \ref{it:non_centered_doob}~of Proposition~\ref{prop:prop_doob_transform} ensures that it satisfies Assumption~\ref{assump:non_centered}. Therefore, we only need to prove the formula for a non-transformed process that satisfies those assumptions.

    By \eqref{eq:loi_conv} and the definition of $G$, 
        \[G\big( (0,i), (x,j) \big)  = \sum_{n=0}^{+\infty} \left(\mu^{\star n}\right)_{i,j} (x).\]
    Using $\widehat{\mu^{\star n}} = \widehat{\mu}^n$ (see \eqref{eq:fourier_convol_link}) and inverse Fourier, we obtain \[\left(\mu^{\star n}\right)_{i,j} (x) = \frac{1}{(2 \pi)^d} \int_{\left[ - \pi , \pi \right]^d} \left(\widehat{\mu} ^n(\theta)\right)_{i,j} e^{- \mathbf{i} x \cdot \theta} \mathrm{d}\theta.\]
    Therefore, 
    \begin{equation}\label{eq:integral_formula_finite_sum}
    \sum_{n=0}^N (\mu^{\star n})_{i,j} (x) = \frac{1}{(2\pi)^d} \int_{\left[- \pi,\pi\right]^d} \sum_{n=0}^N (\widehat{\mu}^n (\theta))_{i,j} e^{- \mathbf{i} x \cdot \theta} \mathrm{d}\theta.
    \end{equation}
    
    The limit of the left-hand side of \eqref{eq:integral_formula_finite_sum} when $N$ goes to $+\infty$ is $G\big( (0,i), (x,j) \big)$. We prove that the right-hand side converges to \[\frac{1}{(2 \pi)^d} \displaystyle\int_{[-\pi, \pi]^d} \left(\left( \mathrm{I}_p - \widehat{\mu}(\theta) \right)^{-1}\right)_{i,j} e^{- \mathbf{i}x \cdot \theta} \mathrm{d}\theta\] using the dominated convergence theorem.
    For $\theta \in \left[- \pi,  \pi\right]^d \setminus \{0\}$, the spectral radius of $\widehat{\mu}(\theta)$ is smaller than $1$ (Lemma~\ref{lemme:spec_rad_aper}), so $\mathrm{I}_p - \widehat{\mu}(\theta)$ is invertible  and $\widehat{\mu}(\theta)^{N+1} \xrightarrow[N \to +\infty]{} 0$. Thus, we get 
    \begin{equation}\label{eq:formule_int_cv_dom}
    \displaystyle\sum_{n=0}^N \widehat{\mu}^n(\theta) = \left( \mathrm{I}_p - \widehat{\mu}(\theta) \right)^{-1} \left( \mathrm{I}_p - \widehat{\mu}(\theta)^{N+1} \right) \xrightarrow[N \to +\infty]{} \left( \mathrm{I}_p - \widehat{\mu}(\theta) \right)^{-1},
    \end{equation}
    hence the pointwise convergence of the integrand of \eqref{eq:integral_formula_finite_sum}.
    Now we need to dominate 
    \begin{equation}\label{eq:form_int_somme_geom}
    \displaystyle\sum_{n=0}^N \widehat{\mu}^n(\theta) = \left( \mathrm{I}_p - \widehat{\mu}(\theta) \right)^{-1} \left( \mathrm{I}_p - \widehat{\mu}(\theta)^{N+1} \right).
    \end{equation}
    The operator norm $\|\cdot \|$ on $\mathcal{M}_p(\mathbb C)$ associated with the norm $\| \cdot \|_\infty$ on $\mathbb C^p$ is the maximum $\ell_1$ norm of the rows, so $\left\| \widehat{\mu}(\theta) \right\| \leq 1$ for all $\theta \in \R^d$. Therefore, \[\left\| \mathrm{I}_p - \widehat{\mu}(\theta)^{N+1} \right\| \leq 1 + \left\| \widehat{\mu}(\theta)^{N+1} \right\| \leq 1 + \left\| \widehat{\mu}(\theta) \right\|^{N+1} \leq 2,\]
    so \eqref{eq:form_int_somme_geom} leads to \[\left\|\displaystyle\sum_{n=0}^N \widehat{\mu}^n(\theta) \right\| \leq 2 \left\| \left( \mathrm{I}_p - \widehat{\mu}(\theta) \right)^{-1} \right\|.\]
    Thus, in order to dominate $\displaystyle\sum_{n=0}^N \widehat{\mu}^n(\theta)$, we only need to prove the integrability of $\left\| \left( \mathrm{I}_p - \widehat{\mu}(\theta) \right)^{-1} \right\| $. Since it is a continuous function except at $0$ and we integrate on a compact set, we only need to prove its integrability on a neighborhood of $0$. Using the decomposition $\widehat{\mu}(\theta) = k(\theta) p(\theta) + Q(\theta)$ of Proposition~\ref{prop:jordan_chevalley}, which holds in a neighborhood of $0$, the equalities $p(\theta)  Q(\theta) = Q(\theta) p(\theta) = 0$ lead to $\widehat{\mu}(\theta)^n = k(\theta)^n p(\theta) + Q(\theta)^n$, hence 
    \begin{equation}\label{eq:inv_jordan_chevalley}
    \left( \mathrm{I}_p - \widehat{\mu}(\theta) \right)^{-1}  = \sum_{n=0}^{+ \infty} \widehat{\mu}(\theta)^n = \frac{1}{1-k(\theta)} p(\theta) + \sum_{n=0}^{+ \infty} Q(\theta)^n
    \end{equation}
    (the geometric series are valid for $\theta \neq 0$, once again thanks to Lemma~\ref{lemme:spec_rad_aper}). Notice that, since the spectral radius of $Q(0)$ is smaller than $|k(0)| = 1$, there exists an operator norm~$\| \cdot \|_\mathrm{op}$ such that $\|Q(0)\|_\mathrm{op} < 1$. By continuity of $Q$, there exists a neighborhood $U$ of $0$ such that $ \alpha := \sup \left\{ \|Q(\theta)\|_\mathrm{op}~|~\theta \in U \right\} < 1$, and we obtain: \[\forall \theta \in U,~~~~\left \| \sum_{n=0}^{+ \infty} Q(\theta)^n \right\|_\mathrm{op} \leq \sum_{n=0}^{+ \infty} \left \|Q(\theta)\right\|_\mathrm{op}^n \leq \frac{1}{1-\alpha}.\] Moreover, $p$ is a continuous function, so it is bounded around $0$. Therefore, going back to Equation \eqref{eq:inv_jordan_chevalley}, we see that it is sufficient to prove the integrability of $\frac{1}{1-k}$ on a neighborhood of $0$. 
    
    To do so, we use the asymptotics of $k$ coming from Theorem~\ref{thm:diff_k}: \[k(\theta) = 1+ \mathbf{i}  m \cdot \theta  -\frac{1}{2}  \theta \cdot \sigma \theta  + o\left(\|\theta\|^2\right).\] It leads to \[|1 - k(\theta)| = \left| -\mathbf{i}    m \cdot \theta  + \frac{1}{2}  \theta \cdot \sigma \theta  + o\left( \|\theta\|^2 \right) \right| \geq \left| -\mathbf{i}   m \cdot \theta  + \frac{1}{2}  \theta \cdot \sigma \theta  \right| -  o\left( \|\theta\|^2 \right)\] and therefore, for $\theta$ small enough, \[|1 - k(\theta)| \geq \frac{1}{2} \left| -\mathbf{i}   m \cdot \theta  + \frac{1}{2}  \theta \cdot \sigma \theta  \right| = \frac{1}{2} \sqrt{\frac{1}{4}  (\theta \cdot \sigma \theta) ^2 + ( m\cdot \theta) ^2} .\]
     Using Proposition~\ref{prop:energy_mat_def_pos}, we obtain that $\theta \mapsto \sqrt{ \theta \cdot \sigma \theta }$ is a norm. By equivalence of norms in finite dimension, there exists a constant $c>0$ such that $\frac{1}{4}  (\theta \cdot\sigma \theta)^2  \geq c \|\theta\|^4$ for all $\theta \in \R^d$. Besides, writing $\theta = \theta_1 e_1 + \ldots + \theta_d e_d$ where $(e_1, \ldots, e_d)$ is an orthonormal basis such that $e_1 = \frac{1}{\|m\|}m$, we get $ m \cdot \theta = \|m\| \theta_1$. It leads to \[ |1 - k(\theta)| \geq \frac{1}{2}\sqrt{c \|\theta\|^4 + \|m\|^2 \theta_1^2} \geq \Tilde{c} \sqrt{\| \theta \|^4 + \theta_1^2},\] where $\Tilde{c} = \frac{\min (\sqrt{c}, \|m\|)}{2}$. Applying Young's inequality, namely $ab \leq \frac{a^p}{p} + \frac{b ^q}{q}$ where $a,b \geq 0$ and $\frac{1}{p} + \frac{1}{q} = 1$, to $a = \|\theta\|^{4/3}$, $b = |\theta_1|^{4/3}$, $p = 3$ and $q = \frac{3}{2}$, we get $ \| \theta \|^{4/3} |\theta_1|^{4/3} \leq \frac{\|\theta\|^4}{3} + \frac{2}{3} \theta_1^2 \leq \frac{2}{3} \left( \|\theta\|^4 + \theta_1^2 \right)$, i.e., $\sqrt{\|\theta\|^4 + \theta_1^2} \geq \sqrt{\frac{3}{2}} \|\theta\|^{2/3} |\theta_1|^{2/3}$. 
    Besides, writing $\theta' = \theta - \theta_1 e_1$, we have $\|\theta'\| \leq \|\theta\|$, which finally leads to \[ |1-k(\theta)| \geq c' \|\theta'\|^{2/3} |\theta_1|^{2/3}\] with $c' = \sqrt{\frac{3}{2}}\Tilde{c}$, hence \[ \frac{1}{|1-k(\theta)|} \leq \frac{1}{c' \|\theta'\|^{2/3} |\theta_1|^{2/3}} .\] The function $\theta \mapsto \frac{1}{\|\theta'\|^{2/3} |\theta_1|^{2/3}}$ is integrable on a neighborhood of $0$ by Tonelli's theorem, because $\theta_1 \mapsto \frac{1}{|\theta_1|^{2/3}}$ and $\theta' \mapsto \frac{1}{\|\theta'\|^{2/3}}$ both are, since in dimension $n$, $x \mapsto \frac{1}{\|x\|^\alpha}$ is integrable on a neighborhood of $0$ if and only if $\alpha < n$. Therefore, $\frac{1}{|1 - k(\theta)|}$ is integrable on a neighborhood of $0$. This allows us to use the dominated convergence theorem to take the limit of \eqref{eq:integral_formula_finite_sum} and concludes the proof of the integral formula.

    The domination we made in the proof shows that the Green function is finite, hence the transience of the process.
    \end{proof}

\subsection{Dominating part of the integral}\label{sec:dominating_part_without_proofs}

In this section, we summarize the main stages in the identification of the dominating part of the integral formula from Theorem~\ref{th:form_int}. The goal is to obtain 
\begin{equation}\label{eq:dom_part}
    G_c\big( (0,i), (x,j) \big) = \left(\frac{1}{(2 \pi)^d} \int_{\R^d} \frac{f(\theta)}{1-k_c( \theta)} e^{-\mathbf{i} x \cdot \theta} \mathrm{d} \theta\right) {p_c(0)}_{i,j} + \underset{\|x\| \to + \infty}{o} \left( \|x\|^{- \frac{d-1}{2}} \right).
\end{equation}
In \eqref{eq:dom_part}, the little-$o$ is uniform in $c \in \partial \mathcal{C}$, $k_c$ and $p_c$ come from Proposition~\ref{prop:jordan_chevalley} and $f$ is any smooth, compactly supported (with small enough support), radial function, locally constant to $1$ on a neighborhood of $0$. 

The technical details and proofs are postponed to Appendix~\ref{sec:dominating_part}. They come from a mix between Woess's and Babillot's techniques in \cite{Wo-00} and \cite{MR0978023} respectively.

\subsubsection{Restriction around $0$}\label{subsubsec:restriction_0}

We first restrict the integral \eqref{eq:form_int_doob} to any neighborhood of $0$, in order to use forthcoming decompositions of $\widehat{\mu}_c(\theta)$ in the integral formula. To do so, we multiply the integrand by a compactly supported function and show that the difference between the original integral and the new one can be neglected.
     \begin{prop}\label{prop:restriction_zero}
    Let $ \delta \in ]0,1[$ and $f$ a $\mathcal{C}^\infty$, radial function on $[- \pi, \pi]^d$ such that $f(\theta) = 1$ when $\|\theta\| \leq \frac{\delta}{2}$ and $f(\theta) = 0$ when $\|\theta\| \geq \frac{2}{3}  \delta$. Then for every $c \in \partial C$ and $i,j \in \{1 , \ldots, p\}$, \[\int_{[- \pi, \pi]^d} (1-f(\theta)) \left( (\mathrm{I}_p - \widehat{\mu}_c(\theta))^{-1} \right)_{i,j} e^ {\mathbf{i}x \cdot \theta} \mathrm{d}\theta = \underset{\|x\| \to + \infty}{O} \left( \|x\|^{-k}\right)\] for every $k \in \N$. Besides, the big-O is uniform in $c \in \partial C$.    
    \end{prop}

We now explain how to choose a suitable $\delta$ in the previous proposition. Given our goal \eqref{eq:dom_part}, we need at least $k_c$ to exist. To do so, we need a version of the Jordan-Chevalley decomposition from Proposition~\ref{prop:jordan_chevalley} for the Doob transform in a neighborhood of $0$ which is \emph{independent of $c$}. Although the domain of the integral is a subset of $\R^d$, we will temporarily go to $\C^d$ since analyticity will be useful to prove easily that some functions are smooth.

\begin{lemma}\label{lem:jordan_chevalley_uniform}
        There exists a neighborhood $V$ of $0$ in $\C^d$ such that for all $c \in \partial \mathcal{C}$ and $\theta \in V$, the decomposition 
        \begin{equation}\label{eq:jordan_chevalley_uniform}
           \widehat{\mu}_c(\theta) = k_c(\theta) p_c(\theta) + Q_c(\theta) 
        \end{equation}
         is valid, where:
    \begin{itemize}
        \item $k_c(\theta) \in \mathbb{C}$ is the simple leading eigenvalue of $\widehat{\mu}_c(\theta)$;
        \item $p_c(\theta)$ is the projection onto the eigenspace associated with the eigenvalue $k_c(\theta)$, along the sum of the other generalized eigenspaces;
        \item $p_c(\theta) Q_c(\theta) = Q_c(\theta)p_c(\theta) = 0$;
        \item the spectral radius of $Q_c(\theta)$ is smaller than $|k_c(\theta)|$.
    \end{itemize}
        
        Hence for $\theta \in \left( V \cap \R ^d\right)  \setminus \{0\}$, \[\left( \mathrm{I}_p - \widehat{\mu}_c(\theta) \right)^{-1} = \frac{1}{1-k_c(\theta)} p_c(\theta) + \sum_{n = 0}^{+ \infty} Q_c(\theta) ^n.\]

        Besides, the functions $\widehat{\mu}_c$, $k_c$, $p_c$ and $Q_c$ are analytic.
    \end{lemma}

From now on, $f$ will denote a function like in Proposition~\ref{prop:restriction_zero}, with $\delta >0$ small enough so that the support of $f$ is included in $\frac{1}{2}V$, where $V$ is the neighborhood from Lemma~\ref{lem:jordan_chevalley_uniform}.

\subsubsection{Identification of the leading term via dyadic splitting of integrals}\label{subsubsec:leading_term}

From the integral formula \eqref{eq:form_int_doob} and Proposition~\ref{prop:restriction_zero}, we get 
\begin{equation*}
    G_c\big( (0,i), (x,j) \big) = \frac{1}{(2 \pi)^d} \int_{\R^d} f(\theta) \left( (\mathrm{I}_p - \widehat{\mu}_c(\theta))^{-1} \right)_{i,j} e^ {\mathbf{i}x \cdot \theta} \mathrm{d}\theta + \underset{\|x\| \to + \infty}{o} \left( \|x\|^{- \frac{d-1}{2}} \right)
\end{equation*}
with a little-$o$ uniform in $c$, hence using Lemma~\ref{lem:jordan_chevalley_uniform}, we obtain the following asymptotics, uniformly in $c$:
\begin{equation} \label{eq:integral_formula_restricted_around_0}
    G_c\big( (0,i), (x,j) \big) = \frac{1}{(2 \pi)^d} \int_{\R^d}  \left(  \frac{f(\theta)}{1-k_c(\theta)} p_c(\theta) + f(\theta) \sum_{n = 0}^{+ \infty} Q_c(\theta) ^n \right)_{i,j} e^{\mathbf{i} x \cdot \theta} \mathrm{d}\theta + \underset{\|x\| \to + \infty}{o} \left( \|x\|^{- \frac{d-1}{2}} \right).
\end{equation}

We will proceed in two steps to obtain \eqref{eq:dom_part}. The first one, rather elementary, is to get rid of the part of the integral with the infinite sum. 

\begin{prop}\label{prop:Qc_part_neglect}
        For all $k \in \N$ and $c \in \partial \mathcal{C}$, \[\displaystyle \int_{[- \pi, \pi]^d} f(\theta) \sum_{n=0}^{+ \infty} Q_c(\theta)^n e^{\mathbf{i}x \cdot \theta} \mathrm{d}\theta = \underset{\|x\| \to + \infty}{O}\left( \|x\|^{-k} \right).\] Besides, the big-$O$ is uniform in $c \in \partial \mathcal{C}$.
    \end{prop}

The last step to obtain \eqref{eq:dom_part}, that is to replace $p_c(\theta)$ with $p_c(0)$, is more technical. The goal is to prove that 
\begin{equation}\label{eq:result_dyadic_integrals}
    \int_{\R^d} \frac{f(\theta)}{1 - k_c(\theta)} e^{- \mathbf{i}x \cdot \theta} \left( p_c(\theta) - {p_c(0)} \right) \mathrm{d}\theta = \underset{\|x\| \to + \infty}{o} \left( \|x\|^{-\frac{d-1}{2}} \right)
\end{equation} with a little-$o$ uniform in $c$. The first and natural idea is to show that this is a Fourier coefficient of a smooth enough function. However, $\frac{d-1}{2}$ is not necessarily an integer, so the usual link between smooth functions and the decay of their Fourier coefficients is not that simple; and above all, the function is not smooth because of the $1-k_c(\theta)$ at the denominator, which leads to a singularity at $0$. The method of dyadic splitting of integrals used by Babillot in \cite{MR0978023} and later by Guibourg in \cite{these_guibourg} allows to overcome those difficulties. It consists in splitting the space of integration into smaller and smaller concentric annuli with adapted shapes, avoiding the singularity at $0$, and using the decay of Fourier coefficient on each annulus. The details and proofs are given in Appendix~\ref{subsec:dyadic_integrals}.

From \eqref{eq:integral_formula_restricted_around_0}, Proposition~\ref{prop:Qc_part_neglect} and \eqref{eq:result_dyadic_integrals}, we obtain \eqref{eq:dom_part}.

\subsection{The complex analysis ingredient}\label{sec:complex_analysis_ingredient_without_proofs}

In this section, we state the complex analysis results that will be useful in the proof of the main result. These techniques come from Woess's proof of the Ney and Spitzer theorem in his book \cite{Wo-00}. Proofs are postponed to Appendix \ref{sec:app:complex_analysis}

 Following Woess in \cite[(25.19)]{Wo-00}, we define the rotated complex Laplace transform of $\mu$ 
 \begin{equation}\mathcal{L}\mu_u(z) = \left( \sum_{x \in \Z^d} e^{z \cdot R_u x} {\left(\mu_{c(u)}\right)}_{i,j}(x) \right)_{1 \leq i,j \leq p} 
 \end{equation}\label{eq:rotated_laplace_transform}
        where:
        \begin{itemize}
            \item $z \in \mathbb C ^d$;
            \item $u$ is in the sphere $\mathbb{S}^{d-1}$;
            \item $R_u$ is the rotation of $\R^d$ that sends $u$ onto the first vector $e_1$ of the canonical basis and leaves the orthogonal of $\{u,e_1\}$ invariant;
            \item $c(u)$ was defined after Theorem~\ref{thm:hennequin}.
        \end{itemize}
        We denote $\kappa_u(z)$ its leading eigenvalue, provided it exists.
    The next lemma ensures that $\kappa_u$ is defined in a neighborhood of $0$ \emph{independent of $u$}.

     \begin{lemma}\label{lem:analytic_uniformly}
        There exists a neighborhood $V$ of $0$ such that for all $u \in \mathbb{S}^{d-1}$ and $z \in V$, $\mathcal{L}\mu_u(z)$ has a unique leading eigenvalue $\kappa_u(z)$, which is simple. Besides, for all $u \in \mathbb{S}^{d-1}$, $\kappa_u$ is analytic on $V$.
    \end{lemma}

    The main result on the rotated Laplace transform is the following Weierstrass preparation theorem, which is an adaptation of Proposition 25.20 of Woess in \cite{Wo-00}. It requires the following definition. For $u \in \mathbb{S}^{d-1}$, we define the matrix $\sigma_u$ as the energy matrix of the process obtained by rotation $R_u$ of the Doob transform with parameter $c(u)$, and the matrix $\sigma_u^{(1)}$ as the matrix obtained from $\sigma_u$ by deleting its first row and column.

    \begin{prop}\label{prop:weierstrass_prep}
        \begin{enumerate}
            \item There is $\varepsilon >0$ such that for every $u \in \mathbb{S}^{d-1}$, there exist functions $\mathcal{A}_u$ and $\mathcal{B}_u$, analytic in the balls $B(0, \varepsilon)_{\mathbb C ^{d-1}}$ and $B(0, \varepsilon)_{\mathbb C ^d}$, such that in $B(0, \varepsilon)_{\mathbb C ^d}$ one can decompose 
        \begin{equation}\label{eq:weierstrass_decomp}
        1 - \kappa_u(z) = \left(z_1 - \mathcal{A}_u(z')\right) \mathcal{B}_u(z),
        \end{equation}
        where $z = (z_1, z') \in \C^d$, $z_1 \in \C$ and $z' \in \C^{d-1}$. Besides, $\mathcal{B}_u$ does not vanish on $B(0, \varepsilon)_{\mathbb C^d}$.
        
         \item \label{it:weierstrass_2} The function $\mathcal{A}_u(z')$ (resp. $\mathcal{B}_u(z)$) and its derivatives are continuous with respect to the variables $(z',u)$ (resp. $(z,u)$). 

        \item \label{it:weierstrass_3}$\mathcal{A}_u(0) = 0$, $\nabla \mathcal{A}_u(0) = 0$, the Hessian of $\mathcal{A}_u$ is $- \frac{1}{\|m_{c(u)}\|} \sigma_u^{(1)}$ and $\mathcal{B}_u(0) = -\left\| m_{c(u)}\right\|$. 

        \item \label{it:weierstrass_4} Therefore, $\mathcal{A}_u(\mathbf{i}z') =  \frac{1}{2\|m_{c(u)}\|} z' \cdot \sigma_u^{(1)} z' + O\left( \left\|z'^3\right\|\right)$, with a big-$O$ uniform in $u$, and we can choose $\varepsilon$ small enough so that when $z'$ is in the \emph{real} ball $B(0, \varepsilon)_{\R ^{d-1}}$, the modulus of the big-$O$ is smaller than $\frac{1}{4\|m_{c(u)}\|} z' \cdot \sigma_u^{(1)} z'$.
        \end{enumerate}

    \end{prop}

    \subsection{Proof of the main result}
    We now can prove the main result, Theorem \ref{thm:main}. We first work on the Green function of the Doob transform, adapting the proof from \cite{Wo-00}.
    \begin{thm}\label{thm:main_with_doob_transform}
        For $x \in \Z^d$, we choose $c \in \partial \mathcal{C}$ such that $\frac{m_c}{\|m_c\|} = \frac{x}{\|x\|}$, that is, with the notations following Theorem~\ref{thm:hennequin}, $c =c \left(u \right)$ where $u = \frac{x}{\|x\|}$.
        Then \[G_{c(u)}\big( (0,i), (x,j) \big) = \frac{1}{(2\pi)^{\frac{d-1}{2}}} \frac{\left\| m_{c(u)} \right\|^{\frac{d-3}{2}}}{\sqrt{\det \left( \sigma_u^{(1)} \right)}} \left( p_{c(u)}(0) \right)_{i,j} \|x\|^{\frac{d-1}{2}} + \underset{\|x\| \to + \infty}{o}\left( \|x\|^{\frac{d-1}{2}} \right).\]
    \end{thm}

    \begin{proof}
            We start from \eqref{eq:dom_part} which we recall: \[  G_{c(u)}\big( (0,i), (x,j) \big) = \left(\frac{1}{(2 \pi)^d} \int_{\R^d} \frac{f(\theta)}{1-k_{c(u)}( \theta)} e^{-\mathbf{i} x \cdot \theta} \mathrm{d} \theta\right) {p_{c(u)}(0)}_{i,j} + \underset{\|x\| \to + \infty}{o} \left( \|x\|^{- \frac{d-1}{2}} \right).\]  Note that $c(u) = c \left( \frac{x}{\|x\|} \right)$ depends on $x$, however the little-$o$ of \eqref{eq:dom_part} is uniform in $c$, so we may use \eqref{eq:dom_part}.
            
            We remind that in \eqref{eq:dom_part}, $f$ is any smooth, compactly supported in a ball $B(0, \delta)$, radial function, locally constant to $1$ on a small enough neighborhood of $0$. Since the decompositions from Lemma~\ref{lem:jordan_chevalley_uniform} and Proposition~\ref{prop:weierstrass_prep} are uniform in $c$ and $u$, we choose a function $f$ with support allowing those decompositions, independently of $x$. Therefore, if $\theta \in \supp f$,
            \begin{equation}
            1 - k_{c(u)}(\theta) = 1 - \kappa_u\left( \mathbf{i} R_u \theta \right) = \left( \mathbf{i}\left(R_u \theta\right)_1 - \mathcal{A}_u \left(\mathbf{i} \left( R_u \theta \right) ' \right) \right) \mathcal{B}_u \left( \mathbf{i}R_u \theta \right)
            \end{equation}
            where the prime symbol denotes a vector in which we remove the first coordinate and the subscript ``$1$'' denotes this first coordinate, so that $\theta = (\theta_1, \theta')$. Hence
        \[\int_{\R^d} \frac{f(\theta)}{1-k_{c(u)}( \theta)} e^{-\mathbf{i} x \cdot \theta} \mathrm{d} \theta = \int_{\R^d} \frac{f(\theta)}{ \left( \mathbf{i} \left(R_u \theta\right)_1 - \mathcal{A}_u \left( \mathbf{i}\left( R_u \theta \right) ' \right) \right) \mathcal{B}_u \left(\mathbf{i} R_u \theta \right) } e^{-\mathbf{i} x \cdot \theta} \mathrm{d} \theta.\]
        The linear change of variable $\omega = R_u \theta$ and the invariance of $f$ by rotation lead to 
        \begin{align*}
            \int_{\R^d} \frac{f(\theta)}{1-k_{c(u)}( \theta)} e^{-\mathbf{i} x \cdot \theta} \mathrm{d} \theta &= \int_{\R^d} \frac{f\left(  \omega \right)}{ \left( \mathbf{i}\omega_1 - \mathcal{A}_u \left(\mathbf{i}\omega '\right)\right) \mathcal{B}_u(\mathbf{i}\omega)} e^{- \mathbf{i} x \cdot \left( R_u^{-1} \omega \right)} \mathrm{d} \omega\\
            &= \int_{\R^d} \frac{f\left( \omega \right)}{ \left( \mathbf{i}\omega_1 - \mathcal{A}_u \left(\mathbf{i}\omega '\right)\right) \mathcal{B}_u(\mathbf{i}\omega)} e^{- \mathbf{i} (R_u x) \cdot  \omega } \mathrm{d} \omega\\
            &= \int_{\R^d} \frac{f\left(  \omega \right)}{ \left( \mathbf{i}\omega_1 - \mathcal{A}_u \left(\mathbf{i}\omega '\right)\right) \mathcal{B}_u(\mathbf{i}\omega)} e^{- \mathbf{i} \|x\|  \omega_1 } \mathrm{d} \omega,
            \end{align*}
        where the last equality uses $R_u x =  \|x\| e_1$.
        We denote $f_u(\omega) = \frac{f(\omega)}{\mathcal{B}_u(\mathbf{i} \omega)}$, which is a smooth, compactly supported function, so that 
        \[\int_{\R^d} \frac{f(\theta)}{1-k_{c(u)}( \theta)} e^{-\mathbf{i} x \cdot \theta} \mathrm{d} \theta = \int_{\R^d} \frac{f_u \left(  \omega \right)}{  \mathbf{i}\omega_1 - \mathcal{A}_u \left(\mathbf{i}\omega '\right)} e^{- \mathbf{i} \|x\|  \omega_1 } \mathrm{d} \omega.\]
        
        Since $\mathcal{A}_u(\mathbf{i}z') =  \frac{1}{2\left\|m_{c(u)}\right\|} z' \cdot \sigma_u^{(1)} z' + O\left( z'^3\right)$ where the modulus of the big-$O$ is smaller than $\frac{1}{4\left\|m_{c(u)}\right\|} {z' \cdot \sigma_u^{(1)} z'}$ according to Proposition \ref{prop:weierstrass_prep}, we have $\Re \left( \mathcal{A}_u(\mathbf{i}z') \right) >0$, therefore $\Re \left( \mathbf{i}\omega_1 - \mathcal{A}_u \left(\mathbf{i}\omega '\right) \right)<0$, which leads to 
        \[\frac{1}{\mathbf{i}\omega_1 - \mathcal{A}_u \left(\mathbf{i}\omega '\right)} = - \int_{0}^{+ \infty} \exp\left( s \left( \mathbf{i}\omega_1 - \mathcal{A}_u \left(\mathbf{i}\omega '\right) \right)  \right) \mathrm{d}s.\]
        Therefore,
        \begin{align}\label{eq:int_formula_fourier_first_variable}
        \int_{\R^d} \frac{f(\theta)}{1-k_{c(u)}( \theta)} e^{-\mathbf{i} x \cdot \theta} \mathrm{d} \theta &=  - \int_{\R^d} \int_0^{+\infty} f_u(\omega) e^{\mathbf{i}\omega_1 \left(s- \|x\|\right)} e^{-s \mathcal{A}_u \left( \mathbf{i}\omega' \right)} \mathrm{d}s~\mathrm{d}\omega \nonumber\\
        &=  -\int_{\R^{d-1}} \int_{0}^{+\infty} \int_{\R} f_u\left(\omega_1, \omega'\right) e^{\mathbf{i}\omega_1 \left(s- \|x\|\right)} \mathrm{d}\omega_1~ e^{-s \mathcal{A}_u \left( \mathbf{i}\omega' \right)} \mathrm{d}s~\mathrm{d}\omega' \nonumber \\
        &=- \int_{\R^{d-1}} \int_{-\|x\|}^{+\infty} \int_{\R} f_u\left(\omega_1, \omega'\right) e^{\mathbf{i}\omega_1 t} \mathrm{d}\omega_1~ e^{-\left( t + \|x\| \right) \mathcal{A}_u \left( \mathbf{i}\omega' \right)} \mathrm{d}t~\mathrm{d}\omega' \nonumber \\
        &=- \int_{B(0,\delta)_{\R^{d-1}}} \int_{-\|x\|}^{+\infty} \int_{\R} f_u\left(\omega_1, \omega'\right) e^{\mathbf{i}\omega_1 t} \mathrm{d}\omega_1~ e^{-\left( t + \|x\| \right) \mathcal{A}_u \left( \mathbf{i}\omega' \right)} \mathrm{d}t~\mathrm{d}\omega'  \nonumber \\
        & = -\int_{B(0,\delta)_{\R^{d-1}}} \int_{- \|x\|}^{+\infty} \widehat{f}_u(t, \omega') e^{-\left( t + \|x\| \right) \mathcal{A}_u \left( \mathbf{i}\omega' \right)} \mathrm{d}t~\mathrm{d}\omega', 
        \end{align}
        where $\widehat{f}_u(t, \omega') = \int_\R f_u(\omega_1, \omega') e^{\mathbf{i}t \omega_1} \mathrm{d}\omega_1$ is the Fourier transform of $f_u$ in its first coordinate.

        For fixed $\omega'$, $t \mapsto \widehat{f}_u(t, \omega')$ is the Fourier transform of a smooth function, so $\widehat{f}_u (t, \omega') = \underset{t \to \pm \infty}{O}\left( t^{-l} \right)$ for every $l \in \N$, and the constant in the $O\left( t^{-l}\right)$ is controlled by the Fourier transform of $\frac{\partial^l f_u}{\partial \omega_1^l}(\cdot, \omega')$. We use this control to prove that the big-$O$ is uniform in $(u, \omega')$. Thanks to \ref{it:weierstrass_2} of Proposition~\ref{prop:weierstrass_prep}, $(u, \omega) \mapsto \frac{\partial^l f_u}{\partial \omega_1^l} (\omega)$ is continuous (as a consequence of the result on $\mathcal{B}_u(z)$), so the compactness of $\mathbb{S}^{d-1} \times \overline{B(0, \delta)_{\R^d}}$ allows us to bound $\frac{\partial^l f_u}{\partial \omega_1^l} (\omega)$ uniformly in $(u, \omega)$, and thus, using $\mathrm{supp}(f_u) \subset B(0, \delta)_{\R^d}$, the Fourier transform of $\frac{\partial^l f_u}{\partial \omega_1^l}(\cdot, \omega')$ can be bounded uniformly in $(u, \omega')$, so that for every $l \in \N$, 
        \begin{equation}\label{eq:decrease_fourier_first_variable}
        \text{$\widehat{f}_u(t, \omega') = \underset{t \to \pm \infty}{O}\left( t^{-l} \right)$ with a big-$O$ uniform in $(u, \omega')$.}
        \end{equation}
        We denote $C_l$ a positive constant such that for every $u \in \mathbb{S}^{d-1}$, $\omega' \in B(0, \delta)_{\R^{d-1}}$ and $t \in \R$, $\left| \widehat{f}_u(t,\omega') \right| \leq C_l |t|^{-l}.$  
        
        In \eqref{eq:int_formula_fourier_first_variable}, we split the inner integral into two parts, the first part for $s \in \left[ - \|x\|, - \frac{\|x\|}{2} \right]$ and the second one for $s \in \left[-\frac{\|x\|}{2}, +\infty\right)$. We do this is because the quantity $\left( \frac{t}{\|x\|} + 1 \right)^{-\frac{d-1}{2}}$ will appear, so by taking $t > - \frac{\|x\|}{2}$ in a part of the integral, we make sure it will not explode. We get:
        \begin{multline}\label{eq:int_splitted}
             \int_{\R^d} \frac{f(\theta)}{1-k_{c(u)}( \theta)} e^{-\mathbf{i} x \cdot \theta} \mathrm{d} \theta = - \int_{B(0,\delta)_{\R^{d-1}}} \int_{- \|x\|}^{- \frac{\|x\|}{2}} \widehat{f}_u(t, \omega') e^{-\left( t + \|x\| \right) \mathcal{A}_u \left( \mathbf{i}\omega' \right)} \mathrm{d}t~\mathrm{d}\omega'  \\
             - \int_{B(0,\delta)_{\R^{d-1}}} \int_{- \frac{\|x\|}{2}}^{+\infty} \widehat{f}_u(t, \omega') e^{-\left( t + \|x\| \right) \mathcal{A}_u \left( \mathbf{i}\omega' \right)} \mathrm{d}t~\mathrm{d}\omega'.
        \end{multline}
        In the first part of \eqref{eq:int_splitted}, the decay comes from \eqref{eq:decrease_fourier_first_variable}. Indeed, the real part in the exponential is negative, so the modulus of the exponential is smaller than $1$, therefore, for all $l \in \N$,
        \[\left| \int_{B(0,\delta)_{\R^{d-1}}} \int_{- \|x\|}^{- \frac{\|x\|}{2}} \widehat{f}_u(t, \omega') e^{-\left( t + \|x\| \right) \mathcal{A}_u \left( \mathbf{i}\omega' \right)} \mathrm{d}t~\mathrm{d}\omega' \right| \leq \mathrm{Leb}\left( B(0,\delta)_{\R^{d-1}} \right) \frac{\|x\|}{2} C_l \left( \frac{\|x\|}{2} \right)^{-l}.\] Thus, the first integral of \eqref{eq:int_splitted} is a $O\left( \|x\|^{-(l-1)} \right)$ for every $l \in \N$.
        In particular, it is a $o\left( \|x\|^{- \frac{d-1}{2}}\right)$.\\
        We now look at the second part of \eqref{eq:int_splitted}, in which we made sure that the $t+\|x\|$ in the exponential is far enough from $0$. The change of variable $\lambda' = \sqrt{t + \|x\|} \omega'$ leads to 
        \begin{align*}
        \MoveEqLeft[3] {\int_{B(0,\delta)_{\R^{d-1}}} \int_{- \frac{\|x\|}{2}}^{+ \infty} \widehat{f}_u(t, \omega') e^{-\left( t + \|x\| \right) \mathcal{A}_u \left( \mathbf{i}\omega' \right)} \mathrm{d}t~\mathrm{d}\omega'}\\
        &=  \int_{- \frac{\|x\|}{2}}^{+ \infty} \int_{B(0,\delta)_{\R^{d-1}}} \widehat{f}_u(t, \omega') e^{-\left( t + \|x\| \right) \mathcal{A}_u \left( \mathbf{i}\omega' \right)}\mathrm{d}\omega' ~  \mathrm{d}t \\
        &= \int_{- \frac{\|x\|}{2}}^{+ \infty} \left( t + \|x\| \right)^{-\frac{d-1}{2}}\int_{B\left( 0, \sqrt{t + \|x\|} \delta \right)_{\R^{d-1}}} \widehat{f}_u \left(t, \frac{\lambda '}{\sqrt{t + \|x\|}} \right)  e^{-\left( t + \|x\| \right) \mathcal{A}_u \left( \mathbf{i}\frac{\lambda'}{\sqrt{t + \|x\|}} \right)} \mathrm{d} \lambda' ~\mathrm{d}t\\
        &= \|x\|^{- \frac{d-1}{2}} \int_{- \frac{\|x\|}{2}}^{+ \infty} \left( \frac{t}{\|x\|} + 1 \right)^{-\frac{d-1}{2}}\int_{B\left( 0, \sqrt{t + \|x\|} \delta \right)_{\R^{d-1}}}  \widehat{f}_u \left(t, \frac{\lambda '}{\sqrt{t + \|x\|}} \right)
        \\
        &~~~~e^{-\left( t + \|x\| \right) \mathcal{A}_u \left( \mathbf{i}\frac{\lambda'}{\sqrt{t + \|x\|}} \right)} \mathrm{d} \lambda' ~\mathrm{d}t.
        \end{align*}
        Therefore, \eqref{eq:int_splitted} becomes
        \begin{align}
            \MoveEqLeft[3] {\int_{\R^d} \frac{f(\theta)}{1-k_{c(u)}( \theta)} e^{-\mathbf{i} x \cdot \theta} \mathrm{d} \theta} \notag\\
            &= \|x\|^{- \frac{d-1}{2}} \int_{- \frac{\|x\|}{2}}^{+ \infty} \left( \frac{t}{\|x\|} + 1 \right)^{-\frac{d-1}{2}}\int_{B\left( 0, \sqrt{t + \|x\|} \delta \right)_{\R^{d-1}}}  \widehat{f}_u \left(t, \frac{\lambda '}{\sqrt{t + \|x\|}} \right) 
         \\
        &e^{-\left( t + \|x\| \right) \mathcal{A}_u \left( \mathbf{i}\frac{\lambda'}{\sqrt{t + \|x\|}} \right)} \mathrm{d} \lambda' ~\mathrm{d}t + o\left( \|x\|^{- \frac{d-1}{2}} \right).
        \end{align}

        We now use the dominated convergence theorem to prove that 
        \begin{equation}\label{eq:dom_conv_end_proof}
        \begin{split}
            \int_{- \frac{\|x\|}{2}}^{+ \infty} \left( \frac{t}{\|x\|} + 1 \right)^{-\frac{d-1}{2}} \int_{B\left( 0, \sqrt{t + \|x\|} \delta \right)_{\R^{d-1}}} \widehat{f}_u \left(t, \frac{\lambda '}{\sqrt{t + \|x\|}} \right) e^{-\left( t + \|x\| \right) \mathcal{A}_u \left( \mathbf{i}\frac{\lambda'}{\sqrt{t + \|x\|}} \right)} \mathrm{d} \lambda' ~\mathrm{d}t\\ 
            - \int_{- \infty}^{+ \infty} \int_{\R^{d-1}} \widehat{f}_u(t,0) e^{- \frac{1}{2\left\| m_{c(u)} \right\|} \lambda' \cdot \sigma_u^{(1)} \lambda'} \mathrm{d}\lambda'~ \mathrm{d}t \xrightarrow[\|x\| \to + \infty]{} 0.
            \end{split}
        \end{equation}
        
        We first use it in the inner integral. We fix $t$ and consider the function \[g: (x,\lambda') \mapsto \mathds{1}_{B\left( 0, \sqrt{t + \|x\|} \delta \right)_{\R^{d-1}}} \widehat{f}_u \left(t, \frac{\lambda '}{\sqrt{t + \|x\|}} \right) e^{-\left( t + \|x\| \right) \mathcal{A}_u \left( \mathbf{i}\frac{\lambda'}{\sqrt{t + \|x\|}} \right)}. \]
        The continuity of $\widehat{f}_u$ and the expansion $\mathcal{A}_u \left( \mathbf{i}\frac{\lambda'}{\sqrt{t + \|x\|}} \right) = \frac{1}{2 \left\|m_{c(u)} \right\| }\frac{1}{t + \|x\|} \lambda' \cdot \sigma_u^{(1)} \lambda' + O\left(\left(\frac{\lambda'}{\sqrt{t + \|x\|}}\right)^3 \right)$ coming from Proposition~\ref{prop:weierstrass_prep} show that for every $\lambda'$, 
        \begin{equation}\label{eq:pointwise_cv_first_DCT}
        g(x, \lambda') - \widehat{f}_u(t,0) e^{- \frac{1}{2\left\| m_{c(u)} \right\|} \lambda' \cdot \sigma_u^{(1)} \lambda'} \xrightarrow[\|x\| \to + \infty]{}0  .
        \end{equation}
        Notice that, since $u$ depends on $x$, we used the uniformity in $u$ of the big-$O$ and the equicontinuity of the family $\left(\widehat{f}_u (t, \cdot)\right)_{u \in \mathbb{S}^{d-1}}$ at $0$ to get \eqref{eq:pointwise_cv_first_DCT}. This family is even equi-lipschitz around $0$ since its derivative is bounded around $0$, uniformly in $u$, because of the usual argument of continuity of the functions $(u, \omega) \mapsto \frac{\partial \widehat{f}_u(t,\omega)}{\partial \omega_i}$, $i \geq 2$, and compactness of $\mathbb{S}^{d-1}$.
        To prove the domination assumption of the dominated convergence theorem, first notice that $\widehat{f}_u$ is bounded uniformly in $u$ (it is a consequence of \eqref{eq:decrease_fourier_first_variable} for $l = 0$).
        Besides, by \ref{it:weierstrass_4}~of Proposition~\ref{prop:weierstrass_prep}, $\Re \left( \mathcal{A}_u(\mathbf{i}z') \right) \geq \frac{1}{4 \left\|m_{c(u)}\right\|} z' \cdot \sigma_u^{(1)} z'$, so 
        \begin{align*} 
        \Re \left( -(t+\|x\|) \mathcal{A}_u\left( \mathbf{i} \frac{\lambda'}{\sqrt{t+\|x\|}} \right) \right)&\leq - \frac{1}{4 \left\|m_{c(u)}\right\|} \lambda' \cdot \sigma_u^{(1)} \lambda' \\
        &\leq - \frac{1}{4 \left\|m_{c(u)}\right\|} \min\left( \mathrm{Sp} \left(\sigma_u^{(1)}\right)\right) \left\|\lambda'\right\|^2 \tag{$\mathrm{Sp}$ denotes the spectrum of a matrix}\\
        & \leq - \inf_u \left(\frac{1}{4 \left\|m_{c(u)}\right\|} \right) \inf_u \left( \min\left( \mathrm{Sp} \left(\sigma_u^{(1)}\right)\right)\right) \left\| \lambda' \right\|^2
        \end{align*}
        where the infima are strictly positive by continuity in the variable $u$ and compactness of $\mathbb{S}^{d-1}$.
        Denoting $\alpha$ this upper bound, we can dominate $\left| g(x, \lambda') - \widehat{f}_u(t,0) e^{- \frac{1}{2\left\| m_{c(u)} \right\|} \lambda' \cdot \sigma_u^{(1)} \lambda'} \right|$ by $2\displaystyle\sup_{(u,\omega)} |\widehat{f}_u(\omega)| e^{- \alpha \left\| \lambda' \right\|^2}$, which is integrable. The dominated convergence theorem yields the convergence of the inner integral in \eqref{eq:dom_conv_end_proof}, i.e., 
        \begin{equation}\label{eq:inner_int}
        \begin{split}
        \int_{B\left( 0, \sqrt{t + \|x\|} \delta \right)_{\R^{d-1}}} \widehat{f}_u \left(t, \frac{\lambda '}{\sqrt{t + \|x\|}} \right) e^{-\left( t + \|x\| \right) \mathcal{A}_u \left( \mathbf{i}\frac{\lambda'}{\sqrt{t + \|x\|}} \right)} \mathrm{d} \lambda' - \int_{\R^{d-1}} \widehat{f}_u(t,0) e^{- \frac{1}{2\left\| m_{c(u)} \right\|} \lambda' \cdot \sigma_u^{(1)} \lambda'} \mathrm{d}\lambda' \\ 
        \xrightarrow[\|x\| \to +\infty]{}0 .
        \end{split}
        \end{equation}

        We use once again the dominated convergence theorem to prove the convergence \eqref{eq:dom_conv_end_proof}. We thus define \[h(x,t) = \mathds{1}_{\left[ - \frac{\|x\|}{2}, + \infty \right[}(t) \left( \frac{t}{\|x\|} + 1 \right)^{-\frac{d-1}{2}} \int_{B\left( 0, \sqrt{t + \|x\|} \delta \right)_{\R^{d-1}}} \widehat{f}_u \left(t, \frac{\lambda '}{\sqrt{t + \|x\|}} \right) e^{-\left( t + \|x\| \right) \mathcal{A}_u \left( \mathbf{i}\frac{\lambda'}{\sqrt{t + \|x\|}} \right)} \mathrm{d} \lambda' .\] Thanks to \eqref{eq:inner_int}, it is straightforward that for every $t \in \R$, \[h(x,t) - \int_{\R^{d-1}} \widehat{f}_u(t,0) e^{- \frac{1}{2\left\| m_{c(u)} \right\|} \lambda' \cdot \sigma_u^{(1)} \lambda'} \mathrm{d}\lambda' \xrightarrow[\|x\| \to + \infty]{}0.\]
        Besides, the previous inequality on the real part of the exponential and the elementary inequality $1+ \frac{t}{\|x\|} \geq \frac{1}{2}$ when $t \geq - \frac{\|x\|}{2}$ show that
        \begin{align*}
            |h(x,t)| & \leq 2^{\frac{d-1}{2}}\displaystyle\sup_{(u,\omega') \in \mathbb{S}^{d-1} \times B(0, \delta)} \left|\widehat{f}_u(t, \omega')\right| \int_{\R^{d-1}} e^{- \alpha \left\| \lambda' \right\|^2} \mathrm{d}\lambda', 
        \end{align*} and $\displaystyle \int_\R \displaystyle\sup_{(u,\omega') \in \mathbb{S}^{d-1} \times B(0, \delta)} \left|\widehat{f}_u(t, \omega')\right| \mathrm{d}t < \infty$ as a consequence of \eqref{eq:decrease_fourier_first_variable}, so we obtain the domination assumption that allows to use the dominated convergence theorem to prove \eqref{eq:dom_conv_end_proof}.\\
        
        At this point, \eqref{eq:dom_part}, \eqref{eq:int_splitted} and \eqref{eq:dom_conv_end_proof} prove that
        \begin{equation}\label{eq:result_without_explicit_computations}
        \begin{split}
            G_{c(u)}\big( (0,i), (x,j) \big) = - \frac{1}{(2 \pi)^d} \|x\|^{- \frac{d-1}{2}} \int_{-\infty}^{+\infty} \int_{\R^{d-1}} \widehat{f}_u(t,0) e^{- \frac{1}{2\left\| m_{c(u)} \right\|} \lambda' \cdot \sigma_u^{(1)} \lambda'} \mathrm{d}\lambda'~ \mathrm{d}t ~~{p_{c(u)}(0)}_{i,j}\\
            +  \underset{\|x\| \to + \infty}{o}\left( \|x\|^{-\frac{d-1}{2}} \right).
            \end{split}
        \end{equation}
        To conclude, remind that for any symmetric, positive-definite matrix $A$ of size $n$, $\displaystyle \int_{\R^n} e^{- \frac{1}{2} x \cdot Ax} \mathrm{d}x = \frac{(2 \pi)^{n/2}}{ \sqrt{\det(A)}}$. Therefore $\displaystyle\int_{\R^{d-1}}  e^{- \frac{1}{2\left\| m_{c(u)} \right\|} \lambda' \cdot \sigma_u^{(1)} \lambda'} \mathrm{d}\lambda' = \frac{\left(2 \pi \left\| m_{c(u)} \right\|\right)^{\frac{d-1}{2}}}{\sqrt{\det \left( \sigma_u^{(1)} \right) }}$. Besides, inverse Fourier and the equality \sloppy ${\mathcal{B}_u(0) = - \left\| m_{c(u)} \right\|}$ from \ref{it:weierstrass_3} of Proposition~\ref{prop:weierstrass_prep} yield \[\displaystyle \int_{-\infty} ^{+\infty} \widehat{f}_u(t,0) \mathrm{d}t =  2 \pi f_u(0,0) = 2 \pi \frac{f(0)}{\mathcal{B}_u(0)} =  \frac{-2 \pi }{\left\|m_{c(u)}\right\|}.\]

        Therefore, \eqref{eq:result_without_explicit_computations} becomes
        \[G_{c(u)}\big( (0,i), (x,j) \big) = \frac{1}{(2\pi)^{\frac{d-1}{2}}} \frac{\left\| m_{c(u)} \right\|^{\frac{d-3}{2}}}{\sqrt{\det \left( \sigma_u^{(1)} \right)}} \|x\|^{-\frac{d-1}{2}} {p_{c(u)}(0)}_{i,j} + \underset{\|x\| \to + \infty}{o}\left( \|x\|^{-\frac{d-1}{2}} \right).\qedhere\]
    \end{proof}

 The main result, Theorem~\ref{thm:main}, is then a direct application of Theorem~\ref{thm:main_with_doob_transform} and Proposition~\ref{prop:link_green_functions}, that expresses the Green function from the Green function of the Doob transform. We give a more precise formulation that explicitly describes the functions $\chi_{i,j}$ of Theorem~\ref{thm:main}.

 \begin{thm}\label{thm:main_precise}
     Let $i,j \in \{1, \ldots, p\}$. As $\|x\|$ tends to infinity in a fixed direction, \[\frac{G \big( (0,i), (x,j) \big)}{\|x\|^{-\frac{d-1}{2}} e^{- c \left( \frac{x}{\|x\|} \right) \cdot x}}\]
     converges, and the convergence is uniform in the directions. More precisely, with previous notations, as $\|x\|$ tends to infinity,
     \[G \big( (0,i), (x,j) \big) \sim \frac{1}{(2\pi)^{\frac{d-1}{2}}} \frac{\left( \varphi_{c\left( \frac{x}{\|x\|} \right)} \right)_i \left\| m_{c\left( \frac{x}{\|x\|} \right)} \right\|^{\frac{d-3}{2}} {p_{c\left( \frac{x}{\|x\|} \right)}(0)}_{i,j}}{\left( \varphi_{c\left( \frac{x}{\|x\|} \right)} \right)_j \sqrt{ \det \left( \sigma_{\frac{x}{\|x\|}}^{(1)} \right)} } \|x\|^{-\frac{d-1}{2}} e^{- c \left( \frac{x}{\|x\|} \right) \cdot x}.\]
 \end{thm}

\section{Appendix: Proofs of Section \ref{sec:dominating_part_without_proofs}}\label{sec:dominating_part}
In this appendix, we prove the results from Section~\ref{sec:dominating_part_without_proofs}. 

\subsection{Uniform decay of Fourier coefficients}\label{subsec:annex_A1}

The key argument of the proofs of Appendix \ref{subsec:annex_A1} consists in using the decay of Fourier coefficients of smooth functions. We remind here the proof of this classical fact, as it shows how to control the constants in the little-$o$, which will be crucial to obtain asymptotic expansions that are uniform in $c \in \partial \mathcal{C}$, like in Proposition~\ref{prop:restriction_zero}.

\begin{lemma}\label{lem:decrease_fourier_coeff}
        Let $g : \R^d \longrightarrow \mathbb{C} $ a $2\pi$-periodic in each variable function, integrable on $[-\pi, \pi]^d$. For $x \in \Z^d$, we note $c_x(g) = \frac{1}{(2 \pi)^d} \displaystyle \int_{[- \pi, \pi] ^d} g(\theta) e^{\mathbf{i} x \cdot \theta} \mathrm{d}\theta$ its $x^{\mathrm{th}}$ Fourier coefficient.
        If $g \in \mathcal{C}^k \left( \R ^d \right)$, then ${c_x(g) = \underset{\|x\| \to + \infty}{o}\left( \|x\|^{-k} \right)}$.
    \end{lemma}

    \begin{proof}
        Since all the norms on $\R^d$ are equivalent, we can work with $\| \cdot \|_1$ defined by $\|x\|_1 = \sum_{i=1}^d |x_i|$.
        Repeated integrations by parts lead to: \[ \forall \alpha \in \N^d \text{ s.t. } \|\alpha\|_1 \leq k,~~~~c_x \left( \partial_{\theta_1}^{\alpha_1} \ldots \partial_{\theta_d}^{\alpha_d}  g \right) = x_1^{\alpha_1} \ldots x_d^{\alpha_d} c_x(g).\] We then develop $\|x\|_1^k$ using the multinomial theorem, as a sum of the form $\|x\|_1^k = \sum_{\|\alpha\|_1 = k } C_\alpha |x_1|^{\alpha_1} \ldots |x_d|^{\alpha_d}$. Therefore, 
        \begin{equation}\label{eq:decrease_fourier_coeff_explicit_constant}
        \|x\|_1 ^k |c_x(g)| = \sum_{\|\alpha\|_1 = k} C_\alpha \left| c_x \left( \partial_{\theta_1}^{\alpha_1} \ldots \partial_{\theta_d}^{\alpha_d}  g \right) \right|. 
        \end{equation}
        Because of the Riemann-Lebesgue lemma, the right-hand side tends to $0$ as $\|x\|_1$ goes to $+\infty$, hence ${c_x(g) = \underset{\|x\| \to + \infty}{o}\left( \|x\|^{-k} \right)}$.
    \end{proof}

We can now prove the results from Sections \ref{subsubsec:restriction_0} and \ref{subsubsec:leading_term}.

\begin{proof}[Proof of Proposition~\ref{prop:restriction_zero}]
        The function $f$ is smooth and identically zero outside a neighborhood of $0$ included in $(-\pi, \pi)^d$, so it can be extended into a smooth, $2 \pi$-periodic in each variable function, still noted $f$.
        
        Since ${\theta \mapsto  (\mathrm{I}_p - \widehat{\mu}_c(\theta))^{-1}}$  is smooth outside of $0$ and $1 - f$ is identically $0$ in a neighborhood of $0$, the function ${g_c : \theta \mapsto (1-f(\theta)) \left( (\mathrm{I}_p - \widehat{\mu}_c(\theta))^{-1} \right)_{i,j}}$ is smooth. Therefore, 
        \[\frac{1}{(2\pi)^d}\int_{[- \pi, \pi]^d} (1-f(\theta)) \left( (\mathrm{I}_p - \widehat{\mu}_c(\theta))^{-1} \right)_{i,j} e^ {\mathbf{i}x \cdot \theta} \mathrm{d}\theta = c_{x}(g_c)\] 
        is a Fourier coefficient of a smooth periodic function. Lemma~\ref{lem:decrease_fourier_coeff} implies that $c_{x}(g_c)= o \left(\|x\|^{-k}\right)$ for every $k \in \N$, hence $c_{x}(g_c)= O \left(\|x\|^{-k}\right)$ for every $k \in \N$. Besides, going back to \eqref{eq:decrease_fourier_coeff_explicit_constant}, we see that the constant in the $O\left(\|x\| ^{-k} \right)$ is controlled by the Fourier coefficients of the partial derivatives of $g_c$ of order $k$.
        Since we want a big-$O$ uniform in $c \in \partial \mathcal{C}$, we need to bound those partial derivatives \emph{uniformly in $c$}. To do so, we will prove that
        \begin{equation}\label{eq:g_c_and_der_continuous}
            \text{the function $g_c$ and its derivatives are continuous functions of the variables $(\theta,c)$.}
        \end{equation}
        With \eqref{eq:g_c_and_der_continuous}, the Fourier coefficients of the derivatives of $g_c$ of order $k$ are uniformly bounded by \sloppy $\displaystyle \max_{\|\alpha\|_1 = k} ~\max_{(\theta, c) \in [-\pi, \pi]^d \times \partial C} \left|\partial_{\theta_1}^{\alpha_1} \ldots \partial_{\theta_d}^{\alpha_d} g_c(\theta)\right|$. This bound is finite as a consequence of the continuity from \eqref{eq:g_c_and_der_continuous} and compactness of $[- \pi, \pi]^d \times \partial \mathcal{C}$, which comes from the compactness of $\partial \mathcal{C}$ in Proposition \ref{prop:ensemble_C}. Thus the constant in the $O\left( \|x\|^{-k} \right)$ is uniform, which ends the proof.

        To prove \eqref{eq:g_c_and_der_continuous}, we write explicitly $\widehat{\mu}_c(\theta)$:
            \begin{equation*}
            \widehat{\mu}_c(\theta) = \left(\sum_{x \in \Z^d} e^{\mathbf{i}x \cdot \theta} {(\mu_c)}_{i,j}(x)\right)_{1 \leq i,j \leq p}= \left(\sum_{x \in \Z^d} e^{\mathbf{i}x \cdot \theta} \frac{(\varphi_c)_j}{ (\varphi_c)i}e^{c \cdot x} \mu_{i,j}(x)\right)_{1 \leq i,j \leq p}.
            \end{equation*}
            Since $1$ is a simple eigenvalue of $L\mu(c)$, where $L\mu$ is a continuous function of $c$, the eigenvector $\varphi_c$ is continuous in the variable $c$ (at least if we choose $\varphi_c$ to be a unit vector), see \cite[Chapter Two, \textsection 5]{kato}.
            Therefore, it is straightforward that $\widehat{\mu}_c$ and its derivatives with respect to $\theta$ are continuous in the variables $(\theta,c)$. Besides, if we denote $\mathrm{adj}(A)$ the adjugate matrix of an invertible matrix $A$, the formula $A^{-1} = \frac{1}{\det(A)} \mathrm{adj}(A)$ shows that matrix inversion, i.e. $A \mapsto A^{-1}$, is a rational function of the coefficients of $A$. Therefore, the derivatives of $\left( \mathrm{I}_p - \widehat{\mu}_c(\theta) \right)^{-1}$ are rational functions of $\widehat{\mu}_c(\theta)$ and of its derivatives, which are continuous in the variables $(\theta,c)$. Consequently, $g_c$ and its derivatives are continuous in the variables $(\theta, c)$, hence \eqref{eq:g_c_and_der_continuous}.\end{proof}

\begin{proof}[Proof of Lemma~\ref{lem:jordan_chevalley_uniform}]
    To have the decomposition \eqref{eq:jordan_chevalley_uniform}, we only need $\widehat{\mu}_c(\theta)$ to have a single leading eigenvalue (the rest is a consequence of the Jordan-Chevalley decomposition). Thanks to the Perron-Frobenius theorem, we know it is the case for $\widehat{\mu}_c(0)$ for all $c \in \partial \mathcal{C}$, so the continuity of $(\theta, c) \mapsto \widehat{\mu}_c(\theta)$ and of the spectrum ensures that, for every $c \in \partial \mathcal{C}$, there exist open neighborhoods $V_c$ of $0$ and $U_c$ of $c$ such that for all $(\theta, c') \in V_c \times U_c$, $\widehat{\mu}_{c'}(\theta)$ has a simple leading eigenvalue. From the open cover of the compact set $\partial \mathcal{C}$ by $\left(U_c\right)_{c \in \partial \mathcal{C}}$, we extract a finite cover $U_{c_1}, \ldots, U_{c_m}$. Then $V := \bigcap_{l = 1}^m V_{c_l}$ is a neighborhood of $0$ independent of $c$, and for every $(\theta,c) \in V \times \partial \mathcal{C}$, $\hat{\mu}_c(\theta)$ as a simple leading eigenvalue.

    The analyticity of $\widehat{\mu}_c$ follows from the definition, and the analyticity of the other terms of the decomposition is classical (see \cite[Chapter 2, \textsection 1]{kato}).
\end{proof}

\begin{proof}[Proof of Proposition~\ref{prop:Qc_part_neglect}]
        The idea is similar to  Proposition~\ref{prop:restriction_zero}: we prove that $\displaystyle \int_{[- \pi, \pi]^d} f(\theta) \sum_{n=0}^{+ \infty} Q_c(\theta)^n e^{\mathbf{i}x \cdot \theta} \mathrm{d}\theta$, $x \in \Z^d$, are Fourier coefficients of a smooth periodic function whose derivatives are bounded uniformly in $c \in \partial \mathcal{C}$. Writing $\sum_{n=0}^{+ \infty} Q_c(\theta)^n = \left( \mathrm{I}_p - Q_c(\theta) \right)^{-1}$, which is possible because the spectral radius of $Q_c(\theta)$ is smaller than $1$, the same argument of rationality of the inversion as in the proof of Proposition~\ref{prop:restriction_zero} shows that we only have to prove that the functions $Q_c$ are smooth and that their derivatives are continuous in the variables $(\theta,c)$ on $\supp(f) \times \partial \mathcal{C}$.
        
        The functions $Q_c$ are analytic in the neighborhood $V$ from Lemma~\ref{lem:jordan_chevalley_uniform}, and therefore on the polydisc $D(0, 2 \delta)$, where $\delta$ was chosen after Lemma~\ref{lem:jordan_chevalley_uniform}. Lemma~\ref{lem:continuity_derivatives_holom} tells us that the derivatives of $Q_c$ are continuous in the variables $(\theta,c)$ on $D(0, \delta) \times \partial \mathcal{C}$, hence on $\supp(f) \times \partial \mathcal{C}$.
            \end{proof}

    \subsection{Dyadic splitting of integrals method} \label{subsec:dyadic_integrals}
        In this section, we prove the formula \eqref{eq:result_dyadic_integrals} which we remind here: \begin{equation*}
    \int_{\R^d} \frac{f(\theta)}{1 - k_c(\theta)} e^{- \mathbf{i}x \cdot \theta} \left( p_c(\theta) - {p_c(0)} \right) \mathrm{d}\theta = \underset{\|x\| \to + \infty}{o} \left( \|x\|^{-\frac{d-1}{2}} \right),
\end{equation*} uniformly in $c \in \partial \mathcal{C}$. We use the dyadic splitting of integrals, which we adapt from Babillot in \cite{MR0978023}.

        We first define concentric annuli with shapes adapted to the denominator $1-k_c(\theta)$. To understand what is an adapted shape, we use the following lemma from \cite[(2.38)]{MR0978023}, for which we remind that since the beginning of Section \ref{sec:proof_main}, we work under the non centering Assumption \ref{assump:non_centered}.
        \begin{lemma}
            There exist a neighborhood $U$ of $0$ in $\R^d$ and two constants $C_1,C_2>0$ such that for every $c \in \partial \mathcal{C}$ and $\theta \in U$, \[C_1 \left| \mathbf{i} \theta \cdot \frac{m_c}{\|m_c\|} + \left\| \theta'_c \right\|^2 \right| \leq |1-k_c(\theta)| \leq C_2 \left| \mathbf{i} \theta \cdot \frac{m_c}{\|m_c\|} + \left\| \theta'_c \right\|^2 \right|, \] where $\theta'_c$ denotes the orthogonal projection of $\theta$ onto the orthogonal of $m_c$.
            
        \end{lemma}
        \begin{proof}
            This is a simple consequence of the Taylor expansion of $k_c$ around $0$. The neighborhood $U$ and the constants $C_1, C_2$ can be chosen independently of $c$ because $c$ lies in the compact set $\partial \mathcal{C}$.
        \end{proof}
        In other words, if $R_c$ denotes the rotation that sends $\frac{m_c}{\|m_c\|}$ onto the first vector $e_1$ of the canonical basis and leaves the orthogonal of $\{m_c,e_1\}$ invariant, we have 
            \begin{equation}\label{eq:motiv_annulus}
            C_1 \left| \mathbf{i} \left(R_c \theta\right)_1 + \left\| \left(R_c \theta \right)' \right\|^2 \right| \leq |1-k_c(\theta)| \leq C_2 \left| \mathbf{i} \left(R_c \theta\right)_1 + \left\| \left(R_c \theta \right)' \right\|^2 \right|
            \end{equation}
            where the prime symbol denotes the orthogonal projection onto the orthogonal of $e_1$, so that $\theta = (\theta_1, \theta')$.
        
        In light of \eqref{eq:motiv_annulus}, \cite[(2.38)]{MR0978023} defines the concentric annuli
        \begin{equation}
            \mathscr{C}_k := \left\{ y \in \R^d~|~ 4^{-k-1} \leq \left| \mathbf{i} y_1 + \|y'\|^2 \right| < 4^{-k} \right\}
        \end{equation}
        and 
        \begin{equation}
            \mathscr{C}^c_k := R_c^{-1} \mathscr{C}_k.
        \end{equation}

        \begin{figure}[htbp]
            \centering
            \includegraphics[scale=0.8]{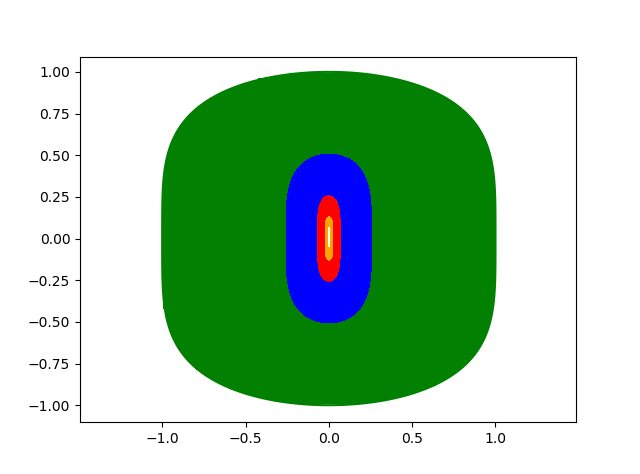}
            \caption{The annuli $\mathscr{C}_k$ for $0 \leq k \leq 3$ when $d = 2$.}
        \end{figure}
        
        We have $\mathscr{C}_{k+1} = D \mathscr{C}_k$ where \begin{equation}
            D = \diag\left( \frac{1}{4}, \frac{1}{2}, \ldots, \frac{1}{2}\right)
        \end{equation}
        and $\mathscr{C}^c_{k+1} = D_c \mathscr{C}^c_k$ where
        \begin{equation}
            D_c = R_c^{-1} D R_c.
        \end{equation}

        Splitting the integral \[\int_{\R^d} \frac{f(\theta)}{1 - k_c(\theta)} e^{- \mathbf{i}x \cdot \theta} \left( p_c(\theta) - {p_c(0)} \right) \mathrm{d}\theta\] as a sum of integrals onto the annuli $\mathscr{C}_k$ comes down to writing $\mathds 1_{\bigcup_{k \geq 0} \mathscr C_k} (\theta) = \sum_{k \geq 0} \mathds{1}_{\mathscr{C}_k}(\theta) = \sum_{k \geq 0} \mathds 1_{\mathscr C_0}\left( D^{-k} \theta\right)$. However, the indicator function $\mathds 1_{\mathscr C_0}$ is not smooth, so \cite{MR0978023} introduced a regularized version of $\mathds 1_{\mathscr C_0}$ like in the following lemma. Details of the construction of such a partition of unity can be found in Proposition A.5 of Guibourg's PhD thesis \cite{these_guibourg}.
        \begin{lemma}
            There exists a smooth, non-negative, compactly supported in a neighborhood of $\mathscr{C}_0$ that avoids $0$, function $\rho$ such that:
            \[\forall \theta \in \bigcup_{k \geq 0} \mathscr C_k,~~~~\sum_{k \geq 0} \rho \left( D^{-k} \theta \right) = 1\]
            and
            \begin{equation}\label{eq:couronne_lisse}
               \forall k \geq 1,~~~~ \forall y \in D^k (\supp \rho),~~~~ \frac{1}{2} 4^{-k-1} \leq \left| \mathbf{i} y_1 + \|y'\|^2 \right| \leq 2 \cdot 4^{-k}.
            \end{equation}
        \end{lemma}
        Inequalities \eqref{eq:couronne_lisse} mean that, up to constants, $\supp \rho$ allows the same control on $\left| \mathbf{i} y_1 + \|y'\|^2 \right|$ as the annulus $\mathscr C_k$.

        As a consequence, if we define 
        \begin{equation}\label{eq:def_rho_c}
        \rho_c := \rho \circ R_c,
        \end{equation}
        we have 
        \begin{equation}\label{eq:partition_unite}
            \forall \theta \in \bigcup_{k \geq 0} \mathscr{C}^c_k,~~~~ \sum_{k \geq 0} \rho_c\left( D_c^{-k} \theta \right) = 1
        \end{equation}
        and, since $D_c^k (\supp \rho_c) = R_c^{-1} D^k (\supp \rho)$,
        \begin{equation}\label{eq:couronnes_regularisees_controle}
            \forall \theta \in D_c^k (\supp \rho_c),~~~~ \frac{1}{2} 4^{-k-1} \leq \left| \mathbf{i} \left(R_c \theta\right)_1 + \left\| \left(R_c \theta \right)' \right\|^2 \right| \leq 2 \cdot 4^{-k}.
        \end{equation}
        Note that, since $c \mapsto R_c$ is continuous, $\rho_c$ and its derivatives are continuous functions of the variables $(\theta,c)$, i.e., for every $k \geq 0$, $(\theta,c) \mapsto \rho_c^{(k)}(\theta)$ is continuous.
        
        We stick to Babillot's notation and define 
        \begin{equation}\label{eq:definition_h}
            h_c(\theta) = \frac{f(\theta)}{1 - k_c(\theta)}  \left( p_c(\theta) - {p_c(0)} \right)
        \end{equation}
        so that the asymptotics \eqref{eq:result_dyadic_integrals} we want to prove becomes $\widehat{h}_c(x) = o\left( \|x\|^{- \frac{d-1}{2}} \right)$, uniformly in $c$. The splitting of the integral of \eqref{eq:result_dyadic_integrals} using the smooth partition of unity \eqref{eq:partition_unite} can then be written the following way.
        
        \begin{prop}\label{prop:hat{h}_into_sum}
            The set $V := \left(\bigcap_{c\in \partial \mathcal{C}}\bigcup_{k \geq 0} \mathscr{C}_k^c \right) \cup \{0\}$ is a neighborhood of $0$ in $\R^d$.
            If the support of $f$ is included in $V$, then for every $c \in \partial \mathcal{C}$ and $x \in \Z^d$,
            \begin{equation}\label{eq:decoupage_dyadique}
                 \widehat{h}_c(x) = \sum_{k \geq 0} \frac{1}{2^{k(d+1)}} \widehat{\rho}_{c,k}\left( D_c^k x \right)
            \end{equation}
            where $\rho_{c,k}(\theta) = \rho_c(\theta) h_c\left( D_c^k \theta\right)$.
        \end{prop}

    In the previous proposition, the purpose of the set $V$ is to provide a neighborhood of $0$ \emph{independent of $c$} in which the equality \eqref{eq:decoupage_dyadique} holds.
        
        \begin{proof}
            Since $\mathscr{C}_k^c$ is the image of $\mathscr{C}_k$ by a rotation, if we note $r = \inf \left\{ \|\theta\|~|~\theta \in \mathscr{C}_0 \right\} > 0$, we have $B(0,r) \subset V$.

            Let $x \in \Z^d$ and $c \in \partial \mathcal{C}$. Using \eqref{eq:partition_unite} and the inclusion of the support of $h_c$ in $\bigcup_{k \geq 0} \mathscr{C}_k^c \cup \{0\}$, we get
                \[\widehat{h}_c(x) = \int_{\R^d} e^{-\mathbf{i}x\cdot \theta} h_c(\theta) \mathrm{d}\theta = \sum_{k \geq 0} \int_{\R^d} e^{-\mathbf{i}x\cdot \theta} h_c(\theta) \rho_c\left( D_c^{-k} \theta \right) \mathrm{d}\theta .\]
                With the change of variable $\lambda = D^k_c \theta$ and the symmetry of $D_c^k$,
                \begin{align*}
                \widehat{h}_c(x)&= \sum_{k \geq 0} \int_{\R^d} e^{-\mathbf{i}x\cdot D_c^k\lambda} h_c\left( D_c^k\lambda \right) \rho_c\left( \lambda \right) \left| \det \left( D_c^k \right) \right| \mathrm{d}\lambda\\
                &= \sum_{k \geq 0} \frac{1}{2^{k(d+1)}} \int_{\R^d} e^{-\mathbf{i} D_c^kx\cdot \lambda} h_c\left( D_c^k\lambda \right) \rho_c\left( \lambda \right)  \mathrm{d}\lambda \\
                &=\sum_{k \geq 0} \frac{1}{2^{k(d+1)}} \widehat{\rho}_{c,k}\left( D_c^k x \right). \qedhere
            \end{align*}
        \end{proof}
        From now on, we assume that $f$ was chosen with support included in $V$, so that we can use Proposition~\ref{prop:hat{h}_into_sum}.

        We remind the definition of finely Hölder continuous functions and their link with Fourier transform stated by Babillot in \cite{MR0978023}.
        A function $l $ is called finely Hölder continuous of order $\alpha$ if $f$ is Hölder continuous of order $\alpha$ and if $\displaystyle\frac{\|l(y) - l(x)\|}{\|y - x\|^\alpha} \xrightarrow[y \to x]{}0$ for every $x$. When $\alpha = 0$, it simply means that $l$ is continuous.

        Let $K$ be a compact subset of $\R^d$ and $r > 0$. The space $C^r(K)$ is defined as the set of $\lfloor r \rfloor$ times differentiable functions supported in $K$ whose $\lfloor r \rfloor^{\mathrm{th}}$ differential is finely Hölder continuous of order $r - \lfloor r \rfloor$.  When $r$ is an integer, $C^r(K)$ is simply the space of $r$ times continuously differentiable functions supported in $K$.
        
        The space $C^r(K)$ is normed with 
        \begin{equation}
        \|f\|_r = \left\{
        \begin{array}{ll}
            \displaystyle \sum_{n=0}^{\lfloor r \rfloor} \left\| f^{(n)} \right\|_\infty + \sup_{\|x-y\| \leq 1} \frac{\left\| f^{\left( \lfloor r \rfloor \right)} (y) - f^{\left( \lfloor r \rfloor \right)}(x)\right\|}{\|y - x \|^{r - \lfloor r \rfloor}} & \text{ if $r$ is not an integer,}  \\
            \displaystyle \sum_{n=0}^{r} \left\| f^{(n)} \right\|_\infty & \text{ if $r$ is an integer.}
        \end{array}\right.
        \end{equation}

    The two following lemmas on the spaces $C^r(K)$ will be useful in the proof of Proposition \ref{prop:majoration_norme_rho_ck}. They are implicitly used by Babillot in \cite[(2.42)]{MR0978023}.
        
        \begin{lemma}\label{lem:inclusions_Cr}
            Let $K$ be a compact subset of $\R^d$. If $r_1 \leq r_2$, then $C^{r_2}(K) \subset C^{r_1}(K)$, and if $f \in  C^{r_2}(K)$, then $\|f\|_{r_1} \leq \|f\|_{r_2}$.
        \end{lemma}

        \begin{proof}
            We write $r_1 = m_1 + \alpha_1$ and $r_2 = m_2 + \alpha_2$ where $m_1, m_2$ are integers and $0 \leq \alpha_1, \alpha_2 < 1$. When $\alpha_1 =0$, i.e., $r_1$ is an integer, the result of Lemma \ref{lem:inclusions_Cr} is clear. We now assume $\alpha_1 >0$. Let $f \in C^{r_2}(K)$. In particular, $f$ is $m_1$ times differentiable.
            
            If $m_1 = m_2$, then $\alpha_1 \leq \alpha_2$. Therefore, for every $x,y \in \R^d$ such that $\|x-y\|\leq 1$, we have  
            \begin{equation}\label{eq:inclusions_espaces_Cr}
            \frac{\left\| f^{(m_1)} (x) - f^{(m_1)}(y) \right\|}{\|x-y\|^{\alpha_1}} = \frac{\left\| f^{(m_1)} (x) - f^{(m_1)}(y) \right\|}{\|x-y\|^{\alpha_2}} \|x-y\|^{\alpha_2 - \alpha_1} \leq \frac{\left\| f^{(m_1)} (x) - f^{(m_1)}(y) \right\|}{\|x-y\|^{\alpha_2}},
            \end{equation}
            but $ \displaystyle \frac{\left\| f^{(m_1)} (x) - f^{(m_1)}(y) \right\|}{\|x-y\|^{\alpha_2}} \xrightarrow[y \to x]{}0$ because $f^{(m_1)}$ is $\alpha_2$-finely Hölder continuous, so \[{\displaystyle\frac{\left\| f^{(m_1)} (x) - f^{(m_1)}(y) \right\|}{\|x-y\|^{\alpha_1}} \xrightarrow[y \to x]{} 0},\] hence $f \in C^{r_1}(K)$. Besides, \eqref{eq:inclusions_espaces_Cr} yields \[\sup_{\|x-y\| \leq 1} \frac{\left\| f^{(m_1)} (x) - f^{(m_1)}(y) \right\|}{\|x-y\|^{\alpha_1}} \leq \sup_{\|x-y\| \leq 1} \frac{\left\| f^{(m_1)} (x) - f^{(m_1)}(y) \right\|}{\|x-y\|^{\alpha_2}},  \] so $\|f\|_{r_1} \leq \|f\|_{r_2}$.

            If $m_2 \geq m_1 + 1$, then $f$ is $m_1 + 1$~times continuously differentiable and  the mean value inequality leads to 
            \[\frac{\left\| f^{(m_1)} (x) - f^{(m_1)}(y) \right\|}{\|x-y\|^{\alpha_1}} = \frac{\left\| f^{(m_1)} (x) - f^{(m_1)}(y) \right\|}{\|x-y\|}\|x-y\|^{1 - \alpha_1} \leq \left\| f^{(m_1 + 1)} \right\|_\infty \|x-y\|^{1 - \alpha_1},\]
            hence $\displaystyle\frac{\left\| f^{(m_1)} (x) - f^{(m_1)}(y) \right\|}{\|x-y\|^{\alpha_1}} \xrightarrow[y \to x]{}0$, so $f \in C^{r_1}(K)$. Besides, $\displaystyle \sup_{\|x-y\| \leq 1} \frac{\left\| f^{(m_1)} (x) - f^{(m_1)}(y) \right\|}{\|x-y\|^{\alpha_1}} \leq \left\| f^{(m_1 + 1)} \right\|_\infty $, 
        so 
        \[\|f\|_{r_1} \leq \sum_{n=0}^{m_1} \left\| f^{(n)} \right\|_\infty + \left\| f^{(m_1 + 1)} \right\|_\infty \leq \|f\|_{r_2}. \qedhere \]
        \end{proof}

        \begin{lemma}\label{lem:espaces_C^r_produit}
            Let $K$ be a compact subset of $\R^d$ and $r$ an \emph{integer}. The set $C^r(K)$ is stable under multiplication and  for every $f,g \in C^r(K)$, \[\|fg\|_r \leq   \binom{r}{\left \lfloor \frac{r}{2} \right \rfloor} \|f\|_r \|g\|_r.\]
        \end{lemma} 

        \begin{proof}
            The general Leibniz rule states that for any integer $n \leq r$, 
            \begin{equation}\label{eq:leibniz}
                (fg)^{(n)} = \sum_{k=0}^n \binom{n}{k} f^{(k)} g^{(n-k)}.
            \end{equation}
            Therefore, if we denote $C_r := \dbinom{r}{\left \lfloor \frac{r}{2} \right \rfloor} = \max \left\{ \binom{n}{k}~|~0 \leq k,n \leq r \right\}$, we have
            \begin{equation}
                \left\| (fg)^{(n)} \right\|_\infty \leq C_r \sum_{k=0}^n \left\| f^{(k)} \right\|_\infty \left\| g^{(n-k)} \right\|_\infty.
            \end{equation}
            Therefore,
            \begin{align*}
            \|fg\|_r &\leq C_r \sum_{n = 0}^r \sum_{k=0}^n \left\| f^{(k)} \right\|_\infty \left\| g^{(n-k)} \right\|_\infty\\
            &\leq C_r \sum_{i,j = 0}^r \left\| f^{(i)} \right\|_\infty \left\| g^{(j)} \right\|_\infty\\
            &= C_r \|f\|_r \|g\|_r.\qedhere
            \end{align*}
        \end{proof}
        It is possible to prove a similar result for non integer $r$. We obtained the constant $3 \dbinom{\lfloor r \rfloor}{\left \lfloor \frac{r}{2} \right \rfloor}$ instead of $\dbinom{r}{\left \lfloor \frac{r}{2} \right \rfloor}$.

        It is well known that if a function $f$ is $k$ times continuously differentiable with integrable derivatives, then as $|x|$ goes to infinity, $\widehat f(x) = o\left( |x|^{-k} \right)$. The two following propositions that come from \cite[(2.21)]{MR0978023}, extend this decay to functions of $C^r(K)$ for non integer $r$.
        
    \begin{prop}\label{prop:lien_Cr_decroissance_fourier}
        Let $K$ be a compact set of $\R^d$, $r >0$ and $f \in C^r(K)$. Then $\|x\|^r \left| \widehat{f}(x) \right| \xrightarrow[\|x \| \to + \infty]{}0$.
    \end{prop}
        
        \begin{prop}\label{prop:lien_norme_r_transfo_fourier}
            Let $K$ be a compact set of $\R^d$ and $r >0$. Then there exists a constant $C$ such that for every $f \in C^r(K)$, $\sup_{x} \|x\|^r \left|\widehat{f}(x) \right| \leq C \|f\|_r$.
        \end{prop}

        These two propositions can be summarized by the following statement.
        \begin{cor}\label{cor:fourier_continu_espace_Cr_Dr}
            The Fourier transform is a continuous linear function between the space $C^r(K)$ and the space $D^r$ of continuous functions that are $\underset{\|x\| \to + \infty}{o}\left( \|x\|^{-r} \right)$, normed with \[\|g\|_{D^r} := \sup_{x \in \R^d} \|x\|^r \left|g(x) \right| .\]
        \end{cor}
        
        In light of \eqref{eq:decoupage_dyadique}, in order to prove \eqref{eq:result_dyadic_integrals}, we want to bound $\|x\|^{\frac{d-1}{2}} \left|\widehat{\rho}_{c,k}(x)\right|$. Because of Proposition \ref{prop:lien_norme_r_transfo_fourier}, it comes down to bounding $ \left\| \rho_{c,k} \right\|_{\frac{d-1}{2}}$. This is the goal of the following proposition, which is the analog of (2.42) of Babillot in \cite{MR0978023}.

        \begin{prop}\label{prop:majoration_norme_rho_ck}
            There exists a compact set $K$ of $\R^d \setminus \{0\}$ such that for every $c \in \partial \mathcal{C}$ and $k$, $\rho_{c,k} \in C^{\frac{d-1}{2}}(K)$, where $\rho_{c,k}$ was defined in Proposition \ref{prop:hat{h}_into_sum}. Besides, there exists $M >0$ such that for every $c \in \partial \mathcal{C}$ and $k$, 
            \begin{equation}\label{eq:norme_rho_ck}
                \left\| \rho_{c,k} \right\|_{\frac{d-1}{2}} \leq M 2^k.
            \end{equation}
        \end{prop}

        Although our proof is close to Babillot's in \cite{MR0978023}, the function $h_c$ we use, which was defined in \eqref{eq:definition_h}, is slightly simpler than hers. Indeed, in the identification of the dominating part of the integral, Babillot showed that she could replace the denominator $1- k_c(\theta)$ by its second order expansion around $0$, which comes down to say that the analog of what Babillot proved is that \[ \int_{\R^d} f(\theta) \left(\frac{1}{1 - k_c(\theta)} p_c(\theta) - \frac{1}{-\mathbf{i}m_c \cdot \theta + \frac{1}{2} \theta \cdot \sigma_c \theta} p_c(0) \right) e^{- \mathbf{i}x \cdot \theta}  \mathrm{d}\theta = \underset{\|x\| \to + \infty}{o} \left( \|x\|^{-\frac{d-1}{2}} \right).\] Thus her function $h_c$ would be $h_c(\theta) = f(\theta) \left(\frac{1}{1 - k_c(\theta)} p_c(\theta) - \frac{1}{-\mathbf{i}m_c \cdot \theta + \frac{1}{2} \theta \cdot \sigma_c \theta} p_c(0) \right)$. 
        In our case, the complex analysis factorization of Proposition \ref{prop:weierstrass_prep} that we will use to prove the main theorem, coming from Woess's technique in \cite{Wo-00}, allows us to keep $1-k_c(\theta)$ at the denominator, hence the more simple function $h_c(\theta) = \frac{f(\theta)}{1- k_c(\theta)} (p_c(\theta) - p_c(0))$. This more simple function $h_c$ enables us to have $2^k$ in \eqref{eq:norme_rho_ck}, while \cite[(2.42)]{MR0978023} had $4^{k\left(1 - \frac{\varepsilon}{2}\right)}$ for some $\varepsilon>0$.

        \begin{proof}

        Let $c \in \partial \mathcal{C}$ and $k \geq 0$.
            The support of $\rho_{c,k}$ is included in the support of $\rho_c$, which is $R_c^{-1} (\supp \rho)$. Denoting $r = \inf \left\{ \|x\|~|~x \in \supp \rho \right\}$ and $R = \sup \left\{ \|x\|~|~x \in \supp \rho \right\}$, we thus have \[\supp \rho_{c} \subset \overline{B(0,R)} \setminus B(0,r).\] Besides, $r > 0$ and $R < + \infty$ because $\supp \rho$ is a compact subset of $\R^d \setminus\{0\}$, so the set $K = \overline{B(0,R)} \setminus B(0,r)$ is a compact subset of $\R^d \setminus \{0\}$ that contains the supports of all the functions $\rho_{c,k}$. Since $K$ avoids $0$, $\rho_{c,k}$ is smooth, hence $\rho_{c,k} \in C^{\frac{d-1}{2}}(K)$. 

            To prove \eqref{eq:norme_rho_ck}, we will prove that there exists a constant $M$ such that for every $c \in \partial C$ and $k$, 
            \[ \left\| \rho_{c,k} \right\|_{\left \lceil \frac{d-1}{2} \right \rceil } \leq M 2^k\] and then use Lemma \ref{lem:inclusions_Cr} to get \[\left\| \rho_{c,k} \right\|_{ \frac{d-1}{2}} \leq \left\| \rho_{c,k} \right\|_{\left \lceil \frac{d-1}{2} \right \rceil } \leq M 2^k.\]
            Following Babillot, we define \[\varepsilon_c(\theta) = f(\theta) (p_c(\theta) - p_c(0)),~~~~~~a_{c,k}(\theta) = \varepsilon_c \left( D_c^k \theta \right) ~~~~ \text{and} ~~~~v_{c,k}(\theta) = 1 - k_c\left( D_c^k \theta \right),\]  so that      
             \begin{equation}
                \rho_{c,k}(\theta) = \rho_c(\theta) \frac{a_{c,k}(\theta)}{v_{c,k}(\theta)}.
            \end{equation}
        Then we would like to apply Lemma \ref{lem:espaces_C^r_produit} to obtain a constant $M_1$ such that 
        \begin{equation*}
        \left\| \rho_{c,k} \right\|_{\left \lceil\frac{d-1}{2}\right\rceil} \leq M_1 \left\| \rho_c \right\|_{\left \lceil\frac{d-1}{2}\right\rceil} \left\|a_{c,k} \right\|_{\left \lceil\frac{d-1}{2}\right\rceil} \left\| \frac{1}{v_{c,k}} \right\|_{\left \lceil\frac{d-1}{2}\right\rceil}.
        \end{equation*}
        Note that these quantities are not well defined, as $\frac{1}{v_{c,k}}$ is not bounded around $0$. However, an easy refinement of Lemma \ref{lem:espaces_C^r_produit} leads to the existence of a constant $M_1$ such that
        \begin{equation}
            \label{eq:rho_c_k_inegalite_produit}
        \left\| \rho_{c,k} \right\|_{\left \lceil \frac{d-1}{2}\right\rceil} \leq M_1 \left\| \rho_c \right\|_{\left \lceil\frac{d-1}{2}\right\rceil} \left\|a_{c,k} \right\|_{\left \lceil\frac{d-1}{2}\right\rceil,\supp \rho_c} \left\| \frac{1}{v_{c,k}} \right\|_{\left \lceil\frac{d-1}{2}\right\rceil,\supp \rho_c},
        \end{equation}
        where the index ``$\supp \rho_c$'' means that the suprema that appear in the norm are taken for the restriction onto $ \supp \rho_c $ of the function. But $\supp \rho_c$ avoids $0$, so $\left\| \frac{1}{v_{c,k}} \right\|_{\left \lceil \frac{d-1}{2} \right \rceil,\supp \rho_c}$ is well defined.
        
        As noticed after equation \eqref{eq:partition_unite}, the function $\rho_c$ and its derivatives are continuous in the variables $(\theta, c)$,  and the set $K \times \partial \mathcal{C}$ is compact, so the derivatives of $\rho_c$ of order at most $\left\lceil \frac{d-1}{2} \right \rceil$ are bounded uniformly in $c$, hence the existence of a constant $M_2$ independent of $c$ such that for every $c \in \partial \mathcal{C}$, \[\|\rho_c\|_{\left \lceil \frac{d-1}{2} \right\rceil}  \leq M_2.\] 
        To conclude, we claim that, up to constants and uniformly in $c$, $\left\|a_{c,k} \right\|_{\frac{d-1}{2}, \supp \rho_c} \leq \frac{1}{2^k}$ and ${\left\| \frac{1}{v_{c,k}} \right\|_{\frac{d-1}{2}, \supp \rho_c} \leq 4^k}$, hence \eqref{eq:rho_c_k_inegalite_produit} yields \eqref{eq:norme_rho_ck}. Before proving these bounds, we give a little intuition. Indeed, from the definition $a_{c,k}(\theta) = \varepsilon_c \left( D_c^k \theta \right)$, we could expect $a_{c,k}$ to be of order $\left\|D_c\right\|^k = \frac{1}{2^k}$. As for $\frac{1}{v_{c,k}}(\theta) = \frac{1}{1-k_c\left( D_c^k \theta\right)} = \frac{1}{k_c(0) - k_c\left( D_c^k \theta\right)}$, we could expect $\frac{1}{v_{c,k}}$ to be of order $ \left\|D_c^{-k}\right\| = 4^k$.
        
        We start with the proof of the bound for $a_{c,k}$. The mean value inequality yields that for every $\theta \in \supp \rho_c$,
        \[\|a_{c,k}(\theta)\| \leq f\left(D_c^k \theta\right) \sup_{\tilde \theta \in \left[ 0, D_c^k \theta \right] }\left\| p_c' \left( \tilde \theta \right) \right\|  \left\|D_c^k \theta \right\| \leq \sup_{\theta \in \supp f}\left\| p_c' (\theta) \right\| R \left\|D_c\right\|^k,\]
        where $R$ was defined at the beginning of the proof.
        Besides, since $R_c$ is a rotation, we have $\|D_c\| = \|R_c^{-1} D R_c\| = \|D\| = \frac{1}{2}$, so we get 
        \[\|a_{c,k}\|_{\infty, \supp \rho_c} \leq \sup_{\theta \in \supp f}\left\| p_c' (\theta) \right\| R \frac{1}{2^k}. \]
        To bound $\|a_{c,k}\|_{\left \lceil\frac{d-1}{2}\right\rceil, \supp \rho_c}$, we also need to bound the differentials of $a_{c,k}$. When $n \geq 1$, the $n^{\mathrm{th}}$ differential $a_{c,k}^{(n)}(\theta)$ is the $n$-linear function $(h_1, \ldots, h_n) \mapsto \varepsilon_c^{(n)}(D_c^k \theta) \left( D_c^k h_1, \ldots, D_c^k h_n \right)$, so \[\left\| a_{c,k}^{(n)}(\theta) \right\| \leq \left\| \varepsilon_c^{(n)}(D_c^k\theta) \right\| \left\| D_c \right\|^{kn} \leq \left\| \varepsilon_c^{(n)}(D_c^k\theta) \right\| \frac{1}{2^k}.\]  
        Hence, if we note $M_3 := \displaystyle\sup_{(\theta,c) \in \supp f \times \partial \mathcal{C} }\left\| p_c' (\theta) \right\| R + \displaystyle \sum_{n=1}^{\left \lceil \frac{d-1}{2} \right\rceil} \sup_{(\theta,c) \in \supp f \times \partial \mathcal{C}} \left\| \varepsilon_c^{(n)} (\theta)\right\|$, we get 
        \[\left\|a_{c,k} \right\|_{\left \lceil \frac{d-1}{2} \right\rceil, \supp \rho_c} \leq M_3 \frac{1}{2^k}.\] 
        The constant $M_3$ does not depend on $c$ and it is finite because $(\theta,c) \mapsto p_c'(\theta)$ and $(\theta,c) \mapsto \varepsilon_c^{(n)}(\theta)$ are continuous functions (which comes from the analyticity of $p_c$ and Lemma \ref{lem:continuity_derivatives_holom}) and $\supp f \times \partial \mathcal{C}$ is a compact set.

        For $v_{c,k}$, the proof is the same as Babillot's in \cite[(2.42)]{MR0978023}. We remind it and precise one point of the proof. Differentiating $n$ times the inverse of a function $f$, we see that we obtain 
        \begin{equation*}
        \left(\frac{1}{f}\right)^{(n)} = \frac{P_n\left(f, f', \ldots, f^{(n)} \right)}{f^{n+1}}
        \end{equation*}
        for some multivariate polynomial $P_n$ independent of $f$. More precisely, one can prove by induction on $n$ that $\left(\frac{1}{f}\right)^{(n)}$ is a linear combination of terms of the form 
        \begin{equation}\label{eq:derivee_inverse}
            \frac{(f')^j (f'')^{\alpha_2} \ldots \left(f^{(n)}\right)^{\alpha_n}}{f^i}
        \end{equation}
            where $2i-j \leq n+2$ and $\sum_{\ell = 2}^n \ell \alpha_\ell = n-j$.
        Similarly to $a_{c,k}$, each time we differentiate $v_{c,k}$, the norm is multiplied by $\|D_{c}^k\|$, i.e., up to a constant, 
        \begin{equation}\label{v_majoration_derivee}
           \left\|v_{c,k}^{(\ell)}\right\|_{\infty, \rho_c} \leq \|D_{c}\|^{k\ell} = \frac{1}{2^{k\ell}}. 
        \end{equation}
          We even have a tighter bound on the first differential $v_{c,k}'$, which is, up to a constant, 
         \begin{equation}\label{eq:v_majoration_premiere_derivee}
            \left\| v_{c,k}' \right\|_{\infty, \supp \rho_c} \leq \frac{1}{4^k}.
         \end{equation} Details on \eqref{eq:v_majoration_premiere_derivee} will be given at the end of the proof.
          Besides, using inequalities \eqref{eq:motiv_annulus} and \eqref{eq:couronnes_regularisees_controle}, we obtain:
        \begin{equation*}
        \forall \theta \in \supp \rho_c,~~~~ \left|v_{c,k}(\theta) \right| \geq  C_1 \left| \mathbf{i} \left(R_c D_c^k \theta\right)_1 + \left\| \left(R_c D_c^k \theta \right)' \right\|^2 \right| \geq \frac{C_1}{2} 4^{-k-1},\end{equation*}
        so, up to a constant,
        \begin{equation}\label{eq:minoration_v}
            \left\| \frac{1}{v_{c,k}} \right\|_{\infty, \supp \rho_c} \leq 4^k
        \end{equation}
        Using the inequalities \eqref{v_majoration_derivee}, \eqref{eq:v_majoration_premiere_derivee} and \eqref{eq:minoration_v} in equality \eqref{eq:derivee_inverse}, we obtain that, up to a constant,
        \[\left\| \left(\frac{1}{v_{c,k} } \right)^{(n)} \right\|_{\infty, \supp \rho_c} \leq 4^{k(i-j)} 2^{-k\sum_{\ell = 2} ^n \ell \alpha_\ell} = 2^{k (2i -2j - n + j)} = 2^{k(2i-j-n)}.\] But $2i-j \leq n+2$, so we get $\left\| \left(\frac{1}{v_{c,k} } \right)^{(n)} \right\|_{\infty, \supp \rho_c} \leq 2^{2k} = 4^k$, up to a constant. Summing for $n \leq \left \lceil \frac{d-1}{2} \right \rceil$, we obtain a constant $M_4$ such that for every $k$, ${\left\| \frac{1}{v_{c,k}} \right\|_{\frac{d-1}{2}, \supp \rho_c} \leq M_4 4^k}$, which concludes the proof.

        To precise Babillot's proof, we give more details about \eqref{eq:v_majoration_premiere_derivee}. We have $v_{c,k}'(\theta) = - D_c^k \nabla k_c\left( D_c^k \theta \right)$ and \[D_c^k \nabla k_c(0) = D_c^k m_c = R_c^{-1} D^k R_c m_c = \|m_c\|R_c^{-1} D^k e_1 = \frac{\|m_c\|}{4^k} R_c^{-1} e_1,\] so 
         \[\forall \theta \in \supp \rho_c,~~~~ \left\|\frac{\|m_c\|}{4^k}R_c^{-1}e_1 - v'_{c,k}(\theta) \right\|  = \left\|D_c^k \left( \nabla k_c(0) - \nabla k_c \left( D_c^k \theta \right) \right)\right\|. \]
         The mean value inequality leads to
        \begin{align*}
        \forall \theta \in \supp \rho_c,~~~~ \left\|\frac{\|m_c\|}{4^k}R_c^{-1}e_1 - v'_{c,k}(\theta) \right\|   &\leq \left\| D_c^k\right\| \left( \sup_{\|\theta\| \leq R, c \in \partial \mathcal{C}} \|k_c''(\theta)\| \right) \left\|  D_c^k \theta \right\|\\
        & \leq \left\| D_c^k\right\| \left(\sup_{\|\theta\| \leq R, c \in \partial \mathcal{C}} \|k_c''(\theta)\| \right) \left\| D_c^k\right\| R\\
        &= \dfrac{\left(\displaystyle\sup_{\|\theta\| \leq R, c \in \partial \mathcal{C}} \|k_c''(\theta)\|\right) R}{4^k}
        \end{align*}
        where we remind that $R = \displaystyle\sup_{\theta \in \supp \rho} \|\theta\|$.
        Hence 
        \[\left\| v_{c,k}'(\theta) \right\| \leq \dfrac{\displaystyle \sup_{c \in \partial \mathcal{C}}\|m_c\| + \left(\displaystyle\sup_{\| \theta\| \leq R, c \in \partial \mathcal{C}} \|k_c''(\theta)\| \right) R}{4^k}. \qedhere\]
        \end{proof}
        
    \begin{prop}
            We have \eqref{eq:result_dyadic_integrals}, that is:
    \begin{equation}
    \int_{\R^d} \frac{f(\theta)}{1 - k_c(\theta)} e^{- \mathbf{i}x \cdot \theta} \left( p_c(\theta) - {p_c(0)} \right) \mathrm{d}\theta = \underset{\|x\| \to + \infty}{o} \left( \|x\|^{-\frac{d-1}{2}} \right)
\end{equation}
with a little-$o$ uniform in $c \in \partial \mathcal{C}$.
        \end{prop}

        \begin{proof}
        As mentioned after the definition of $h_c$ in \eqref{eq:definition_h}, we must prove that $\widehat{h}_c(x) = o\left(\|x\|^{- \frac{d-1}{2}}\right)$, uniformly in $c \in \partial \mathcal{C}$. We chose a function $f$ whose support satisfies the assumption of Proposition~\ref{prop:hat{h}_into_sum}, thus 
        \begin{equation}\label{eq:ecriture_h_chapeau_somme}
        \widehat{h}_c(x) = \sum_{k \geq 0} \frac{1}{2^{k(d+1)}} \widehat{\rho}_{c,k}\left( D_c^k x \right).
        \end{equation}

        Propositions \ref{prop:lien_norme_r_transfo_fourier} and \ref{prop:majoration_norme_rho_ck} show that there exists a constant $A >0$ such that for every $k \in \N$, $c \in \partial \mathcal{C}$ and $x \in \Z^d$, \[ \left\| D_c^k x \right\|^{\frac{d-1}{2}} \left| \widehat{\rho}_{c,k}\left( D_c^k x \right) \right| \leq A 2^k,\]
        hence 
        \begin{align*}
             \frac{\|x\|^{\frac{d-1}{2}}}{2^{k(d+1)}} \left|  \widehat{\rho}_{c,k}\left( D_c^k x \right) \right| & \leq  \frac{A 2^k}{2^{k(d+1)}}  \frac{\|x\|^{\frac{d-1}{2}}}{\left\| D_c^k x \right\|^{\frac{d-1}{2}}} \\
            & \leq  \frac{A}{2^{kd}} \left\| D_c^{-1} \right\|^{k\frac{d-1}{2}} \\
            & \leq A \frac{4^{k\frac{d-1}{2} }}{2^{kd}} \\
            &= \frac{A}{2^{k}}.
        \end{align*}
        Let $\varepsilon >0$ and $N \in \N$ such that $\sum_{k \geq N}\frac{A}{2^k} \leq \varepsilon$.  Splitting the sum \eqref{eq:ecriture_h_chapeau_somme} between $k < N$ and $k \geq N$ leads to 
        \begin{equation}\label{eq:inegalite_apres_avoir_coupe}
        \|x\|^{\frac{d-1}{2}} \left|  \widehat{h}_c(x) \right| \leq \sum_{k = 0}^{N-1} \frac{1}{2^{k(d+1)}} \|x\|^{\frac{d-1}{2}} \left|\widehat{\rho}_{c,k}\left( D_c^k x \right)\right| + \varepsilon. 
        \end{equation}
        
        By Propositions~\ref{prop:lien_Cr_decroissance_fourier}~and~\ref{prop:majoration_norme_rho_ck}, if $c \in \partial \mathcal{C}$ is fixed, then for every $k \leq N-1$, 
        \begin{equation}\label{eq:cv_vers_0_dans_D_(d-1)/2}
        \|y\|^{\frac{d-1}{2}} \left|\widehat{\rho}_{c,k}\left( y \right)\right| \xrightarrow[\|y\| \to + \infty]{}0.
        \end{equation}
        Assume temporarily that we have shown that the convergences \eqref{eq:cv_vers_0_dans_D_(d-1)/2} are uniform in $c$, that is:
        \[ \forall \varepsilon >0,~~~~\exists A_\varepsilon>0,~~~~\forall y \in \R^d ~ \text{s.t.}~ \|y\| \geq A_\varepsilon,~~~~\forall c \in \partial \mathcal{C},~~~~ \forall k \leq N-1,~~~~\|y\|^{\frac{d-1}{2}} \left|\widehat{\rho}_{c,k}\left( y \right)\right| \leq \varepsilon.\]
        Then if $\|x\| \geq 4^{N-1} A_{\frac{\varepsilon}{N}}$, we have $\|D_c^k x\| \geq A_{\frac{\varepsilon}{N}}$ for every $k \leq N-1$ and every $c \in \partial \mathcal{C}$, so
        \begin{align*}
        \|x\|^{\frac{d-1}{2}} \left|\widehat{\rho}_{c,k}\left( D_c^k x \right)\right| &\leq \left\|D_c^{-k} \right\|^{\frac{d-1}{2}} \|D_c^k x\|^{\frac{d-1}{2}} \left|\widehat{\rho}_{c,k}\left( D_c^k x \right)\right|\\
        & \leq 4^{k \frac{d-1}{2}} \frac{\varepsilon}{N}\\
        &= 2^{k(d-1)}\frac{\varepsilon}{N},
        \end{align*}
        hence \[\sum_{k = 0}^{N-1} \frac{1}{2^{k(d+1)}} \|x\|^{\frac{d-1}{2}} \left|\widehat{\rho}_{c,k}\left( D_c^k x \right)\right| \leq \varepsilon\]
        and \eqref{eq:inegalite_apres_avoir_coupe} leads to \[\|x\|^{\frac{d-1}{2}} \left|  \widehat{h}_c(x) \right| \leq 2 \varepsilon,\] which ends the proof.

        We now prove that the convergences of \eqref{eq:cv_vers_0_dans_D_(d-1)/2} are uniform in $c$. Let $\varepsilon>0$. For $k \leq N-1$, we denote ${g_{c,k} : y \mapsto \|y\|^{\frac{d-1}{2}} \left| \widehat{\rho}_{c,k}(y) \right|}$. We first claim that the function $\begin{array}[t]{ccc}
            \partial \mathcal{C} & \to & L^\infty  \\
             c& \mapsto & g_{c,k}
        \end{array}$ is continuous. Indeed, the function $(c, \theta) \mapsto \rho_{c,k}(\theta)$ and all its derivatives with respect to $\theta$ are continuous in the variables $(c, \theta)$ (once again because of the analyticity of $p_c$ and $k_c$ and because of Lemma \ref{lem:continuity_derivatives_holom}) and compactly supported, so they are uniformly continuous in the variables $(c,\theta)$, hence the functions $\begin{array}[t]{ccc}
             \partial \mathcal{C}& \to & C^{\frac{d-1}{2}}(K)  \\
             c& \mapsto & \rho_{c,k}
        \end{array}$ are continuous. Therefore, by Corollary \ref{cor:fourier_continu_espace_Cr_Dr}, the function $\begin{array}[t]{ccc}
            \partial \mathcal{C} & \to & D^{\frac{d-1}{2}}  \\
             c& \mapsto & \widehat{\rho}_{c,k}
        \end{array}$ is continuous, which implies the continuity of $\begin{array}[t]{ccc}
            \partial \mathcal{C} & \to & L^\infty  \\
             c& \mapsto & g_{c,k}
        \end{array}$, hence its uniform continuity because $\partial \mathcal{C}$ is compact. Therefore, we have: \[\exists r>0,~~~~ \forall c \in \partial\mathcal{C},~~~~\forall \tilde{c} \in B(c, r), ~~~~ \forall x \in \R^d,~~~~| g_{\tilde c,k}(x) - g_{c,k}(x)| \leq \varepsilon.\]
        From the open cover of the compact set $\partial \mathcal{C}$ by the balls $B(c,r)$, we extract a finite subcover $\displaystyle \partial \mathcal{C} \subset \bigcup_{i=1}^n B(c_i,r)$. By \eqref{eq:cv_vers_0_dans_D_(d-1)/2}, we have $g_{c_i,k}(y) \xrightarrow[\|y\| \to + \infty]{}0$ for each $i$. Since we have a finite number of functions, it yields:
        \[\exists B_\varepsilon>0,~~~~ \forall y \in \R^d ~\text{s.t.}~ \|y\| \geq B_\varepsilon,~~~~\forall i \in \{1, \ldots,n \},~~~~\left| g_{c_i,k}(y) \right| \leq \varepsilon.\]
        Now we take $y$ such that $\|y\| \geq B_\varepsilon$ and any $c \in \partial \mathcal{C}$. Then, taking $i$ such that $c \in B(c_i, r)$, we have \[\left|g_{c,k}(y) \right| \leq \left|g_{c,k}(y) - g_{c_i,k}(y) \right| + \left|g_{c_i,k}(y)\right| \leq 2 \varepsilon.\]
        Since $B_\varepsilon$ does not depend on $c$, this proves that the convergences \eqref{eq:cv_vers_0_dans_D_(d-1)/2}
 are uniform in $c$.        
        \end{proof}

\section{Appendix: Proofs of the complex analysis ingredient}\label{sec:app:complex_analysis}

In this section, we prove the results involving complex analysis, that is, the results of Section \ref{sec:complex_analysis_ingredient_without_proofs} and the following lemma that was useful in Appendix \ref{sec:dominating_part} to get uniformity in $c$.

\begin{lemma}\label{lem:continuity_derivatives_holom}
            Let $K$ be a compact metric space, $D := D(0,r) := D_1 \times \ldots \times D_d$ the closed polydisc of center $0$ and radius $r> 0$ of $\mathbb C^d$ and ${(f_c)}_{c \in K}$ a family of analytic functions in a neighborhood of $D$. We assume that $(z,c) \mapsto f_c(z)$ is continuous. Then for every $\alpha \in \N^d$, \[(z,c) \mapsto \frac{\partial^{|\alpha|} f_c }{\partial z_1^{\alpha_1} \ldots \partial z_d^{\alpha_d} } (z_1, \ldots, z_d)\] is continuous on $ \frac{1}{2}D  \times K$ where $\frac{1}{2}D = D\left(0, \frac{r}{2}\right)$.
        \end{lemma}

        \begin{proof}
            By Cauchy's formula, for all $z = (z_1, \ldots, z_d) \in \frac{1}{2}D$ and $c \in \partial \mathcal{C}$,
            \begin{equation}
                \frac{\partial^{|\alpha|} f_c }{\partial z_1^{\alpha_1} \ldots \partial z_d^{\alpha_d} } (z_1, \ldots, z_d) = \frac{|\alpha|!}{(2 \mathbf{i}\pi)^d} \int_{\partial D_1 \times \ldots \times \partial D_d} \frac{f_c(\zeta_1, \ldots, \zeta_d)}{(\zeta_1 - z_1)^{\alpha_1 +1} \ldots (\zeta_d - z_d)^{\alpha_d + 1}} \mathrm{d} \zeta_1 \ldots \mathrm{d}\zeta_d.
            \end{equation}
            The integrand is continuous in the variables $(z,c)$. Besides, it is bounded by \[\dfrac{\displaystyle\max_{(\zeta_1, \ldots, \zeta_d,c) \in \partial D_1 \times \ldots \times \partial D_d \times K} \left| f_c(\zeta_1, \ldots, \zeta_d) \right| } {\left( \frac{r}{2}\right)^{|\alpha| + d}},\] which is finite by compactness of $\partial D_1 \times \ldots \times \partial D_d \times K$ and continuity. Therefore, the theorem of continuity of parameter-dependent integrals leads to the continuity of $\frac{\partial^{|\alpha|} f_c }{\partial z_1^{\alpha_1} \ldots \partial z_d^{\alpha_d} } (z_1, \ldots, z_d)$ with respect to variables $(z, c)$.
        \end{proof}

We now prove the results of Section \ref{sec:complex_analysis_ingredient_without_proofs}.

\begin{proof}[Proof of Lemma~\ref{lem:analytic_uniformly}]
    If $u \in \mathbb{S}^{d-1}$ is fixed, then $\mathcal{L}\mu_u(0) = \widehat{\mu}_{c(u)}(0)$ is the transition matrix of the Markovian part of the Doob transform, so Assumptions \ref{assump:main_assumptions}, which are still valid for the Doob transform, ensure that it is a non-negative, irreducible, aperiodic matrix. Therefore, the Perron-Frobenius theorem ensures that $1$ is the only leading eigenvalue of $\mathcal{L}\mu_u(0)$, and it is simple. The function $(z,u') \mapsto \mathcal{L}\mu_{u'}(z)$ is continuous and so is the spectrum, so there exist open neighborhoods $V_u$ of $0$ and $W_u$ of $u$ such that for all $(z,u') \in V_u \times W_u$, $\mathcal{L}\mu_{u'}(z)$ has a unique leading eigenvalue which is simple. From the open cover of the compact set $\mathbb{S}^{d-1}$ by the $W_u$ for $u \in \mathbb{S}^{d-1}$, we extract a finite cover. The finite intersection $V$ of the corresponding $V_u$ is an open neighborhood of $0$ such that for all $u' \in \mathbb{S}^{d-1}$, $\mathcal{L}\mu_{u'}(z)$ has a unique leading eigenvalue $\kappa_{u'}(z)$ which is simple. A direct application of the analytic implicit function theorem ensures that, for all $u \in \mathbb{S}^{d-1}$, $\kappa_u$ is analytic on $V$.
\end{proof}

\begin{proof}[Proof of Proposition~\ref{prop:weierstrass_prep}]
    The idea of the proof of \eqref{eq:weierstrass_decomp} is the same as in Woess's book \cite[(25.20)]{Wo-00}: using the proof of the Weierstrass preparation theorem from Hörmander's book \cite[Theorem 7.5.1]{hormander} to obtain the decomposition \eqref{eq:weierstrass_decomp} in a neighborhood of $0$ which is independent of $u$. 
        We have $1- \kappa_u(0) = 0$ and $\nabla \kappa_u(0) = \left\| m_{c(u)} \right\| e_1$. Indeed, without the rotation, the gradient would be $m_{c(u)}$ because of Proposition~\ref{prop:usual_prop_spectral_radius}, and $R_u$ sends $m_{c(u)}$ onto $\left\|m_{c(u)}\right\|e_1$. Therefore, $\frac{\partial \kappa_u}{\partial z_1}(0) = \left\| m_{c(u)} \right\| \neq 0$ because $m_{c(u)} \neq 0$ (Proposition~\ref{prop:prop_doob_transform}), so condition (7.5.1) of Hörmander is true for $k=1$, leading to the decomposition for a fixed $u$.
        
        To obtain the decomposition in a neighborhood of $0$ independent of $u$, we do like Woess and take a look at the proof of Hörmander. First of all, $\kappa_u$ is analytic in a neighborhood $V$ of $0$ independent of $u$ as a consequence of Lemma~\ref{lem:analytic_uniformly}. 
        Then we need 
        $r >0$ such that, for every in $u \in \mathbb{S}^{d-1}$, $\kappa_u$ is analytic at $(z_1,0')$ when $|z_1| < 2r$ and 
        \begin{equation}\label{eq:weierstrass_def_r}
            \text{$1 - \kappa_u(z_1,0) \neq 0$ when $0 < |z_1| <2r$.}
        \end{equation}
        The condition of analyticity is easily fulfilled as $V$ does not depend on $u$. As for \eqref{eq:weierstrass_def_r}, when there is only one analytic function, it comes from the isolated zeros theorem. In our case, with a family of functions indexed by $u$, the argument is the same as in Woess. Since $\Re \left(\frac{\partial \kappa_u}{\partial z_1}(0)\right) = \left\|m_{c(u)}\right\|  \neq 0$ and $(z,u) \mapsto \frac{\partial \kappa_u}{\partial z_1}(z)$ is continuous thanks to Lemma~\ref{lem:continuity_derivatives_holom}, we can find, for each $u \in \mathbb{S}^{d-1}$, open neighborhoods $U_u$ of $0$ and $W_u$ of $u$ such that for every $\left(z,\tilde{u}\right) \in U_u \times W_u$, $\Re \left(\frac{\partial \kappa_{\tilde u}}{\partial z_1}(z) \right) \neq 0$. From the open cover of $\mathbb{S}^{d-1}$ by the $W_u$, we can extract a finite cover. Taking the finite intersection of the corresponding $U_u$, we obtain a neighborhood $U$ of $0$ such that for every $(z,u) \in U \times \mathbb{S}^{d-1}$, $\Re \left(\frac{\partial \kappa_u}{\partial z_1}(z) \right) \neq 0$. We take $r>0$ such that the polydisc $D(0,2r)$ is included in $U$. Then \eqref{eq:weierstrass_def_r} is true, because otherwise, the complex Rolle's theorem of \cite[Theorem 2.1]{complex_rolle} would imply that $\Re \left(\frac{\partial \kappa_u}{\partial z_1}(z_1,0) \right) = 0$ for some $(z_1,0) \in D(0,2r)$.
        
        Finally, Hörmander needs 
        \begin{equation}\label{eq:weierstrass_def_delta}
        \text{$\delta>0$ such that for every $u \in \mathbb{S}^{d-1}$, $1-\kappa_u(z_1,z') \neq 0$ when $|z_1| = r$ and $\|z'\| < \delta$}.
        \end{equation}
         By \eqref{eq:weierstrass_def_r}, this is true when $z' = 0$. Thus, by continuity of $(z_1, z',u) \mapsto \kappa_u(z_1, z')$, if $u$ and $z_1$ are fixed with $|z_1| = r$, there exist open neighborhoods $U_{u, z_1}$, $V_{u, z_1}$ and $W_{u, z_1}$ of $z_1$, $0$ and $u$ such that for every $\left( \tilde z_1, z', \tilde u\right) \in U_{u, z_1} \times V_{u, z_1} \times W_{u, z_1}$, $1- \kappa_{\tilde u}\left( \tilde z_1, z' \right) \neq 0$. From the open cover of $S(0,r) \times \mathbb{S}^{d-1}$ by the $U_{u, z_1} \times W_{u, z_1}$, we extract a finite cover. The finite intersection of the corresponding $V_{u,z_1}$ is an open neighborhood $V'$ of $0$ such that for every $z_1$ such that $|z_1| = r$, $u \in \mathbb{S}^{d-1}$ and $z' \in  V'$, $1 - \kappa_{u}\left(z_1, z' \right) \neq 1$. To get \eqref{eq:weierstrass_def_delta}, we only need to choose $\delta>0$ such that the polydisc $D(0, \delta)$ of $\C^{d-1}$ is included in $V'$.
        
        Once we have these $r$ and $\delta$ that do not depend on $u$, Hörmander's proof works to obtain the decomposition \eqref{eq:weierstrass_decomp} in a neighborhood of $0$ which is independent of $u$.
        It is not explicitly stated in Hörmander's theorem that $\mathcal{B}_u$ does not vanish around $0$, but the proof shows that $z_1 - \mathcal{A}_u(z')$ captures all the zeros of $1- \kappa_u(z)$.

\vspace{0.3cm}

        We now prove item \ref{it:weierstrass_2}. Lemma~\ref{lem:continuity_derivatives_holom} shows that we only need to prove the continuity of $\left( z', u \right) \mapsto \mathcal{A}_u\left( z' \right)$. We fix $z' \in \mathbb C^{d-1}$ such that $\left| z' \right| < \delta$ (the same $\delta$ as in \eqref{eq:weierstrass_def_delta}), $u \in \mathbb{S}^{d-1}$, consider sequences $\left( z'_n \right)_{n \in \N} \in \left(\mathbb C^{d-1}\right)^\N$ and $(u_n)_{n \in \N} \in \left( \mathbb{S}^{d-1}\right)^{\N}$ that converge to $z'$ and $u$, and prove that $\mathcal{A}_{u_n}\left( z'_n \right) \xrightarrow[n \to + \infty]{} \mathcal{A}_u \left( z' \right)$. 
        Looking at the proof of Hörmander, we see that $\mathcal{A}_{u_n}\left(z_n'\right)$ is defined as the only zero of $1 - \kappa_{u_n} \left(\cdot, z_n'\right)$ whose modulus is smaller than $r$. Therefore, $\left(\mathcal{A}_{u_n}\left(z_n'\right)\right)_{n \in \N}$ is bounded by $r$. If $\ell$ is a subsequential limit of $\left(\mathcal{A}_{u_n}\left(z_n'\right)\right)_{n \in \N}$, then $|\ell| \leq r$, and taking the subsequential limit of $1 - \kappa_{u_n} \left( \mathcal{A}_{u_n}\left(z_n'\right), z'_n\right) = 0$, which is possible by continuity of $\left( z, u \right) \mapsto 1 - \kappa_u \left( z \right)$, yields \[1 - \kappa_{u} \left( \ell, z'\right) = 0.\] Therefore, by construction of $\delta$ in \eqref{eq:weierstrass_def_delta}, $|\ell| \neq r$, so $|\ell| < r$. Since $\mathcal{A}_u\left( z' \right)$ is the only zero of $1-\kappa_u\left( \cdot, z' \right)$ of modulus smaller than $r$, we conclude that $\ell = \mathcal{A}_u\left( z' \right)$. Thus $\left( \mathcal{A}_{u_n}\left(z_n'\right) \right)_{n \in \N}$ is a bounded sequence whose only subsequential limit is $\mathcal{A}_{u}\left(z'\right)$, so it converges to $\mathcal{A}_{u}\left(z'\right)$, which proves the continuity of $\left( z', u \right) \mapsto \mathcal{A}_u\left( z' \right)$.

        \vspace{0.3cm}

        Let us prove item \ref{it:weierstrass_3}. The equality $\mathcal{A}_u (0) = 0$ simply comes from \eqref{eq:weierstrass_decomp} at $z=0$.
        Differentiating the implicit equation $\kappa_u\left( \mathcal{A}_u (z'), z'  \right) = 1$ with respect to the variable $z_i$ ($i \in \{2, \ldots, d\}$) yields 
        \begin{equation}\label{eq:implic_eq_diff_1}
        \frac{\partial \mathcal{A}_u}{\partial z_i} (z') \partial_1 \kappa_u \left( \mathcal{A}_u (z'), z'  \right) + \partial_i \kappa_u \left( \mathcal{A}_u (z'), z'  \right) = 0.
        \end{equation}
        At $z' = 0$, we have $\mathcal{A}_u (0) = 0$ and $\nabla \kappa_u(0) = \left\| m_{c(u)} \right\| e_1$, so $\partial_1 \kappa_u(0) = \left\| m_{c(u)} \right\|$ and $\partial_i \kappa_u(0) = 0$, leading to $ \left\| m_{c(u)} \right\| \frac{\partial \mathcal{A}_u}{\partial z_i} (0) = 0 $, and since $m_{c(u)} \neq 0$, \[\frac{\partial \mathcal{A}_u}{\partial z_i} (0) = 0~~~~ \forall i \in \{2 , \ldots, d \}. \]

        Differentiating \eqref{eq:implic_eq_diff_1} with respect to the variable $z_j$ ($j \in \{2, \ldots, d\}$) at $0$ yields
        \begin{equation*}
                \frac{\partial^2 \mathcal{A}_u}{\partial z_j \partial z_i}(0) \partial_1  \kappa_u (0)
                + \frac{\partial \mathcal{A}_u}{\partial z_i}(0) \times \left( \frac{\partial \mathcal{A}_u}{\partial z_j}(0) \times \partial_1 \partial_1 \kappa_u\left(0 \right) + \partial_j \partial_1 \kappa_u \left(0 \right)\right)
            + \frac{\partial \mathcal{A}_u}{\partial z_j}(0) \times \partial_1 \partial_i \kappa_u \left( 0 \right) 
            + \partial_j \partial_i \kappa_u\left( 0\right) =0.
        \end{equation*}
        Thanks to $\nabla \mathcal{A}_u(0) = 0$ and $\nabla \kappa_u(0) = \left\| m_{c(u)} \right\| e_1$, it simplifies into 
        \[ \left\| m_{c(u)} \right\| \frac{\partial^2 \mathcal{A}_u}{\partial z_j \partial z_i}(0)  + \partial_j \partial_i \kappa_u(0) =0.  \] Since the Hessian of $\kappa_u$ at $z = 0$ is the rotated energy matrix $\sigma_u$ (it is an easy adaptation of \ref{it:doob_transf_1}~of Proposition~\ref{prop:prop_doob_transform}), we obtain the expected formula for the hessian of $\mathcal{A}_u$.

        We have $\nabla \kappa_u(0) = \left\| m_{c(u)} \right\| e_1$, so taking the first order expansion of \eqref{eq:weierstrass_decomp} around $0$ yields \[- \left\| m_{c(u)} \right\|z_1 = z_1 \mathcal{B}_u(0),\] hence $\mathcal{B}_u(0) = - \left\| m_{c(u)} \right\|$.

\vspace{0.3cm}

        We finally prove item \ref{it:weierstrass_4}. The expansion of $\mathcal{A}_u$ around $0$ is a simple Taylor expansion. The big-$O$ is uniform because it is controlled by the third derivatives of $\mathcal{A}_u$ at $0$, which are bounded uniformly in $u$ since those derivatives are continuous in $u$ and $u$ lies in the compact set $\mathbb{S}^{d-1}$.

        Finally, $\sigma_u$ is positive-definite, so $\sigma_u^{(1)}$, which is the matrix of the quadratic form $\sigma_u$ restricted to a hyperplane, stays positive-definite. Therefore, by continuity of $u \mapsto \sigma_u^{(1)}$ and compactness of $\mathbb{S}^{d-1}$, there exists a constant $\alpha>0$ such that for every $u \in \mathbb{S}^{d-1}$ and $z' \in B(0, \varepsilon)_{\R^{d-1}}$, $z' \cdot \sigma_u^{(1)} z' \geq \alpha \left\| z' \right\|^2$. Moreover, by continuity of $c \mapsto m_c$ and compactness of $\partial C$, $\frac{1}{\|m_{c(u)}\|}$ is bounded from below by a positive constant, so there exists a constant $\tilde\alpha >0$ such that for every $u \in \mathbb{S}^{d-1}$ and $z' \in B(0, \varepsilon)_{\R^{d-1}}$, $\frac{1}{4\|m_{c(u)}\|}z' \cdot \sigma_u^{(1)} z' \geq \tilde\alpha \left\| z' \right\|^2$. Since the big-$O$ is uniform, it is smaller than $\beta \left\|z' \right\|^3$ for some constant $\beta$ independent of $u$. Therefore, for $z' \in \R^{d-1}$ such that $\|z'\| \leq \frac{\tilde{\alpha}}{\beta}$, \[O\left( \left\|z'\right\|^3 \right) \leq \beta \left\|z'\right\|^3 \leq \tilde\alpha \left\|z'\right\|^2 \leq \frac{1}{4\|m_{c(u)}\|}z' \cdot \sigma_u^{(1)} z'.\qedhere\]
        \end{proof}

        \renewcommand{\abstractname}{Acknowledgments}
\begin{abstract}
The author would like to thank his PhD advisors, Cédric Boutillier and Kilian Raschel, for many helpful discussions, suggestions and corrections. The author is grateful to the anonymous referee for their comments.

This work was conducted within the France 2030 framework programme, Centre Henri Lebesgue ANR-11-LABX-0020-01.

The author is partially supported by the DIMERS project ANR-18-CE40-0033, the RAWABRANCH project ANR-23-CE40-0008, funded by the French National Research Agency, and by  the European Research Council (ERC) project COMBINEPIC under the European Union’s Horizon 2020 research and innovation programme under the Grant Agreement No. 759702.
\end{abstract}
\bibliography{biblio}
\bibliographystyle{plain}
\end{document}